%% file: paper_MultiPatch.tex
\documentclass{article}

\usepackage[hmargin={2.5cm,2.5cm}, vmargin={3cm,3cm}, dvips]{geometry}
\input{header}

\graphicspath{{../}{../m3as-review/}}

\begin{document}


\title{\textbf{Adaptive isogeometric methods with $C^1$ (truncated) hierarchical splines on planar multi-patch domains}}

\author{{Cesare Bracco$^{1}$, Carlotta Giannelli$^{1}$, Mario Kapl$^{2}$, R. V\'azquez$^{3,4}$} \\
\footnotesize{$^1$ Dipartimento di Matematica e Informatica ``U. Dini'',
Universit\`{a} degli Studi di Firenze, Florence, Italy} \\
\footnotesize{$^2$ Department of Engineering $\&$ IT, Carinthia University of Applied Sciences, Villach, Austria} \\
\footnotesize{$^3$ Institute of Mathematics, \'{E}cole Polytechnique F\'{e}d\'{e}rale de Lausanne, Lausanne, Switzerland} \\
\footnotesize{$^4$ Istituto di Matematica Applicata e Tecnologie Informatiche `E.~Magenes' del CNR, Pavia, Italy}
}








\maketitle
\vspace{-0.8cm}
\noindent\rule{\linewidth}{0.4pt}

\begin{abstract}
Isogeometric analysis is a powerful paradigm which exploits the high smoothness of splines for the numerical solution of high order partial differential equations. However, the tensor-product structure of standard multivariate B-spline models is not well suited for the representation of complex geometries, and to maintain high continuity on general domains special constructions on multi-patch geometries must be used. In this paper we focus on adaptive isogeometric methods with hierarchical splines, and extend the construction of $C^1$ isogeometric spline spaces on multi-patch planar domains to the hierarchical setting. 
{\RVV We introduce a new condition for the definition of hierarchical splines, which replaces the hypothesis of local linear independence for the basis of each level by a weaker assumption.}
We also develop a refinement algorithm that guarantees that the assumption is fulfilled by $C^1$ splines on certain suitably graded hierarchical multi-patch mesh configurations, and prove that it has linear complexity. The performance of the adaptive method is tested by solving the Poisson and the biharmonic problems.
\end{abstract}

\textit{Keywords:} Isogeometric analysis; Adaptivity; Hierarchical splines; $C^1$ continuity; Multi-patch domains; Biharmonic problem

\textit{AMS Subject Classification:} {65D07; 65D17; 65N30; 65N50}

\section{Introduction}
\input{s1_intro}
\section{\RVV Hierarchical splines: a relaxed condition for linear independence} \label{sec:hierarchical}
\input{s2_hierarchical}

\section{$C^1$ splines on the {\MK planar} multi-patch setting} \label{sec:multipatch} 
\input{s3_multipatch_setting}

\subsection{The $C^1$ multi-patch isogeometric spline space} \label{sec:C1space}
\input{s4_C1_space}
\section{Theoretical results for the $C^1$ spline space on one level} \label{sec:theoretical_onelevel}
\input{s5_theoretical_one_level}
\section{$C^1$ multi-patch hierarchical spline space} \label{sec:hierarchical-C1}
\input{s6_hierarchical_C1}

\section{Numerical tests} \label{sec:numerical_tests}
\input{s7_numerical_tests}

\section{Conclusions} \label{sec:conclusions}
\input{s8_conclusions}

\appendix
\input{appendices.tex}

\section*{Acknowledgment}
M. Kapl has been partially supported by the Austrian Science Fund (FWF) through the project~P~33023-N. R. V\'azquez has been partially supported by the Swiss National Science Foundation via the project n.200021\_188589 and by the ERC AdG project CHANGE n.694515. 
C. Bracco and C. Giannelli acknowledge the contribution of the National Recovery and Resilience Plan, Mission 4 Component 2 - Investment 1.4 -  CN\_00000013 CENTRO NAZIONALE "HPC, BIG DATA E QUANTUM COMPUTING", spoke 6.
These supports are gratefully acknowledged. C. Bracco, C. Giannelli and R. V\'azquez are members of the INdAM research group GNCS. The INdAM support through GNCS and the project SUNRISE is gratefully acknowledged.

\bibliographystyle{plain}
\bibliography{author}

\end{document}

%% file: header.tex
\usepackage{psfrag}
\usepackage{rotating}
\usepackage{color}
\usepackage{amssymb}
\usepackage{amsmath}
\usepackage{amsthm}
\usepackage{verbatim}
\usepackage{subfigure}
\usepackage{tabularx} 
\usepackage{tikz}
\usepackage{pgfplots}

\usepackage{soul}
\usepackage{latexsym}
\usepackage{amsfonts}
\usepackage{algorithm}
\usepackage{algpseudocode}
\usepackage{stmaryrd}
\usepgfplotslibrary{colorbrewer,external}
\newcommand{\f}[1]{\mathbf{#1}}
\newcommand{\R}{\mathbb R}
\newcommand{\N}{\mathbb N}
\newcommand{\Du}{\partial_{1}}
\newcommand{\Dv}{\partial_{2}}
\newcommand{\Dxi}{\partial_{\xi}}
\DeclareMathOperator{\myspan}{span}
\DeclareMathOperator{\supp}{supp}

\newcommand{\indOmega}{\mathcal{I}_{\Omega}}
\newcommand{\indSigma}{\mathcal{I}_{\Sigma}}
\newcommand{\indSigmaI}{\mathcal{I}_{\Sigma}^{\circ}}
\newcommand{\indSigmaB}{\mathcal{I}_{\Sigma}^{\Gamma}}
\newcommand{\indChi}{\mathcal{I}_{\chi}}
\newcommand{\indChiI}{\mathcal{I}_{\chi}^{\circ}}
\newcommand{\indChiB}{\mathcal{I}_{\chi}^{\Gamma}}

\newcommand{\JOmega}{\mathbf{J}_{\Omega}}
\newcommand{\JSigma}{\mathbf{J}_{\Sigma}}
\newcommand{\JChi}{\mathbf{J}_{\chi}}

\newcommand{\US}[2]{\mathbb{S}_{#1}^{#2}}
\newcommand{\UXI}[2]{\Xi_{#1}^{#2}}
\newcommand{\UN}[3]{N_{#3,#1}^{#2}}
\newcommand{\UM}[3]{M_{#3,#1}^{#2}}
\newcommand{\UV}{\mathbb{V}}
\newcommand{\UW}{\mathbb{A}}
\newcommand{\PhiB}{\Phi}
\newcommand{\phivec}{\boldsymbol{\phi}}
\newcommand{\phiP}[2]{\phi_{#2}^{\Omega^{(#1)}}}
\newcommand{\phiI}[2]{\phi_{#2}^{\Sigma^{(#1)}}}
\newcommand{\phiV}[2]{\phi_{#2}^{\f{x}^{(#1)}}}
\newcommand{\fSigma}{f}

\newcommand{\coefc}{c}
\newcommand{\coefd}{d}
\newcommand{\coefe}{e}
\newcommand{\coefet}{\widetilde{e}}
\newcommand{\coeff}{{\bf K}}

\newcommand{\trunc}{\mathop{\mathrm{trunc}}}
\newcommand{\Trunc}{\mathop{\mathrm{Trunc}}}

\newcommand{\GG}{G} 
\newcommand{\QQ}{\mathcal{Q}} 

\DeclareMathOperator{\dist}{dist}

\newtheorem{theorem}{Theorem}[section]
\newtheorem{lemma}[theorem]{Lemma}
\newtheorem{proposition}[theorem]{Proposition}
\newtheorem{example}{Example}
\newtheorem{remark}{Remark}

\def\MK{\color{black}}
\def\CB{\color{black}}
\def\RV{\color{black}}
\def\RVV{\color{black}}
\definecolor{pinegreen}{rgb}{0.0, 0.47, 0.44}
\definecolor{darkgreen}{rgb}{0.12, 0.35, 0.09}

%% file: s1_intro.tex

Isogeometric Analysis (IgA) is a numerical method for the solution of partial differential equations (PDEs), introduced with the idea of bridging the gap between computer aided design and finite element analysis. The fundamental idea is the use of (rational) spline functions both for the representation of the geometry and for the discretization of the PDEs, allowing a simpler interaction between them. It was very soon realized that one of the main advantages of IgA is the high continuity of splines, which is particularly useful in the solution of high order PDEs such as the stream formulation of linear elasticity \cite{ABBLRS-stream}, the Kirchhoff-Love shell \cite{kiendl-bletzinger-linhard-09}, or the Cahn-Hilliard phase field model \cite{gomez2008isogeometric}, because they allow a straightforward discretization of their direct formulations, that require the basis functions to be $C^1$, or more precisely, $H^2$-conforming.

The high continuity of the basis functions is trivially obtained in a single-patch representation of the domain, 
but the design of complex geometry models requires the use of a globally unstructured representation as the one given by multi-patch domains. While $C^0$ continuity across patches is relatively easy to obtain, as long as the meshes are conforming on the interfaces, the construction of isogeometric $C^1$ spline functions on complex multi-patch domains is more challenging, and has been extensively studied in recent years. The existing methods can be classified into two groups depending on whether the $C^1$ continuity of the spline spaces is exact or just approximate. In case of approximately $C^1$ multi-patch spline spaces, possible examples are methods which enforce the $C^1$-smoothness across the interfaces weakly, e.g. by adding penalty terms to the weak form of the PDE as in \cite{ApBrWuBl15,Guo2015881} or by means of Lagrange multipliers as in \cite{ApBrWuBl15,BenvenutiPhD}, and methods which approximate the $C^1$ continuity directly on the basis functions as
in \cite{SaJu21,TaTo22,WeTa21,WeTa22}. 
In case of exactly $C^1$ multi-patch spline spaces, the methods can be distinguished depending on the employed parameterization of the multi-patch domain. Examples are the use of multi-patch parameterizations which are $C^1$ everywhere and therefore possessing singularities at extraordinary vertices 
such as in \cite{NgPe16,ToSpHu17} or in subdivision based techniques \cite{Peters2,RiAuFe16,ZhSaCi18}, non-singular multi-patch geometries which are $C^1$ everywhere except in the neighborhood of extraordinary vertices where a $G^1$-cap is used \cite{Pe15-2,KaPe17,KaPe18,NgKaPe15} as well as non-singular multi-patch parameterizations which are in general just $G^1$ at all interfaces (or in case of planar domains even just $C^0$). In the latter case,
examples are approaches based on arbitrary topology meshes with specific piecewise polynomial patches \cite{BlMoVi17,BlMoXu20,mourrain2015geometrically}, methods employing generic spline patches \cite{ChAnRa18,ChAnRa19} or techniques using analysis-suitable (AS) $G^1$ multi-patch parameterizations \cite{CoSaTa16}. For more exhaustive explanations about the construction of globally $C^1$ spline spaces
we refer to the two recent surveys in \cite{HuSaTaTo21} and \cite{KaSaTa19b}.

In this paper we will consider {\MK planar} analysis-suitable (AS) $G^1$ multi-patch parameterizations, which contain the subset of (mapped) piecewise bilinear domains \cite{BeMa14,KaBuBeJu16,KaSaTa21} as special case. Their use, however, is not particularly restrictive because generic {\MK planar} multi-patch parameterizations can be usually reparameterized to be AS $G^1$ \cite{KaSaTa17b}. The importance of analysis-suitable $G^1$ geometries is that they allow the construction of $C^1$ multi-patch spline spaces with optimal polynomial reproduction properties, and therefore isogeometric methods based on those spaces have optimal convergence, as numerically shown in \cite{KaSaTa19b,KaSaTa19a}.
In particular, our focus will be on the addition of local refinement capabilities to the $C^1$ spline space from \cite{KaSaTa19a}, since (1) it has 
{\MK an}
explicitly given local basis, (2) its dimension is independent of the parameterizations of the single patches, and (3) the number of basis functions associated to each extraordinary vertex is equal to six independently of the vertex valence. 




The main advantage of adaptive isogeometric methods is that they provide the possibility to locally refine the approximation space of the considered PDE, which allows in general to attain the same accuracy as global refinement with an important reduction in the degrees of freedom and the computational effort. Among all the possible spline spaces with local refinement capabilities, {\RV here we focus} on (truncated) hierarchical splines \cite{giannelli2012,giannelli2016}, because their multilevel structure simplifies their use with respect to other spline spaces, see \cite[Chapter~4]{reviewadaptiveiga} for a discussion on the use of adaptive spline spaces in isogeometric analysis. The theoretical background of adaptive methods with hierarchical B-splines in IgA was investigated in \cite{buffa2016c,buffa2017b,gantner2017}, where the optimal convergence rates for second order elliptic PDEs was proved. The potential benefits of adaptivity with hierarchical splines for the solution of fourth order PDEs were studied to simulate brittle fracture in \cite{hennig2016,hennig2018}, for Kirchhoff plates and Kirchhoff-Love shells in \cite{AnBuCo20,CoAnVaBu20}, and for the simulation of tumor growth in \cite{lorenzo2017}, always restricted to the case of single-patch domains.

The use of $C^1$ multi-patch spaces with local refinement was initiated with works on T-splines, as in \cite{ScSiEvLiBoHuSe13,WeZhLiHu17} and references therein, and subdivision surfaces \cite{WeZhHuSc15}, but the approximation properties of the spaces were not optimal around extraordinary points. In the last years, the construction of $C^1$ splines {\MK with a T-spline refinement} based on degenerate patches (also called D-patches), {\MK see \cite{CaWeToLiHuKiZh20} and \cite{ToSpHu17},}
improved the accuracy around extraordinary vertices with respect to previous works.
{\MK Recently, the concept of D-patches was used in \cite{wei2021thusplines} to generate $C^1$ hierarchical splines}, although the properties required for the hierarchical construction were only studied numerically. The construction based on the $C^1$ space of \cite{KaSaTa19a} was extended to hierarchical splines in \cite{BrGiKaVa20} for two-patch geometries, that is, in the absence of extraordinary points.

In this work we generalize the construction of $C^1$ hierarchical splines from the two-patch to the multi-patch setting. Since the presence of extraordinary points removes the local linear independence property, usually considered so far for the construction of hierarchical spline spaces, our first contribution is the development of a {\RVV weaker assumption for linear independence, which only assumes linear independence of the functions of each level restricted to the unrefined region of its level.}
We prove that this relaxed assumption implies linear independence of the (truncated) hierarchical basis. Second, we present 
the characterization of the $C^1$ space on one level, as well as a counterexample of the local linear independence of its basis and key results about {\RVV local} linear independence of certain basis subsets.
Third, relying on these properties we construct the new hierarchical $C^1$ spline space, with an explicit expression of the refinement mask necessary to apply truncation. We show that under a simple condition on the hierarchical mesh near the extraordinary point {\RVV the new weaker assumption is}
satisfied, and we develop a simple refinement algorithm to ensure that this condition is always satisfied. In practice, whenever an element adjacent to an extraordinary vertex is marked for refinement, other elements of the same patch (at most four) must be also marked. Fourth, we combine the new algorithm with admissible refinement algorithms from \cite{BrGiVa18} and \cite{buffa2016c} to limit the interaction between functions of different levels, and prove its linear complexity {\RV with an estimate depending on the vertex valences, but not on the parameterizations}.  
Finally, {\RV we combine our refinement algorithm with an a priori error estimator for adaptive refinement, and} we show numerical evidence for the optimal convergence properties of the proposed adaptive scheme for different model problems.

The remainder of the paper is organized as follows. Section~\ref{sec:hierarchical} presents the new abstract framework for the definition of (truncated) hierarchical splines. 
In Section~\ref{sec:multipatch} we introduce all the necessary material regarding analysis-suitable $G^1$ multi-patch parameterizations and the considered $C^1$ spline space, with  additional information detailed in~\ref{sec:appendix}. Section~\ref{sec:theoretical_onelevel} investigates the properties of the one level $C^1$ multi-patch space, while Section~\ref{sec:hierarchical-C1} introduces the construction of the $C^1$ hierarchical spline space. 
We further present in this section a refinement algorithm that ensures the property of linear independence and admissibility, and show the linear complexity of the algorithm. In Section~\ref{sec:numerical_tests}, we demonstrate the potential of our novel adaptive method for applications in IgA by solving the Poisson and the biharmonic problems over different AS-$G^1$ multi-patch geometries, and we give some conclusions in Section~\ref{sec:conclusions}.

%% file: s2_hierarchical.tex

This section introduces an abstract framework for the construction of {\RV a} hierarchical spline basis, which relaxes the assumption of \emph{local linear independence} for the underlying sequence of spline bases, as originally considered in \cite{giannelli2014} and also assumed in the construction of $C^1$ hierarchical functions on two-patch geometries \cite{BrGiKaVa20}.

\subsection{Hierarchical splines}
{\RVV We start with some general notation related to (local) linear independence. Let $\Omega \subset \mathbb{R}^n$ be a bounded domain}, and let $\Psi$ be a set of functions in $\Omega$. We say that $\Psi$ is linearly independent in $\widetilde{\Omega} \subseteq \Omega$ if the set of functions
\[
 \Psi\vert_{\widetilde{\Omega}} := \{\beta\vert_{\widetilde{\Omega}} \, : \, \beta \in \Psi, \, \beta \vert_{\widetilde{\Omega} } \neq 0   \} 
\]
is linearly independent. We further say that $\Psi$ is locally linearly independent if $\Psi\vert_{\widetilde{\Omega}}$ is linearly independent for any $\tilde \Omega \subseteq \Omega$.

{\RVV To define hierarchical splines, let} $\mathbb{U}^0\subset\mathbb{U}^1\subset \ldots \subset\mathbb{U}^{N-1}$ be a sequence of nested multivariate spline spaces defined on the closed {\MK and bounded} domain $D\subset \mathbb{R}^d$ and let ${\Omega}^0\supseteq{\Omega}^1{\RV \supseteq \ldots \supseteq}{\Omega}^{N-1}$ be a sequence of closed nested domains with $\Omega^0\subseteq D$. We assume that any space $\mathbb{U}^\ell$, for $\ell=0,\ldots,N-1$, is spanned by the basis $\Psi^\ell$. {\RVV We then define the {\RV spanning set of hierarchical splines} $\cal H$ as follows:
\begin{equation}\label{eq:hbasis}
{\cal H}:=\left\{\psi\in\Psi^\ell: {\supp}^0\psi\subseteq \Omega^\ell \wedge
{\supp}^0\psi\not\subseteq\Omega^{\ell+1}, 
\ell=0,\ldots,N-1
\right\},
\end{equation}
where $\supp^0 \psi := \supp\psi \cap \Omega^0$.

{\RV To be of practical use in local refinement, the support of the basis functions in $\Psi^\ell$ should be local, in the sense that increasing the level reduces the size of the support. Moreover, to} guarantee that the set ${\mathcal H}$ is a basis, in \cite{giannelli2014} it is required that the basis functions in $\Psi^\ell$ are locally linearly independent. In our framework, this is replaced by the following less restrictive condition:
\begin{itemize}
\item[(P1)] {\RVV for each level $\ell$, the set $\Psi^\ell|_{\Omega^\ell \setminus \Omega^{\ell+1}}$ is linearly independent.}
\end{itemize}
Obviously, local linear independence implies (P1), while the opposite is not true. The linear independence of the functions in $\mathcal{H}$ is proved in the following theorem.}

\begin{theorem} \label{thm:lin-indep-H}
{\RV Assuming} that property (P1) holds for the basis $\Psi^\ell$, for $\ell=0,\ldots,N-1$, the {\RV set of hierarchical splines ${\cal H}$} is linearly independent.
\end{theorem}
\begin{proof}
{\RVV To prove linear independence we have to prove that in the following sum
\[
\sum_{\psi\in {\cal H}} c_{\psi} \psi = 
\sum_{\ell = 0}^{N-1} \sum_{\psi\in \Psi^\ell \cap \mathcal{H}} c_{\psi} \psi = 0
\]
all the coefficients are zero. 
In virtue of the definition of the {\RV hierarchical spline set} \eqref{eq:hbasis}, {\RV the functions} of $\Psi^0 \cap \mathcal{H}$ are the only nonzero functions in the expression above acting in $\Omega^0\setminus\Omega^1$. Property (P1) guarantees that they are linearly independent in that region, and consequently the corresponding coefficients $c_\psi$ must be zero for any $\psi\in\Psi^0 \cap \mathcal{H}$.

For any level $\ell=1,\ldots,N-1$ the same argument is exploited recursively: excluding the functions already considered in the sums of previous levels, {\RV the functions} in $\Psi^\ell \cap \mathcal{H}$ are the only
nonzero functions which act on $\Omega^\ell \setminus\Omega^{\ell+1}$, and by property (P1) the coefficient $c_\psi$ associated to any $\psi\in\Psi^\ell \cap \mathcal{H}$ must be zero.
}
\end{proof}

\subsection{Truncated hierarchical splines}

Replacing the property of local linear independence with property (P1) still allows us to define the truncated hierarchical basis, as originally introduced for hierarchical B-splines in \cite{giannelli2012} and later analyzed in a more general setting in \cite{giannelli2014}. 
In fact, in view of the nested nature of the sequence of spline spaces $(\mathbb{U}^\ell)_{\ell=0,\ldots,N-1}$, we can exploit a two-scale relation between bases of consecutive hierarchical levels to express any spline $s$ in the spline space of level $\ell$ ($\myspan\Psi^\ell$) as a linear combination of basis functions in $\Psi^{\ell+1}$, and to define the truncation of $s$ at level $\ell+1$ as
\begin{equation*}
{\trunc}^{\ell+1} (s) := \sum_{\psi\in \Psi^{\ell+1},\, \supp^0 \psi\not\subseteq\Omega^{\ell+1}} c_{\psi}^{\ell+1}(s) \psi, \quad \text{ with } \; s = \sum_{\psi\in \Psi^{\ell+1}} c_{\psi}^{\ell+1}(s) \psi.
\end{equation*} 
The truncated hierarchical {\RV splines are} then given by
\begin{equation}\label{eq:thbasis}
{\cal T}:=\left\{{\Trunc}^{\ell+1}(\psi): \psi \in \Psi^\ell \cap {\cal H}, 
\ell=0,\ldots,N-1
\right\},
\end{equation}
where
\begin{equation}\label{eq:succtrunc}
{\Trunc}^{\ell+1}(\psi) := {\trunc}^{N-1}\left(\ldots\left(
{\trunc}^{\ell+1}\left(\psi\right)\right)\ldots\right)
\end{equation}
defines the successive truncation of the {\RV function} $\psi$ of level $\ell$, for $\ell=0,\ldots,N-2$, and ${\Trunc}^{N}(\psi)=\psi$ for any $\psi\in\Psi^{N-1}$. Any hierarchical basis function $\psi\in {\cal H}$ generates a truncated {\RV spline} $\tau\in {\cal T}$ according to \eqref{eq:succtrunc} and it is called the mother function of the child function $\tau$, and (as in the setting of \cite{giannelli2014}) for the two functions it holds that $\tau_{\vert \Omega^{\ell}\setminus\Omega^{\ell+1}} = \psi_{\vert \Omega^{\ell}\setminus\Omega^{\ell+1}},$ that is used in the proof of the following theorem.

\begin{theorem}\label{thm:lin-indep-T}
{\RV Assuming} that property (P1) {\RV holds} for the basis $\Psi^\ell$, for $\ell=0,\ldots,N-1$, the {\RV set of truncated hierarchical splines} ${\cal T}$ is linearly independent.
\end{theorem}
\begin{proof}
Since a truncated function of level $\ell$ coincides with its mother function in $\Omega^\ell \setminus \Omega^{\ell+1}$, the proof is completely analogous to the one of \cite[Proposition~9]{giannelli2014}, simply replacing the use of local linear independence by property (P1) as it was done in the proof of Theorem~\ref{thm:lin-indep-H}.
\end{proof}

In addition to the linear independence of the hierarchical {\RV splines} $\mathcal{H}$ and the truncated hierarchical {\RV splines} $\mathcal{T}$, the following properties also hold as in the $C^1$ construction for two-patch domains \cite[Proposition~1]{BrGiKaVa20} and the abstract setting in \cite[Proposition~9]{giannelli2014} for the truncated basis, see those references for details. First, the intermediate spline spaces {\RV obtained when considering the construction \eqref{eq:hbasis} up to intermediate levels,} are nested. Second, given a hierarchy of subdomains defined at each level as an enlargement of $(\Omega^\ell)_{\ell=0,\ldots,N-1}$, the original hierarchical spline space is a subspace of the hierarchical spline space defined over the considered enlarged subdomains. And third, the hierarchical spline space generated by the truncated basis in \eqref{eq:thbasis} coincides with the span of the hierarchical basis introduced in \eqref{eq:hbasis}. 

{\RV In our abstract framework, we have removed the non-negativity and partition of unity properties from \cite{giannelli2014}, because the $C^1$ basis functions that we will use in the following do not satisfy them. However, if the two properties are satisfied by the basis functions of each level, they also hold for the truncated basis.}


%% file: s3_multipatch_setting.tex
In this section we describe the geometric configuration and the $C^1$ spline spaces for one single level, following the notation in \cite{BrGiKaVa20,KaSaTa19a,KaSaTa19b} with small changes. We start with a description of the geometry, and then we introduce the $C^1$ spline space.

\subsection{The geometric {\MK planar} multi-patch setting} \label{sec:multipatch-geometry}
 Let us consider an open domain~$\Omega \subset \R^2$, 
which is built up of quadrilateral patches~$\Omega^{(i)}$, $i \in \indOmega$, with $\Omega^{(i)} \cap \Omega^{(j)}= \emptyset$ for $i \neq j$, 
inner edges~$\Sigma^{(i)}$, $i \in \indSigmaI$, and inner vertices~$\f{x}^{(i)}$, $i \in \indChiI$, and whose boundary~$\Gamma = \partial \Omega$ is the union of boundary edges~$\Sigma^{(i)}$, $i \in \indSigmaB$, and boundary vertices~$\f{x}^{(i)}$, $i \in \indChiB$, i.e. 
\[
 \Omega = \Bigg( \bigcup_{i \in \indOmega}\Omega^{(i)} \Bigg) \cup \Bigg( \bigcup_{i \in \indSigmaI} \Sigma^{(i)} \Bigg)
 \cup \Bigg( \bigcup_{i \in \indChiI} \f{x}^{(i)} \Bigg), \quad \Gamma = \Bigg( \bigcup_{i \in \indSigmaB} \Sigma^{(i)} \Bigg) \cup \Bigg( \bigcup_{i \in \indChiB} \f{x}^{(i)} \Bigg).
\]

We assume that no hanging nodes exist, and denote by $\indSigma$ and $\indChi$ the 
{\MK unions} of indices for inner and boundary edges and vertices, respectively, that is, 
$\indSigma=\indSigmaI {\MK \cup} \indSigmaB$ and 
$\indChi=\indChiI {\MK \cup } \indChiB$.

\subsubsection{Univariate spaces and basis functions} \label{sec:univariate_functions}
Let $\US{p}{r}$ be the univariate spline space of degree~$p \geq 3$ and regularity~$1 \leq r \leq p-2$ in $[0,1]$ with respect to the uniform open knot vector
\begin{equation*}
\UXI{p}{r} = (\underbrace{0,\ldots,0}_{(p+1)-\mbox{\scriptsize times}},
\underbrace{\textstyle \frac{1}{k+1}\ldots ,\frac{1}{k+1}}_{(p-r) - \mbox{\scriptsize times}}, 
\underbrace{\textstyle \frac{2}{k+1},\ldots ,\frac{2}{k+1}}_{(p-r) - \mbox{\scriptsize times}},\ldots, 
\underbrace{\textstyle \frac{k}{k+1},\ldots ,\frac{k}{k+1}}_{(p-r) - \mbox{\scriptsize times}},
\underbrace{1,\ldots,1}_{(p+1)-\mbox{\scriptsize times}}),
\end{equation*}
with $k \in \N_{0}$ and 
$k \geq \max(0,\frac{5-p}{p-r-1})$, 
and let $\UN{p}{r}{j}$, $j \in \{ 0,\ldots,n-1 \}$ with $n=p+1+k(p-r)$, be the associated B-splines. We will also need the 
subspaces~$\US{p}{r+1}$ and $\US{p-1}{r}$ defined from the same internal breakpoints~$\frac{j}{k+1}$, $j=0,1,\ldots,k$, and will use for their B-splines the 
analogous notation $\UN{p}{r+1}{j}$, $j \in \{ 0,\ldots,n_{0}-1 \}$, and $\UN{p-1}{r}{j}$, $j \in \{ 0,\ldots,n_{1}-1 \}$, respectively, 
where $n_{0}=p+1+k(p-r-1)$ and $n_1=p+k(p-r-1)$. 

Moreover, we will use the modified basis functions~$\UM{p}{r}{j}$, for $j=0,1$, $\UM{p}{r+1}{j}$, for $j=0,1,2$, and 
$\UM{p-1}{r}{j}$, for $j=0,1$, 
which fulfill
\begin{equation*}
\begin{array}{c}
\Dxi^i \UM{p}{r}{j}(0) = {\RV \delta_{ij}}, \quad \Dxi^i \UM{p-1}{r}{j}(0) = {\RV \delta_{ij}}, \text{ for } i,j=0,1,  \\
\Dxi^i \UM{p}{r+1}{j}(0) = {\RV \delta_{ij}}, \text{ for } i,j=0,1,2,
\end{array}
\end{equation*}
where ${\RV \delta_{ij}}$ is the Kronecker delta, see Appendix~\ref{sec:modified_basis} for their definitions.

\subsubsection{Parameterizations in standard configuration} \label{subsubsec:standard_paramaterization}

Each quadrilateral patch~$\Omega^{(i)}$, $i \in \indOmega$, is given as the open image of a bijective and regular geometry mapping
\[
\f{F}^{(i)}: [0,1]^2 \rightarrow \overline{\Omega^{(i)}},
\]
with $\f{F}^{(i)} \in (\US{p}{r} \otimes \US{p}{r})^2$. The resulting multi-patch parameterization (also called multi-patch geometry) of $\Omega$, which consists 
of the single spline parameterizations~$\f{F}^{(i)}$, $i \in \indOmega$, will be denoted by $\f{F}$. 

Considering a particular edge~$\Sigma^{(i)}$, $i \in \indSigma$, or vertex~$\f{x}^{(i)}$, $i \in \indChi$, we will assume throughout the paper that the geometry 
mappings~$\f{F}^{(i_m)}$ of the corresponding patches~$\Omega^{(i_m)}$ in the vicinity of the edge or vertex are given in standard 
form~(see \cite{KaSaTa19a,KaSaTa19b}). {In case of an edge~$\Sigma^{(i)}$, $i \in \indSigma$, we distinguish between an inner and a boundary edge.} 
For any inner edge~$\Sigma^{(i)}$, $i \in \indSigmaI$, we assume that the two patches $\Omega^{(i_0)}$ and 
$\Omega^{(i_1)}$, $i_0, i_1 \in \indOmega$, with $\Sigma^{(i)} \subset \overline{\Omega^{(i_0)}} \cup \overline{\Omega^{(i_1)}}$, are parameterized in such 
a way that \begin{equation*} 
 \f{F}^{(i_0)}(0,\xi) = \f{F}^{(i_1)}(\xi,0) , \mbox{ }\xi \in (0,1),
\end{equation*}
see Fig.~\ref{fig:standard_form_edge}~(left). Similarly, for any boundary edge~$\Sigma^{(i)}$, $i \in \indSigmaB$, the geometry mapping~$\f{F}^{(i_0)}$ of the 
associated patch~$\Omega^{(i_0)}$, $i_0 \in \indOmega$, with $\Sigma^{(i)} \subset \overline{\Omega^{(i_0)}}$, fulfills
\begin{equation*} 
\Sigma^{(i)} = \f{F}^{(i_0)}(0,(0,1)),
\end{equation*}
see Fig.~\ref{fig:standard_form_edge} (right).
\begin{figure}
\begin{center}
\begin{tabular}{cc}
\resizebox{0.52\textwidth}{!}{
 \begin{tikzpicture}
  \coordinate(A) at (0,0); \coordinate(B) at (0.5,3); \coordinate(C) at (3,-0.5); \coordinate(D) at (3.5,3.5);  \coordinate(E) at (-3,-0.7); 
  \coordinate(F) at (-3.3,3.6);
  \draw (A) .. controls (0,1.5) .. (B);
  \draw (A) .. controls (1.5,0.2) .. (C);
  \draw (B) .. controls (1.5,3) .. (D);
  \draw (C) .. controls (2.8,1.5) .. (D);
  \draw (A) .. controls (-1.5,0) .. (E);
  \draw (B) -- (F);
  \draw (E) .. controls (-2.8,2) .. (F);
  \node at (-1.5,1.5) {$\Omega^{(i_1)}$};
  \node at (1.8,1.5) { $\Omega^{(i_0)}$};
  \node at (-0.3,1.8) { $\Sigma^{(i)}$};
  \draw[->] (0.2,0.17) -- (0.2,1.2);
  \draw[->] (0.2,0.17) -- (1.4,0.3);
  \draw[->] (-0.2,0.17) -- (-0.2,1.2);
  \draw[->] (-0.2,0.17) -- (-1.4,0.15);
  \node at (-0.9,0.4) {\scriptsize $\xi_2$};
  \node at (-0.45,0.7) {\scriptsize $\xi_1$};
  \node at (0.9,0.5) {\scriptsize $\xi_1$};
  \node at (0.45,0.7) {\scriptsize $\xi_2$};
\end{tikzpicture}
 }
 & 
 \resizebox{0.35\textwidth}{!}{
 \begin{tikzpicture}
  \coordinate(A) at (0,0); \coordinate(B) at (-0.5,3); \coordinate(C) at (3,-0.5); \coordinate(D) at (3.5,3.0);
  \draw (A) .. controls (-1,1.5) .. (B);
  \draw (A) .. controls (1.5,0.2) .. (C);
  \draw (B) .. controls (1.5,2.6) .. (D);
  \draw (C) .. controls (2.8,1.5) .. (D);
  \node at (1.8,1.5) { $\Omega^{(i_0)}$};
  \node at (-0.2,1.8) { $\Sigma^{(i)}$};
  \draw[->] (0.1,0.2) -- (-0.45,1.1);
  \draw[->] (0.1,0.2) -- (1.2,0.3);
  \node at (0.8,0.5) {\scriptsize $\xi_1$};
  \node at (0.0,0.8) {\scriptsize $\xi_2$};
\end{tikzpicture}
 }
\end{tabular}
\end{center}
\caption{Representation in standard form with respect to an edge~$\Sigma^{(i)}$, $i \in \indSigma$. Left: Two neighboring patches~$\Omega^{(i_0)}$ and $\Omega^{(i_1)}$ with common inner edge~$\Sigma^{(i)}$, $i \in \indSigmaI$. Right: The patch~$\Omega^{(i_0)}$ with boundary edge~$\Sigma^{(i)}$, $i \in \indSigmaB$.}
 \label{fig:standard_form_edge}
\end{figure}
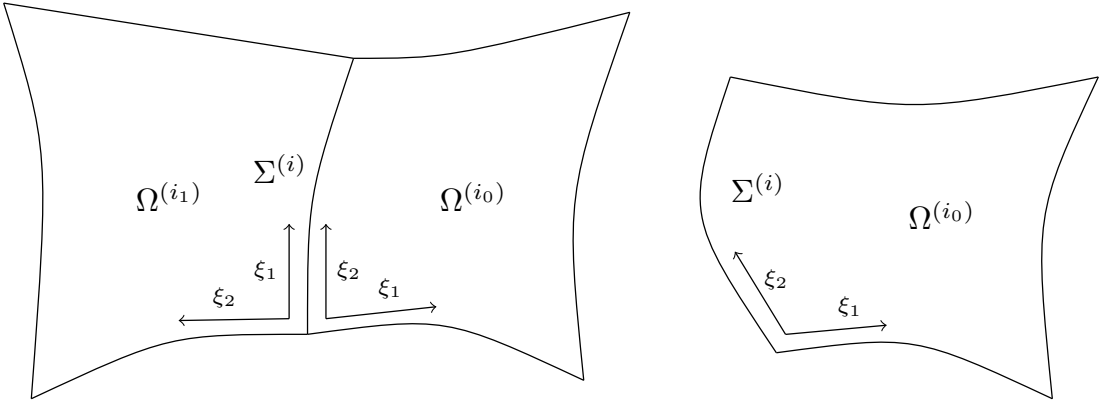

In case of a vertex~$\f{x}^{(i)}$, $i \in \indChi$, we distinguish between an inner and a boundary vertex. For any inner vertex~$\f{x}^{(i)}$, $i \in \indChiI$, with patch valence~$\nu_i \geq 3$, we respectively denote the patches and edges around the vertex~$\f{x}^{(i)}$ in 
counterclockwise order by $\Omega^{(i_m)}$
and $\Sigma^{(i_m)}$, 
for $m=0, \ldots, \nu_i-1$, see Fig.~\ref{fig:standard_form_vertex}~(left).
Moreover, we assume that 
the geometry mappings~$\f{F}^{(i_m)}$, $m=0,\ldots,\nu_i-1$, are parameterized in such a way that
\begin{equation*}
 \f{x}^{(i)}=\f{F}^{(i_m)}(0,0) \; \text{ for } m = 0, \ldots,\nu_i-1,
\end{equation*}
and
\begin{equation*}
\Sigma^{(i_{m+1})}=\f{F}^{(i_m)}(0,(0,1)) = \f{F}^{(i_{m+1})}((0,1),0) 
\; \text{ for } m =0, \ldots, \nu_i-1,
\end{equation*}
where the index~$m$ is considered to be modulo~$\nu_i$.
Similarly, for any boundary vertex~$\f{x}^{(i)}$, $i \in \indChiB$, with patch valence~$\nu_i \geq 1$, the patches and edges around the vertex~$\f{x}^{(i)}$ are labeled in counterclockwise order by $\Omega^{(i_m)}$, for $m=0, \ldots, \nu_i-1$, 
and $\Sigma^{(i_m)}$, 
for $m=0, \ldots, \nu_i$, see Fig.~\ref{fig:standard_form_vertex}~(right). Further, the geometry mappings~$\f{F}^{(i_m)}$, $m=0,\ldots,\nu_i-1$, are parameterized in such a way that
\begin{equation*}
 \f{x}^{(i)}=\f{F}^{(i_m)}(0,0) \; \text{ for } m = 0, \ldots,\nu_i-1,
\end{equation*}
the inner edges~$\Sigma^{(i_{m+1})}$ are given as
\begin{equation*}
\Sigma^{(i_{m+1})}=\f{F}^{(i_m)}(0,(0,1)) = \f{F}^{(i_{m+1})}((0,1),0) 
\; \text{ for } m =1, \ldots, \nu_i-1,
\end{equation*}
while the boundary edges ~$\Sigma^{(i_0)}$ and $\Sigma^{(i_{\nu_i})}$ are
\begin{equation*}
 \Sigma^{(i_{0})}=\f{F}^{(i_0)}((0,1),0) \mbox{ and } \Sigma^{(i_{\nu_i})}= \f{F}^{(i_{\nu_i-1})}(0,(0,1)).
\end{equation*}

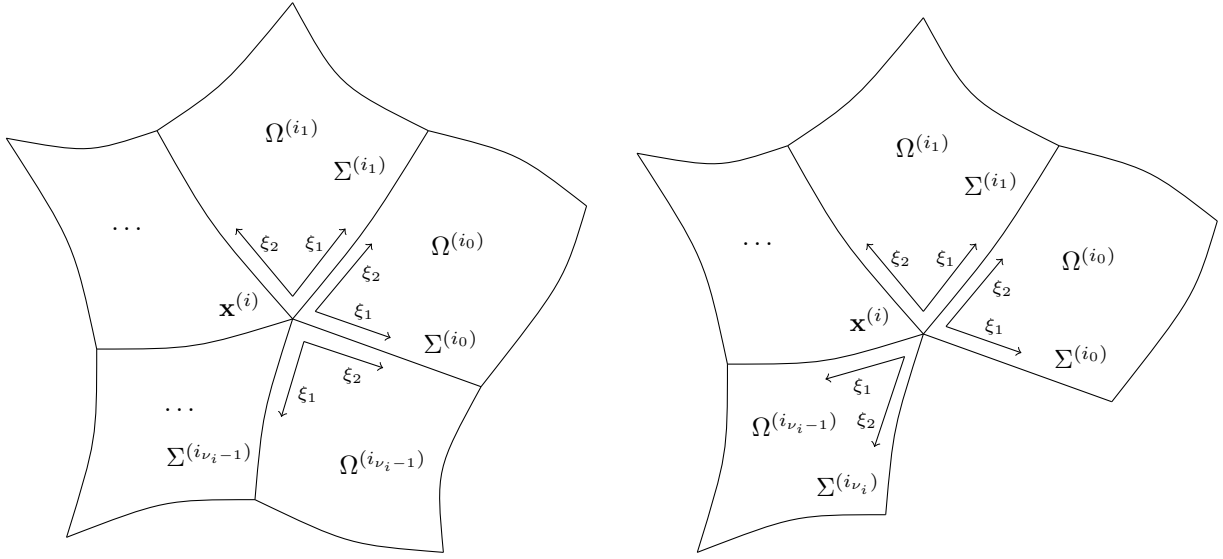
\begin{figure}
\begin{center}
\begin{tabular}{cc}
\resizebox{0.48\textwidth}{!}{
 \begin{tikzpicture}
  \coordinate(A) at (0,0); \coordinate(B) at (1.8,2.5); \coordinate(C) at (0,4.2); \coordinate(D) at (-1.8,2.5);  \coordinate(E) at (-3.8,2.4); 
  \coordinate(F) at (-2.6,-0.4); \coordinate(G) at (-3,-2.9); \coordinate(H) at (-0.5,-2.4); \coordinate(I) at (2,-3.1); 
  \coordinate(J) at (2.5,-0.9); \coordinate(K) at (3.9,1.5);
  \draw (A) .. controls (1,1.2) .. (B);
  \draw (A) .. controls (-1.1,1.25) .. (D);
  \draw (A) .. controls (-1.3,-0.4) .. (F);
  \draw (A) .. controls (-0.4,-1.25) .. (H);
  \draw (A) -- (J);
  \draw (B) .. controls (0.6,3.1) .. (C);
  \draw (C) .. controls (-0.9,3.1) .. (D);
  \draw (D) .. controls (-2.7,2.2) .. (E);
  \draw (E) .. controls (-2.9,1) .. (F);
  \draw (F) .. controls (-2.6,-1.6) .. (G);
  \draw (G) .. controls (-1.7,-2.4) .. (H);
  \draw (H) .. controls (0.75,-3) .. (I);
  \draw (I) .. controls (1.9,-2) .. (J);
  \draw (J) .. controls (3.6,0.5) .. (K);
  \draw (K) .. controls (3,2.2) .. (B);
  \node at (2.2,1) {$\Omega^{(i_0)}$};
  \node at (0,2.5) {$\Omega^{(i_1)}$};
  \node at (1.2,-1.9) {$\Omega^{(i_{\nu_i-1})}$};
  \node at (-2.2,1.2) {$\ldots$};
  \node at (-1.5,-1.2) {$\ldots$};
  \node at (-0.7,0.2) { $\f{x}^{(i)}$};
  \node at (2.1,-0.3) { $\Sigma^{(i_0)}$};
  \node at (0.9,2) { $\Sigma^{(i_1)}$};
  \node at (-1.1,-1.8) { $\Sigma^{(i_{\nu_i-1})}$};
  \draw[->] (0,0.3) -- (-0.75,1.2);
  \draw[->] (0,0.3) -- (0.7,1.2);
  \draw[->] (0.3,0.1) -- (1.05,1);
  \draw[->] (0.3,0.1) -- (1.3,-0.25);
  \draw[->] (0.15,-0.3) -- (1.2,-0.65);
  \draw[->] (0.15,-0.3) -- (-0.15,-1.3);
  \node at (-0.3,1.0) {\scriptsize $\xi_2$};
  \node at (0.3,1.0) {\scriptsize $\xi_1$};
  \node at (1.05,0.6) {\scriptsize $\xi_2$};
  \node at (0.95,0.1) {\scriptsize $\xi_1$};
  \node at (0.8,-0.75) {\scriptsize $\xi_2$};
  \node at (0.20,-1.0) {\scriptsize $\xi_1$};
\end{tikzpicture}
 }
 & 
 \resizebox{0.48\textwidth}{!}{
 \begin{tikzpicture}
  \coordinate(A) at (0,0); \coordinate(B) at (1.8,2.5); \coordinate(C) at (0,4.2); \coordinate(D) at (-1.8,2.5);  \coordinate(E) at (-3.8,2.4); 
  \coordinate(F) at (-2.6,-0.4); \coordinate(G) at (-3,-2.9); \coordinate(H) at (-0.5,-2.4); \coordinate(I) at (2,-3.1); 
  \coordinate(J) at (2.5,-0.9); \coordinate(K) at (3.9,1.5);
   \draw (A) .. controls (1,1.2) .. (B);
   \draw (A) .. controls (-1.1,1.25) .. (D);
   \draw (A) .. controls (-1.3,-0.4) .. (F);
  \draw (A) .. controls (-0.4,-1.25) .. (H);
  \draw (A) -- (J);
  \draw (B) .. controls (0.6,3.1) .. (C);
  \draw (C) .. controls (-0.9,3.1) .. (D);
  \draw (D) .. controls (-2.7,2.2) .. (E);
  \draw (E) .. controls (-2.9,1) .. (F);
   \draw (F) .. controls (-2.6,-1.6) .. (G);
  \draw (G) .. controls (-1.7,-2.4) .. (H);
 \draw (J) .. controls (3.6,0.5) .. (K);
  \draw (K) .. controls (3,2.2) .. (B);
  \node at (2.2,1) {$\Omega^{(i_0)}$};
  \node at (0,2.5) {$\Omega^{(i_1)}$};
  \node at (-1.7,-1.2) {$\Omega^{(i_{\nu_i-1})}$};
  \node at (-2.2,1.2) {$\ldots$};
  \node at (-0.7,0.2) { $\f{x}^{(i)}$};
  \node at (2.1,-0.3) { $\Sigma^{(i_0)}$};
  \node at (0.9,2) { $\Sigma^{(i_1)}$};
  \node at (-1.0,-2.0) { $\Sigma^{(i_{\nu_i})}$};
  \draw[->] (0,0.3) -- (-0.75,1.2);
  \draw[->] (0,0.3) -- (0.7,1.2);
  \draw[->] (0.3,0.1) -- (1.05,1);
  \draw[->] (0.3,0.1) -- (1.3,-0.25);
  \draw[->] (-0.25,-0.3) -- (-1.3,-0.6);
  \draw[->] (-0.25,-0.3) -- (-0.65,-1.5);
  \node at (-0.3,1.0) {\scriptsize $\xi_2$};
  \node at (0.3,1.0) {\scriptsize $\xi_1$};
  \node at (1.05,0.6) {\scriptsize $\xi_2$};
  \node at (0.95,0.1) {\scriptsize $\xi_1$};
  \node at (-0.75,-1.15) {\scriptsize $\xi_2$};
  \node at (-0.8,-0.7) {\scriptsize $\xi_1$};
\end{tikzpicture}
 }
\end{tabular}
\end{center}
\caption{Representation in standard form with respect to a vertex~$\f{x}^{(i)}$, $i \in \indChi$. Left: 
Edges~$\Sigma^{(i_0)}$, $\ldots$, $\Sigma^{(i_{\nu_i -1})}$ and patches $\Omega^{(i_0)}$, $\ldots$, $\Omega^{(i_{\nu_i-1})}$ around a common inner vertex~$\f{x}^{(i)}$, $i \in \indChiI$. Right: Edges~$\Sigma^{(i_0)}$, $\ldots$, $\Sigma^{(i_{\nu_i})}$ and patches $\Omega^{(i_0)}$, $\ldots$, $\Omega^{(i_{\nu_i-1})}$ around a common boundary vertex~$\f{x}^{(i)}$, $i \in \indChiB$.}
 \label{fig:standard_form_vertex}
\end{figure}

\begin{remark} \label{rem:orientation}
Obviously, the standard configuration cannot be attained around every vertex without changing the parameterizations. In fact, for each vertex ${\bf x}^{(i)}$, $i \in \indChi$, with patch valence $\nu_i$ one would have to define a suitable parameterization ${\bf G}_\chi^{(i,m)}$ of $\Omega^{(i_m)}$, for $m = 0, \ldots, \nu_i-1$, to recover the standard configuration. These parameterizations are obtained from $\f{F}^{(i_m)}$ by possibly reversing each one of the parametric directions, and swapping them, which gives a total of eight possible combinations (four if the Jacobian is assumed to be positive). Similar parameterizations ${\bf G}_\Sigma^{(i,0)}$ and ${\bf G}_\Sigma^{(i,1)}$ would be also needed for the standard configuration on each edge. We have preferred to keep $\f{F}^{(i_m)}$ to alleviate notation.
\end{remark}

\subsubsection{Analysis-suitable $G^1$ parameterizations and gluing data}\label{subsubsec:ASG1}
From now on we restrict ourselves to a particular class of multi-patch geometries~$\f{F}$, called analysis-suitable~$G^1$ multi-patch parameterizations, 
which possess for each inner edge~$\Sigma^{(i)}$, $i \in \indSigmaI$, linear functions~$\alpha^{(i,0)}$, $\alpha^{(i,1)}$, $\beta^{(i,0)}$ and 
$\beta^{(i,1)}$, with $\alpha^{(i,0)}$ and $\alpha^{(i,1)}$ relatively prime, such that for all $\xi \in [0,1]$
\begin{equation*}
 \alpha^{(i,0)}  (\xi) \alpha^{(i,1)}  (\xi) > 0
\end{equation*}
and
\begin{equation*}
\alpha^{(i,0)} (\xi)  \Dv \f{F}^{(i_1)}(\xi,0)  +
        \alpha^{(i,1)}(\xi) \Du  \f{F}^{(i_0)}(0,\xi) + \beta^{(i)} (\xi)
        \Dv  \f{F}^{(i_0)}(0,\xi)  =\boldsymbol{0},
\end{equation*}
with
\begin{equation} \label{eq:beta}
 \beta^{(i)} (\xi)= \alpha^{(i,0)} (\xi) \beta^{(i,1)}(\xi)+ \alpha^{(i,1)}(\xi)\beta^{(i,0)}(\xi),
\end{equation}
see \cite{CoSaTa16,KaSaTa19a}. This class of {\RV parameterizations} is exactly the one which allows the design of $C^1$ isogeometric spaces with optimal 
polynomial reproduction
properties \cite{CoSaTa16,KaSaTa17b}.
Details for the computation of the gluing data are given in Appendix~\ref{sec:gluing_data}. We note that, for each boundary edge~$\Sigma^{(i)}$, $i \in \indSigmaB$, we can simply assign trivial functions~$\alpha^{(i,0)} \equiv 1$ and 
$\beta^{(i,0)} \equiv 0$. 

Examples of analysis-suitable $G^1$ multi-patch geometries are e.g. piecewise bilinear parameterizations \cite{CoSaTa16,KaBuBeJu16,KaViJu15}, but there exist 
different methods to generate from a given, possibly non-analysis-suitable $G^1$ multi-patch geometry, an analysis-suitable $G^1$ parameterization, 
see \cite{KaSaTa17b,KaSaTa19b}.

In addition we define, for each inner edge $\Sigma^{(i)}$ with $i \in \indSigmaI$, the vectors
\begin{equation*} 
 \f{t}^{(i)}(\xi) = \Dv \f{F}^{(i_{0})}(0,\xi) = \Du \f{F}^{(i_{1})}(\xi,0),
\end{equation*}
and
\begin{equation*} 
\begin{array}{lll}
  \f{d}^{(i)}(\xi) &=& \frac{1}{\alpha^{(i,0)}(\xi)}\left( \Du \f F^{(i_{0})}(0,\xi) + \beta^{(i,0)}(\xi) \, \Dv \f F^{(i_{0})}(0,\xi)\right) \\
  &=& -\frac{1}{\alpha^{(i,1)}(\xi)}\left( \Dv \f F^{(i_{1})}(\xi,0) + \beta^{(i,1)}(\xi) \, \Du \f F^{(i_{1})}(\xi,0)\right).
\end{array}
\end{equation*}
In case of a boundary edge~$\Sigma^{(i)}$, $i \in \indSigmaB$, with $\Sigma^{(i)}=\f{F}^{(i_{0})}(0,(0,1))$,
we use the same definitions based only on the parameterization $\mathbf{F}^{(i_0)}$, noting that $\alpha^{(i,0)} \equiv 1$ and $\beta^{(i,0)}\equiv 0$.


%% file: s4_C1_space.tex
We now define the $C^1$ spline space on one level, and a particular subspace which maintains 
the same numerical approximation properties.

\subsubsection{Construction of a particular $C^1$ isogeometric subspace} \label{subsec:particularC1space}

The $C^1$ isogeometric space~$\UV$ with respect to the multi-patch geometry~$\f{F}$ is given by
\begin{equation} \label{eq:spaceV}
\UV = \{ \phi \in C^1(\overline{\Omega}) \ :\  \phi \circ \f{F}^{(i)} \in \US{p}{r}\otimes\US{p}{r} , i \in \indOmega \}. 
\end{equation}
The space is associated to a mesh, a partition of the domain determined by the knot vectors of the univariate spaces $\US{p}{r}$ and the parameterization $\mathbf{F}$, {\RV and which is given by
\begin{equation} \label{eq:mesh}
\GG = \left\{ \mathbf{F}^{(i)} (\hat{Q}) : i \in \indOmega, \hat{Q} = \left[
\textstyle{\frac{j}{k+1},\frac{j+1}{k+1}} \right], \text{ for } j = 0, \ldots, k \right\}.
\end{equation}
}

Since the space~$\UV$ has already for the case of two patches a complex structure and its dimension depends on the geometry \cite{KaSaTa17a},
we consider instead the simpler subspace~$\UW$ firstly introduced in \cite{KaSaTa19a}, 
which maintains the 
numerical approximation properties of~$\UV$ and whose
dimension is independent of the geometry. The space~$\UW$ is defined as 
\begin{align} 
& \UW = \mathrm{span}\, \PhiB , \text{ with }\notag \\ 
& \PhiB = \PhiB_\Omega \cup \Phi_\Sigma \cup \Phi_\chi, 
\text{ and } 
\PhiB_\Omega = \bigcup_{i \in \indOmega} \PhiB_{\Omega^{(i)}}, \;
\PhiB_\Sigma = \bigcup_{i \in \indSigma} \PhiB_{\Sigma^{(i)}}, \;
\PhiB_\chi = \bigcup_{i \in \indChi} \PhiB_{\f{x}^{(i)}}, \label{eq:basis_A}
\end{align}
where the three sets of basis functions are respectively called patch interior, edge and vertex basis functions, and are defined in detail below.
All functions are generated in such a way that they are $C^1$-smooth on $\Omega$ and for the case of vertex functions even $C^2$-smooth at the corresponding 
vertex~$\f{x}^{(i)}$. 
To ensure the design of the space~$\UW$, a minimal number~$k$ of different inner knots is needed,
given by $k \geq \max(0,\frac{5-p}{p-r-1})$,
as already requested in the definition of the spline space~$\US{p}{r}$ in Section~\ref{sec:multipatch}. Below, 
we summarize the construction of the single functions and refer to \cite{KaSaTa19a,KaSaTa19b} for more details.

\subsubsection{Patch interior basis functions} \label{subsec:interior}
The patch interior basis functions are ``standard'' isogeometric functions whose value and first derivatives are zero on every edge and vertex. We start defining the 
{\CB univariate index set $I = \{2, \ldots, n-3 \}$, and the multivariate index sets}
\begin{align*}
& {\CB \JOmega= I \times I, \quad \text{ with } \quad \JOmega \subset \widetilde{\bf J}_\Omega = \{\mathbf{j} = (j_1, j_2) : j_1, j_2 = 0, \ldots, n-1 \}.}
\end{align*}
Then, for each patch~$\Omega^{(i)}$, $i \in \indOmega$, we define the set of patch interior basis functions as
\begin{equation*}
\PhiB_{\Omega^{(i)}} = \left\{ \phiP{i}{\f{j}} \ : \ \f{j} \in \JOmega \right\} ,
\end{equation*}
where, {\CB defining $N_{\mathbf{j},p}^r(\xi_1,\xi_2) = N_{j_1,p}^r(\xi_1) N_{j_2,p}^r(\xi_2)$,} the functions~$\phiP{i}{\f{j}}$ are given by
\begin{equation*}
    \phiP{i}{\mathbf{j}} (\f{x}) =  
    \begin{cases}
 \left( N_{\mathbf{j},p}^r \circ \left(\f{F}^{(i)}\right)^{-1}\right)(\f{x}) & \mbox{if }\f \, \f{x} \in \overline{\Omega^{(i)}} ,\\
0 & \mbox{otherwise}.
\end{cases}
\end{equation*}
Therefore, the functions in $\PhiB_{\Omega^{(i)}}$ trivially have support contained in $\Omega^{(i)}$, and their value and their gradient vanish on the boundary of $\Omega^{(i)}$. Moreover, their maximal support is attained for regularity $r=p-2$, and consists of ${\MK \lceil \frac{p+1}{2} \rceil^2}$ elements, see Fig.~\ref{fig:support-interior}.

\begin{figure}[!t]
\begin{subfigure}[Patch interior function.]{
\hspace{-5mm} \includegraphics[width=0.33\textwidth,trim=15cm 9cm 7cm 5cm,clip]{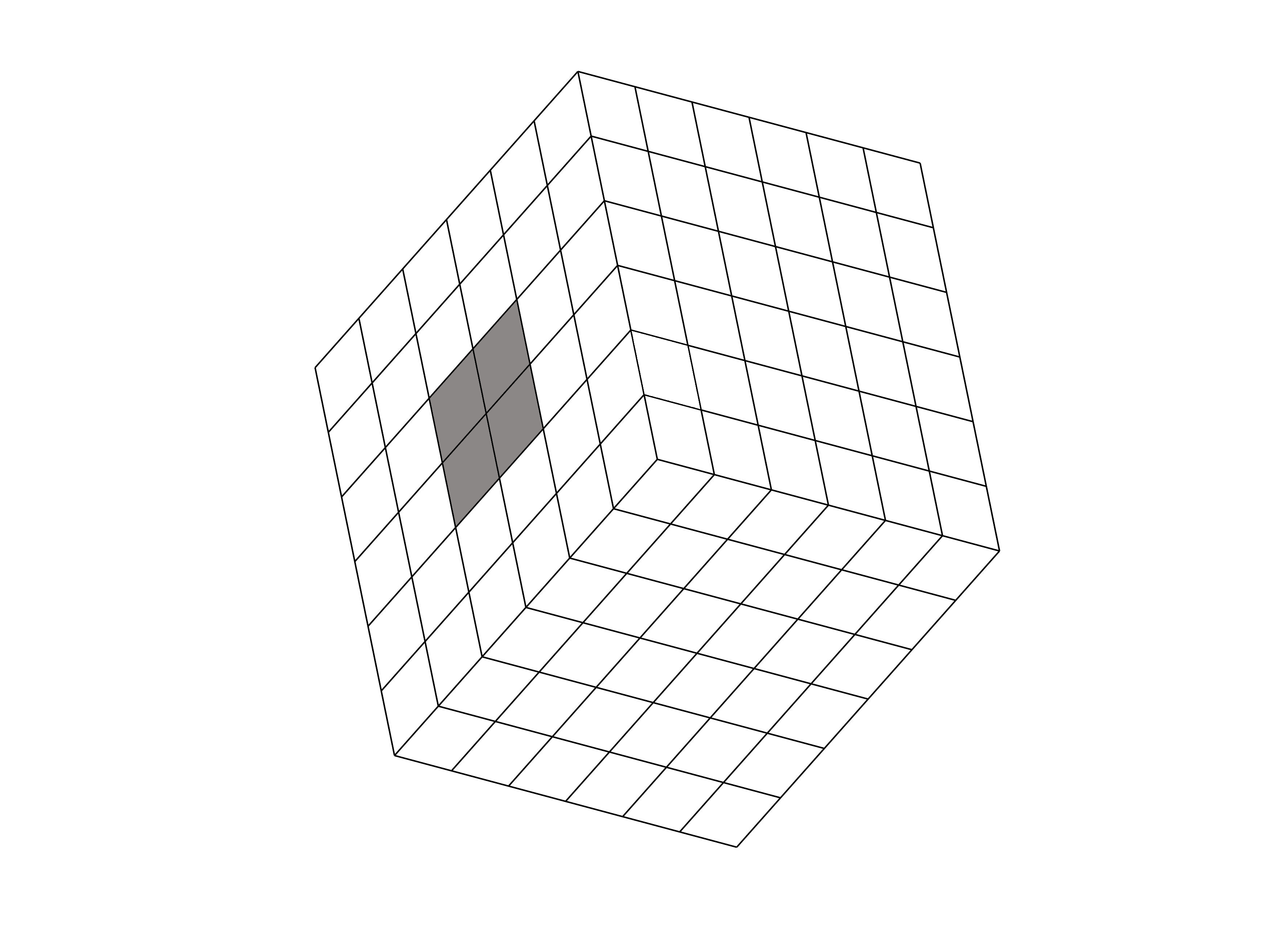}
\label{fig:support-interior}
}
\end{subfigure}
\begin{subfigure}[Edge function.]{
\hspace{-5mm} \includegraphics[width=0.33\textwidth,trim={15cm 9cm 7cm 5cm},clip]{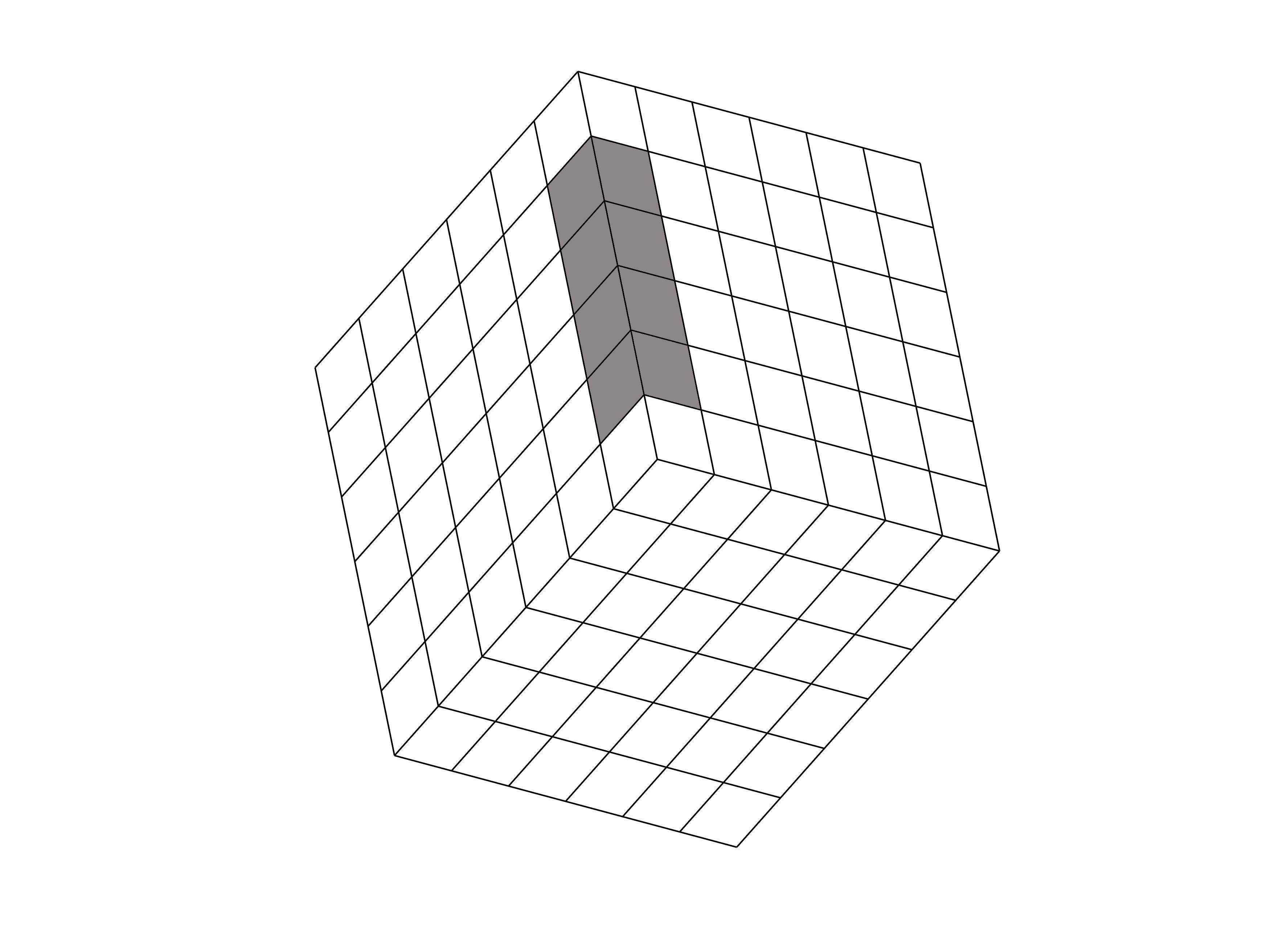}
\label{fig:support-edge}
}
\end{subfigure}
\begin{subfigure}[Vertex function.]{
\hspace{-5mm} \includegraphics[width=0.33\textwidth,trim={15cm 9cm 7cm 5cm},clip]{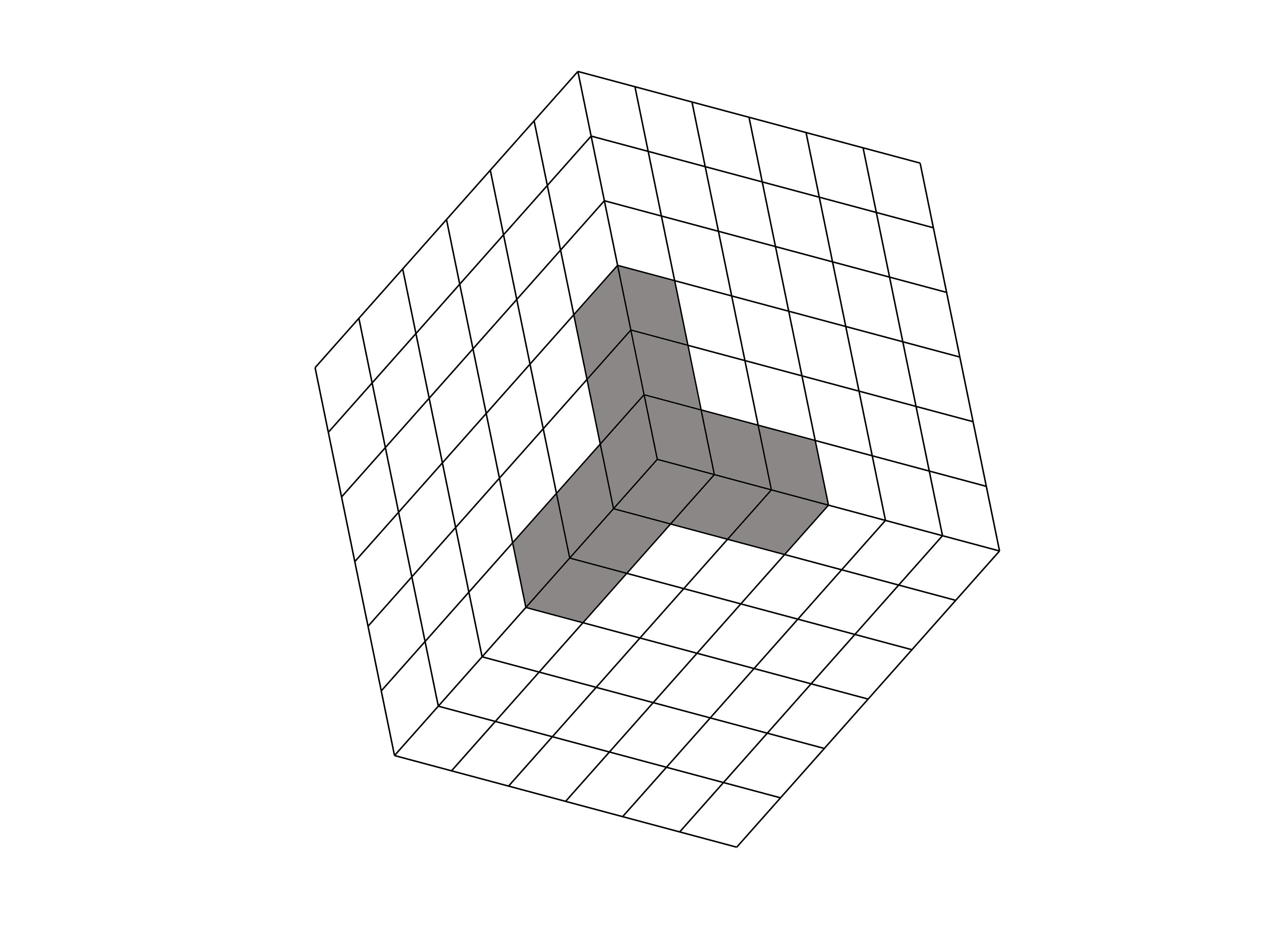}
\label{fig:support-vertex}
}
\end{subfigure}
\caption{Example of maximal support of a patch interior basis function (left), an edge basis function (center) and a vertex basis function (right) for degree $p=3$ and regularity $r=1$.}
\label{fig:basis_supp}
\end{figure}

Note that a function~$\phiP{i}{\f{j}}$ can be also defined for any index~$\f{j}$ of the extended index set $\widetilde{\bf J}_\Omega$, and for later use we also define the sets
\[
\widetilde{\PhiB}_{\Omega} = \bigcup_{i \in \indOmega} \widetilde{\PhiB}_{\Omega^{(i)}}, \text{ with } \widetilde{\PhiB}_{\Omega^{(i)}} = \left\{ \phiP{i}{\f{j}} \ : \ \f{j} \in \widetilde{\bf J}_\Omega \right\}.
\]
However, the functions in $\widetilde{\PhiB}_{\Omega} \setminus \PhiB_{\Omega}$ 
do not belong to the space $\UW$, and for this reason we call them ``extended'' patch interior functions.

\subsubsection{Edge basis functions} \label{sec:edge_functions}
Edge basis functions have support in the two patches adjacent to the edge, or one patch for boundary edges. We start defining the index sets
\begin{align*}
& \widetilde{\bf J}_{\Sigma}=\{\mathbf{j} = (j_1,j_2): j_1= 0,\ldots, n_{j_2}-1; \, j_2=0,1\}, \\
& \JSigma=\{\mathbf{j} = (j_1,j_2): j_1= 3-j_2,\ldots, n_{j_2}-4+j_2; \ j_2=0,1\} \subset \widetilde{\bf J}_{\Sigma},
\end{align*}
{\RV where $n_{j_2}$ will take values $n_0$ or $n_1$, as defined in Section~\ref{sec:univariate_functions}.}
Then, for each edge~$\Sigma^{(i)}$, $i \in \indSigma$, and assuming the same orientation described in Section~\ref{sec:multipatch}, we define the set of edge basis functions associated to the edge $\Sigma^{(i)}$ as 
\begin{equation*}
\PhiB_{\Sigma^{(i)}} = \left\{ \phiI{i}{\f{j}} \ : \ \f{j} \in \JSigma \right\} , 
\end{equation*}
where each basis function~$\phiI{i}{\f{j}}$, is defined as
\begin{equation*}
    \phiI{i}{\f{j}} (\f{x}) =  
    \begin{cases}
 \left( \fSigma_{\mathbf{j}}^{(i,m)} \circ \left(\f{F}^{(i_m)}\right)^{-1}\right)(\f{x}) & \mbox{if }\f \, \f{x} \in \overline{\Omega^{(i_m)}} , \mbox{ }m=0,1,\\
0 & \mbox{otherwise},
\end{cases}
\end{equation*}
where the functions $\fSigma_{\mathbf{j}}^{(i,m)} = \fSigma_{(j_1,j_2)}^{(i,m)}$ have been introduced in \cite{KaSaTa19a}. For completeness, we also give them in Appendix~\ref{sec:function-details}.

It is easy to verify that the functions in $\PhiB_{\Sigma^{(i)}}$ are supported in $\Omega^{(i_0)}$ and $\Omega^{(i_1)}$ (or $\Omega^{(i_0)}$ for boundary edges), and that their value, and first and second derivatives vanish on $\partial \Sigma^{(i)}$. Moreover, the maximal support of an edge function consists of $2 (p+1)$ elements, with $p+1$ elements on each patch, see Fig.~\ref{fig:support-edge}.

Analogously to the patch interior functions, we can also define the edge function~$\phiI{i}{\f{j}}$ for each index~$\f{j} \in \widetilde{\bf J}_{\Sigma}$, giving the sets of functions
\[
\widetilde{\PhiB}_{\Sigma} = \bigcup_{i \in \indSigma} \widetilde{\PhiB}_{\Sigma^{(i)}}, \text{ with } \widetilde{\PhiB}_{\Sigma^{(i)}} = \left\{ \phiI{i}{\f{j}} \ : \ \f{j} \in \widetilde{\bf J}_{\Sigma} \right\}.
\]
However, the functions in $\widetilde{\PhiB}_{\Sigma} \setminus \PhiB_{\Sigma}$ 
do not belong to the space $\UW$, so we call them ``extended'' edge functions.
 
\begin{remark}
In the definition of $\fSigma_{(j_1,j_2)}^{(i,m)}$, the subindex $j_2$ refers to the type of edge basis functions: the values 0 and 1 respectively refer to \emph{trace} and \emph{derivative} edge functions. These were denoted by $\Phi^{\Gamma_0}$ and $\Phi^{\Gamma_1}$ in \cite{BrGiKaVa20}. We have preferred to follow the indexing in \cite{KaSaTa19a,KaSaTa19b}, to reduce the number of symbols.
\end{remark}

\subsubsection{Vertex basis functions} \label{sec:vertex_functions}
Vertex basis functions have support in all the patches adjacent to the vertex, {\RV and there} are always six vertex functions associated to each vertex, independently of its valence. {\MK These properties of the vertex functions are due to the $C^2$ interpolation condition~\eqref{eq:C2interpolation} {\RV below}.} For each vertex~$\f{x}^{(i)}$, $i \in \indChi$ we define the set of vertex basis functions as
\begin{equation*}
\PhiB_{\f{x}^{(i)}} = \left\{ \phiV{i}{\f{j}} \ : \ \f{j} \in \JChi  \right\}, \text{ with } \JChi = \{\mathbf{j} = (j_1,j_2): j_1,j_2 =0,1,2; \ j_1+j_2 \leq 2 \}.
\end{equation*}
To define the vertex basis functions $\phiV{i}{\f{j}}$, and recalling that the patch valence is denoted by $\nu_i$, we first define the factor 
\begin{equation} \label{eq:sigma-factor}
  \sigma_i = \left (  \frac{1}{ p (k+1)  \nu_i} \sum_{m = 0}^{\nu_i-1} \|\nabla \f{F}^{(i_{m})} (0,0) \|_{\infty} \right ) ^{-1},
\end{equation}
which will be used to uniformly scale the vertex functions with respect to the infinity norm. Then, for each index $\f{j}=(j_1,j_2) \in \JChi$, the corresponding vertex function is defined as
\begin{equation}\label{vertdef}
    \phiV{i}{\mathbf{j}} (\f{x}) =  
    \begin{cases}
 \sigma_i^{j_1+j_2}\left( \left( g_{\mathbf{j}}^{(i,m,\rm{prec})} + g_{\mathbf{j}}^{(i,m,\rm{next})} -h_{\mathbf{j}}^{(i,m)}   \right)
 \circ \left(\f{F}^{(i_m)}\right)^{-1}\right)(\f{x}) \quad \mbox{} & \\
\hfill \mbox{if }\f \, \f{x} \in \overline{\Omega^{(i_m)}} , \, m=0,\ldots,\nu_i-1,\\
0 \hfill \mbox{otherwise},
\end{cases}
\end{equation}
where the functions $g_{\mathbf{j}}^{(i,m,\rm{prec})}$ and $g_{\mathbf{j}}^{(i,m,\rm{next})}$ respectively involve the edge preceding and following the patch $\Omega^{(i_m)}$, in counterclockwise direction, while the function $h_{\mathbf{j}}^{(i,m)}$ uses information from both edges. These functions were defined in \cite{KaSaTa19a}, and for convenience we show their expressions in Appendix~\ref{sec:function-details}.
The construction of the vertex functions~$\phiV{i}{\f{j}}$, for $\f{j} \in \JChi$, ensures that their support is contained in the patches around the vertex ${\bf x}^{(i)}$, 
and
\begin{equation} \label{eq:C2interpolation}
\Du^{z_1} \Dv^{z_2} \left(\phiV{i}{\f{j}}\right) (\f{x}^{(i)}) = 
 \sigma_i^{j_1+j_2} {\RV \delta_{j_1 z_1} \delta_{j_2 z_2}}, \mbox{ } 0 \leq z_1,z_2 \leq 2, \mbox{ }z_1+z_2 \leq 2,
\end{equation}
see \cite{KaSaTa19a}. Moreover, the support of a vertex basis function consists at most of five elements per patch, for a total of $5 \nu_i$ elements, see Fig.~\ref{fig:support-vertex}.

\subsection{Representation of the basis in terms of standard B-splines} \label{subsec:standard_representation}
In the proofs of linear independence in Section~\ref{sec:lli_one_level}, and also for an efficient implementation of the $C^1$ space, we will make use of the representation of the $C^1$ basis functions of the previous section in terms of the standard B-spline basis. {\MK For this purpose, we will introduce below several column vectors of functions, \RV both for standard B-splines and for the $C^1$ basis functions, and with some abuse of notation 
we will define the vectors using a double index, with the meaning that the entries on each vector are ordered moving first on the first index, and then on the second, as it {\MK is} usually done for the implementation \cite{vazquez2016}}.

{\RV For indices $i \in \indOmega, j \in \indSigma$ and $k \in \indChi$, let us use the notation}
\begin{equation} \label{eq:vector-all}
\phivec_{\Omega^{(i)}}=[\phi^{\Omega^{(i)}}_{\bf j}]_{{\bf j}\in \JOmega}, \quad
\phivec_{\Sigma^{(j)}}=[\phi^{\Sigma^{(j)}}_{\bf j}]_{{\bf j}\in \JSigma}, \quad
\phivec_{{\bf x}^{(k)}}=[\phi^{{\bf x}^{(k)}}_{\bf j}]_{{\bf j}\in \JChi},
\end{equation}
to respectively indicate the {\MK column} vectors of patch interior basis functions, of edge basis functions and of vertex basis functions. 
As we want to write them in terms of standard B-splines, we also introduce the {\MK column} vector of B-splines basis functions of the space $\US{p}{r} \otimes \US{p}{r}$ mapped into $\Omega^{(i)}$, given by
\begin{align*} 
& \mathbf{N}^{(i)} 
=[N_{\mathbf{j},p}^r \circ (\mathbf{F}^{(i)})^{-1}]^T_{\mathbf{j} \in \widetilde{\mathbf{J}}_\Omega}.
\end{align*}
{\CB Recalling} the set of indices $I$ from Section~\ref{subsec:interior}, we also define the subvectors
\begin{equation}\label{stdbspl-new}
\begin{array}{ll}
{\bf N}_{0}^{(i)}=[N_{\mathbf{j},p}^r\circ({\bf F}^{(i)})^{-1}]_{\mathbf{j} \in \{0,1\} \times \{0,1\}}, &
{\bf N}_{1}^{(i)}=[N_{\mathbf{j},p}^r\circ({\bf F}^{(i)})^{-1}]_{\mathbf{j} \in I \times \{0,1\}}, \\
{\bf N}_{2}^{(i)}=[N_{\mathbf{j},p}^r\circ({\bf F}^{(i)})^{-1}]_{\mathbf{j} \in \{0,1\} \times I}, &
{\bf N}_{3}^{(i)}=[N_{\mathbf{j},p}\circ({\bf F}^{(i)})^{-1}]_{\mathbf{j} \in I \times I},
\end{array}
\end{equation}
respectively corresponding to the basis functions with nonzero value or nonzero first derivatives on the bottom left vertex, on the bottom edge (but zero on every vertex), on the left edge (but zero on every vertex), and internal functions with zero value and derivative on every edge, {\MK see Fig.~\ref{fig:subvectors}}. With some abuse of notation, we denote in the same way the vectors of B-splines extended by zero outside $\Omega^{(i)}$.

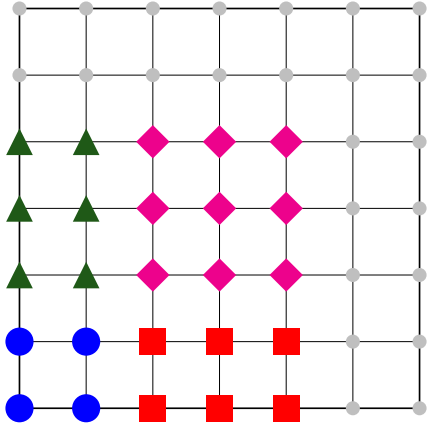
\begin{figure}
\begin{center}
\begin{tabular}{cc}
\resizebox{0.35\textwidth}{!}{
 \begin{tikzpicture}
  \coordinate(A) at (0,0); \coordinate(B) at (6,0); \coordinate(C) at (0,6); \coordinate(D) at (6,6); 
  \draw[line width=0.25mm] (A) -- (B);
  \draw[line width=0.25mm] (B) -- (D);
  \draw[line width=0.25mm] (D) -- (C);
  \draw[line width=0.25mm] (C) -- (A);
  \draw[line width=0.15mm] (0,1) -- (6,1);
  \draw[line width=0.15mm] (0,2) -- (6,2);
  \draw[line width=0.15mm] (0,3) -- (6,3);
  \draw[line width=0.15mm] (0,4) -- (6,4);
   \draw[line width=0.15mm] (0,5) -- (6,5);
  \draw[line width=0.15mm] (1,0) -- (1,6);
  \draw[line width=0.15mm] (2,0) -- (2,6);
  \draw[line width=0.15mm] (3,0) -- (3,6);
  \draw[line width=0.15mm] (4,0) -- (4,6);
  \draw[line width=0.15mm] (5,0) -- (5,6);
  \fill[blue] (0,0) circle (6pt);
  \fill[blue] (1,0) circle (6pt);
  \fill[blue] (0,1) circle (6pt);
  \fill[blue] (1,1) circle (6pt);
  \draw [fill=red, draw=none] (1.8,-0.2) rectangle (2.2,0.2);
  \draw [fill=red, draw=none]
  (2.8,-0.2) rectangle (3.2,0.2);
  \draw [fill=red, draw=none]
  (3.8,-0.2) rectangle (4.2,0.2);
  \draw [fill=red, draw=none]
  (1.8,0.8) rectangle (2.2,1.2);
  \draw [fill=red, draw=none]
  (2.8,0.8) rectangle (3.2,1.2);
  \draw [fill=red, draw=none]
  (3.8,0.8) rectangle (4.2,1.2);
  \draw[fill=darkgreen, draw=none] (-0.2,1.8) -- (0.2,1.8) -- (0,2.2)-- (-0.2,1.8);
   \draw[fill=darkgreen, draw=none] (-0.2,2.8) -- (0.2,2.8) -- (0,3.2)-- (-0.2,2.8);
    \draw[fill=darkgreen, draw=none] (-0.2,3.8) -- (0.2,3.8) -- (0,4.2)-- (-0.2,3.8);
   \draw[fill=darkgreen, draw=none] (0.8,1.8) -- (1.2,1.8) -- (1,2.2)-- (0.8,1.8);
   \draw[fill=darkgreen, draw=none] (0.8,2.8) -- (1.2,2.8) -- (1,3.2)-- (0.8,2.8);
    \draw[fill=darkgreen, draw=none] (0.8,3.8) -- (1.2,3.8) -- (1,4.2)-- (0.8,3.8);
    \draw[fill=magenta, draw=none] (2,1.75) -- (1.75,2) -- (2,2.25)-- (2.25,2)--(2,1.75);
    \draw[fill=magenta, draw=none] (3,1.75) -- (2.75,2) -- (3,2.25)-- (3.25,2)--(3,1.75);
    \draw[fill=magenta, draw=none] (4,1.75) -- (3.75,2) -- (4,2.25)-- (4.25,2)--(4,1.75);
    \draw[fill=magenta, draw=none] (2,2.75) -- (1.75,3) -- (2,3.25)-- (2.25,3)--(2,2.75);
    \draw[fill=magenta, draw=none] (3,2.75) -- (2.75,3) -- (3,3.25)-- (3.25,3)--(3,2.75);
    \draw[fill=magenta, draw=none] (4,2.75) -- (3.75,3) -- (4,3.25)-- (4.25,3)--(4,2.75);
    \draw[fill=magenta, draw=none] (2,3.75) -- (1.75,4) -- (2,4.25)-- (2.25,4)--(2,3.75);
    \draw[fill=magenta, draw=none] (3,3.75) -- (2.75,4) -- (3,4.25)-- (3.25,4)--(3,3.75);
    \draw[fill=magenta, draw=none] (4,3.75) -- (3.75,4) -- (4,4.25)-- (4.25,4)--(4,3.75);
    \fill[lightgray] (5,0) circle (3pt);
    \fill[lightgray] (6,0) circle (3pt);
    \fill[lightgray] (5,1) circle (3pt);
    \fill[lightgray] (6,1) circle (3pt);
    \fill[lightgray] (5,2) circle (3pt);
    \fill[lightgray] (6,2) circle (3pt);
    \fill[lightgray] (5,3) circle (3pt);
    \fill[lightgray] (6,3) circle (3pt);
    \fill[lightgray] (5,4) circle (3pt);
    \fill[lightgray] (6,4) circle (3pt);
    \fill[lightgray] (5,5) circle (3pt);
    \fill[lightgray] (6,5) circle (3pt);
    \fill[lightgray] (5,6) circle (3pt);
    \fill[lightgray] (6,6) circle (3pt);
    \fill[lightgray] (0,5) circle (3pt);
    \fill[lightgray] (0,6) circle (3pt);
    \fill[lightgray] (1,5) circle (3pt);
    \fill[lightgray] (1,6) circle (3pt);
    \fill[lightgray] (2,5) circle (3pt);
    \fill[lightgray] (2,6) circle (3pt);
    \fill[lightgray] (3,5) circle (3pt);
    \fill[lightgray] (3,6) circle (3pt);
    \fill[lightgray] (4,5) circle (3pt);
    \fill[lightgray] (4,6) circle (3pt);
\end{tikzpicture}
}
\end{tabular}
\end{center}
\caption{\MK Ordering of the basis functions ${\bf N}_{0}^{(i)}$ (blue circles), ${\bf N}_{1}^{(i)}$ (red squares), ${\bf N}_{2}^{(i)}$ (green triangles) and ${\bf N}_{3}^{(i)}$ (magenta diamonds) for a patch with $7 \times 7$ basis functions. The basis functions in ${\bf N}^{(i)}$ not contained in any of those vectors are visualized as small gray dots.}
 \label{fig:subvectors}
\end{figure}

We first note that the patch interior basis functions in the vector $\phivec_{\Omega^{(i)}}$ have their support contained in the patch $\Omega^{(i)}$, where they coincide with standard B-splines, so their representation in terms of B-splines is trivially given by $\phivec_{\Omega^{(i)}} = \mathbf{N}_3^{(i)}$.

The edge basis functions associated with an edge $\Sigma^{(i)}$ in the multi-patch case are a subset of the edge functions for the two patch case. Recalling the notation for ``extended'' edge functions in the previous section, and introducing the corresponding vector of functions $\widetilde \phivec_{\Sigma^{(i)}}:=[\phi^{\Sigma^{(i)}}_{\bf j}]_{{\bf j}\in \widetilde{\bf J}_\Sigma}$, we know from the two-patch case \cite{KaSaTa17a,BrGiKaVa20} that edge functions can be written as standard B-splines in the form
\begin{equation}\label{reptwoedge}
\widetilde \phivec_{\Sigma^{(i)}} = \widetilde E_{i,0}{\bf N}^{(i_0)} + \widetilde E_{i,1}{\bf N}^{(i_1)},
\end{equation}
where the only B-splines that play a role are the ones with non-vanishing value or derivative on the edge.
Therefore the relation for the multi-patch case is immediately given by
\begin{equation}\label{repmultiedge}
\phivec_{\Sigma^{(i)}} = E_{i,0}{\bf N}_2^{(i_0)} + E_{i,1}{\bf N}_1^{(i_1)},
\end{equation} 
where $E_{i,k}$ is the submatrix of $\widetilde E_{i,k}$ containing only the rows corresponding to the edge functions in $\PhiB_{\Sigma^{(i)}}$, i.e., to the indices $\JSigma$, and the columns of B-splines with nonzero coefficients.

For the vertex basis functions 
associated with the vertex ${\bf x}^{(i)}$, from their definition \eqref{vertdef} and equations \eqref{eq:func_g1}-\eqref{eq:func_g3} follows the relation
\begin{equation}\label{repvert}
\phivec_{{\bf x}^{(i)}}= \sum_{m=0}^{\nu_i-1} \left(K_{i,m}\widehat E_{i_m,1} 
\begin{bmatrix}
{\bf N}_0^{(i_m)} \\
{\bf N}_1^{(i_m)}
\end{bmatrix}
+ K_{i,m+1}\widehat E_{i_{m+1},0} 
\begin{bmatrix}
{\bf N}_0^{(i_m)} \\
{\bf N}_2^{(i_m)}
\end{bmatrix}
- V_{i,m}{\bf N}_0^{(i_m)} 
\right),
\end{equation}
where $\Sigma^{(i_{m})}$ and $\Sigma^{(i_{m+1})}$ are the two edges of the patch $\Omega^{(i_m)}$ containing the vertex ${\bf x}^{(i)}$, $\widehat E_{i_m,1}$ and $\widehat E_{i_{m+1},0}$ are the submatrices of $\widetilde E_{i_m,1}$ and $\widetilde E_{i_{m+1},0}$ containing only the rows corresponding to the five ``extended'' edge functions in $\widetilde \PhiB_{\Sigma^{(i_{m})}}\setminus \PhiB_{\Sigma^{(i_{m})}}$ and $\widetilde \PhiB_{\Sigma^{(i_{m+1})}}\setminus\PhiB_{\Sigma^{(i_{m+1})}}$ close to the vertex, respectively, and the columns with nonzero coefficients.
The detailed computations which lead to the matrices $K_{i,m}$, $K_{i,m+1}$ and $V_{i,m}$, based on using the expression of the modified basis functions from Appendix~\ref{sec:modified_basis},
are given in Appendix~\ref{sec:computations} for general regularity $r \le p-2$. For convenience we give their expressions here for the case $r=p-2$.


The matrices $K_{i,m}$, $K_{i,m+1}$ are of size $6\times 5$ and each one of their rows, which corresponds to a different value of ${\bf j} = (j_1, j_2)$, is of the form, for $s= m, m+1$, 
\begin{align*}
\sigma_i^{j_1+j_2}\left[\begin{array}{c}
\coefc^{(i_{s})}_{\mathbf{j},0}, 
\coefc^{(i_{s})}_{\mathbf{j},0} + \frac{\coefc^{(i_{s})}_{\mathbf{j},1}}{p(k+1)} , 
\coefc^{(i_{s})}_{\mathbf{j},0} + \frac{3 \coefc^{(i_{s})}_{\mathbf{j},1}}{p(k+1)} + \frac{2 \coefc^{(i_{s})}_{\mathbf{j},2}}{p(p-1)(k+1)^2} , 
\frac{\coefd^{(i_{s})}_{\mathbf{j},0}}{p(k+1)} 
\frac{\coefd^{(i_{s})}_{\mathbf{j},0}}{p(k+1)}+\frac{\coefd^{(i_{s})}_{\mathbf{j},1}}{p(p-1)(k+1)^2}
\end{array}
 \right],
\end{align*} 
with $\sigma_i$ being the factor in \eqref{eq:sigma-factor}, and the $\coefc_\mathbf{j}$ and $\coefd_\mathbf{j}$ coefficients defined in Appendix~\ref{sec:function-details}.

The matrix $V_{i,m}$ is of size $6\times 4$ and each one of their rows, corresponding to a different value of ${\bf j} = (j_1, j_2)$, is then of the form
\begin{align*}
\sigma_i^{j_1+j_2}\left[\begin{array}{c}
\coefc_{\mathbf{j},0}^{(i_m)}  ,  \,
\coefc_{\mathbf{j},0}^{(i_m)} + \frac{\coefc_{\mathbf{j},1}^{(i_m)}}{p(k+1)} , \,
\coefc_{\mathbf{j},0}^{(i_m)} + \frac{\coefc_{\mathbf{j},1}^{(i_{m+1})}}{p(k+1)},  \,
\coefc_{\mathbf{j},0}^{(i_m)} + \frac{\left(\coefc_{\mathbf{j},1}^{(i_m)}+\coefc_{\mathbf{j},1}^{(i_{m+1})}+\frac{\coefe_{\mathbf{j},(1,1)}^{(i_m)}}{p(k+1)}\right)}{p(k+1)}
\end{array}
 \right],
\end{align*}
with the four columns corresponding to the four B-splines in $\mathbf{N}_0^{(i_m)}$, and we have used the relationships between the $\coefe_{\mathbf{j}}$ and the $\coefc_{\mathbf{j}}$ coefficients in Appendix~\ref{sec:function-details}.

These considerations allow us to write all the $C^1$ basis functions restricted to the patch $\Omega^{(k)}$ as linear combinations of the (mapped) B-spline basis of $\US{p}{r} \otimes \US{p}{r}$.  That is, there exists a matrix $C_k$ such that $\phivec\vert_{\Omega^{(k)}} = C_k \mathbf{N}^{(k)}$, and using the notation in \eqref{eq:vector-all} the vector $\phivec\vert_{\Omega^{(k)}}$ collects all the non-vanishing basis functions on $\Omega^{(k)}$, i.e., the patch interior functions $\phivec_{\Omega^{(k)}}$, the edge functions $\phivec_{\Sigma^{(i)}}$ from the four edges on the boundary of $\Omega^{(k)}$, and the vertex functions $\phivec_{\mathbf{x}^{(i)}}$ from the four vertices on the boundary of $\Omega^{(k)}$. It is then possible, using standard techniques, to pass from the B-spline representation to the Bernstein polynomial representation, the so-called B\'ezier extraction.

Finally, since there is a one-to-one correspondence between ``extended'' patch interior functions and standard B-splines, and employing the latter representation, for a function~$\phi \in \UW$ we denote by $\coeff(\phi)$ the set~$(i,\mathbf{j})\in \indOmega \times \widetilde{\bf J}_\Omega $ such that the corresponding coefficient~$\coefet_{\mathbf{j}}^{(i)}$ in the B-spline representation is nonzero, i.e.
\[
{\RV \coeff(\phi) = \left\{(i,\mathbf{j}) \in \indOmega  \times  \widetilde{\bf J}_\Omega : \coefet_{\mathbf{j}}^{(i)} \neq 0 \right\}, \text{ with }\phi(\f{x}) =\sum_{i\in \indOmega}\sum_{\mathbf{j} \in \widetilde{\bf J}_\Omega} \coefet_{\mathbf{j}}^{(i)} \phiP{i}{\f{j}}(\f{x}). }
\]
We use a similar notation for a set of functions $\Psi \subset \UW$, namely
\begin{equation}\label{eq:coeff_set}
\coeff(\Psi) = \bigcup_{\psi \in \Psi} \coeff(\psi),
\end{equation}
the interesting case being when $\Psi \subset \PhiB$ is a subset of the basis.

%% file: s5_theoretical_one_level.tex
In the following we analyze some properties of the subspace $\UW$ that will be needed to apply the construction of the hierarchical space. 
\subsection{Characterization of the space $\UV$ and the subspace $\UW$}
The characterization of the subspace $\UW$, that we will use to prove nestedness in the hierarchical construction, was only given implicitly in previous works. We introduce it here explicitly for the sake of clarity.

The space~$\UV$ in \eqref{eq:spaceV} can be characterized as follows (see \cite{CoSaTa16,KaSaTa19a,KaSaTa19b}): A function~$\phi$ belongs to the space~$\UV$ if and only if for each patch~$\Omega^{(i)}$, $i \in \indOmega$, the functions $\phi \circ \f{F}^{(i)} $ satisfy that
\begin{equation} \label{eq:space}
 \phi \circ \f{F}^{(i)} \in \US{p}{r}\otimes\US{p}{r},
\end{equation}
and for each inner edge~$\Sigma^{(i)}$, $i \in \indSigmaI$, with the two corresponding neighboring patches~$\f{F}^{(i_0)}$ and $\f{F}^{(i_1)}$ possessing the same orientation as described in Section~\ref{sec:multipatch}, the functions $\phi \circ \f{F}^{(i_0)}$ and $\phi \circ \f{F}^{(i_1)}$ fulfill
\begin{equation} \label{eq:trace_init}
 \left( \phi \circ \f{F}^{(i_0)} \right) (0,\xi) = \left( \phi \circ \f{F}^{(i_1)} \right) (\xi,0), \mbox{ } \xi \in [0,1],
\end{equation}
and
\begin{equation*} 
\alpha^{(i,0)} (\xi)  \Dv \big(\phi \circ \f{F}^{(i_1)} \big)(\xi,0) +
        \alpha^{(i,1)}(\xi) \Du  \big( \phi \circ \f{F}^{(i_0)}\big)(0,\xi) 
+ \beta^{(i)} (\xi)
        \Dv  \big(\phi \circ \f{F}^{(i_0)}\big)(0,\xi)  =0
\end{equation*}
for $\xi \in [0,1]$. Due to~\eqref{eq:beta}, the previous equation 
is further equivalent to
\begin{equation} \label{eq:derivative_init2}
\begin{array}{c}
 \frac{1}{\alpha^{(i,0)}(\xi)}\left( \Du \left( \phi \circ \f F^{(i_{0})}\right)(0,\xi) + \beta^{(i,0)}(\xi) \, \Dv \left( \phi \circ \f F^{(i_{0})}\right)(0,\xi) \right)  =  \\
   -\frac{1}{\alpha^{(i,1)}(\xi)} \left( \Dv \left( \phi \circ \f F^{(i_{1})}\right)(\xi,0) + \beta^{(i,1)}(\xi) \, \Du \left(\phi \circ \f F^{(i_{1})}\right)(\xi,0) \right).
   \end{array}
\end{equation}
We denote for each inner edge~$\Sigma^{(i)}$, $i \in \indSigmaI$, the equally valued terms~\eqref{eq:trace_init} and \eqref{eq:derivative_init2} by the functions $f_0^{(i)}: [0,1] \rightarrow \R$ and $f_1^{(i)}:[0,1] \rightarrow \R$, respectively, which describe the trace and a specific directional derivative of the function~$\phi$ across the associated interface~$\Sigma^{(i)}$, see e.g. \cite{BrGiKaVa20,CoSaTa16,KaSaTa17a} for more details. Analogously, we define for each boundary edge~$\Sigma^{(i)}$, $i \in \indSigmaB$, with the associated patch~$\f{F}^{(i_0)}$ possessing the same orientation as described in Section~\ref{sec:multipatch}, the functions~$f^{(i)}_0$ and $f^{(i)}_1$ as the left-hand sides of \eqref{eq:trace_init} and \eqref{eq:derivative_init2}, respectively,
where $f^{(i)}_1$ can be simplified due to the selection of $\alpha^{(i,0)} \equiv 1$ and $\beta^{(i,0)} \equiv 0$ for boundary edges.

Equations \eqref{eq:space}--\eqref{eq:derivative_init2} fully characterize the space $\UV$. 
Similarly, we can give a characterization for the subspace $\UW$, that we state in the following proposition. The proof is a direct consequence of the construction of the basis functions, and is omitted.
\begin{proposition}\label{prop:characterization}
A function~$\phi$ belongs to $\UW$ if and only if it satisfies the conditions~\eqref{eq:space}, \eqref{eq:trace_init} and \eqref{eq:derivative_init2}, and moreover for each edge~$\Sigma^{(i)}$, $i \in \indSigma$, it holds that $f_0^{(i)} \in \US{p}{r+1} \mbox{ and } f_1^{(i)} \in \US{p-1}{r},$
respectively, and for each vertex~$\f{x}^{(i)}$, $i \in \indChi$, we have $\phi \in C^2(\f{x}^{(i)})$.
\end{proposition}

\subsection{Counterexample of local linear independence} \label{subsec:locallineardependence}

We will briefly demonstrate on the basis of an example, that the $C^1$ basis functions $\PhiB$, and in particular the {\RV vertex basis functions,}
can be locally linearly dependent. {\MK We consider the bilinearly parameterized three-patch domain~$\Omega$ shown in Fig.~\ref{fig:counterexample}~(left), which is hence trivially AS-$G^1$. The three patches are given in the standard form of Section~\ref{subsubsec:standard_paramaterization} with respect to the inner vertex~$\f{x}^{(0)}$, with $\Omega^{(0)}$ parameterized by}
\[
 \f{F}^{(0)}(\xi_1,\xi_2)= (2 (\sqrt{3}+1) \xi_1 + ( \sqrt{3}-3) \xi_2, -2 (\sqrt{3}-1) \xi_1 +   (3 \sqrt{3}+1) \xi_2),
\]
{\MK and with $\Omega^{(1)}$ and $\Omega^{(2)}$ obtained by rotating $\Omega^{(0)}$ by 120 and 240 degrees.}

\begin{figure}
\begin{center}
\begin{tabular}{cc}
  \includegraphics[trim=280mm 100mm 250mm 60mm,width=0.37\textwidth,clip]{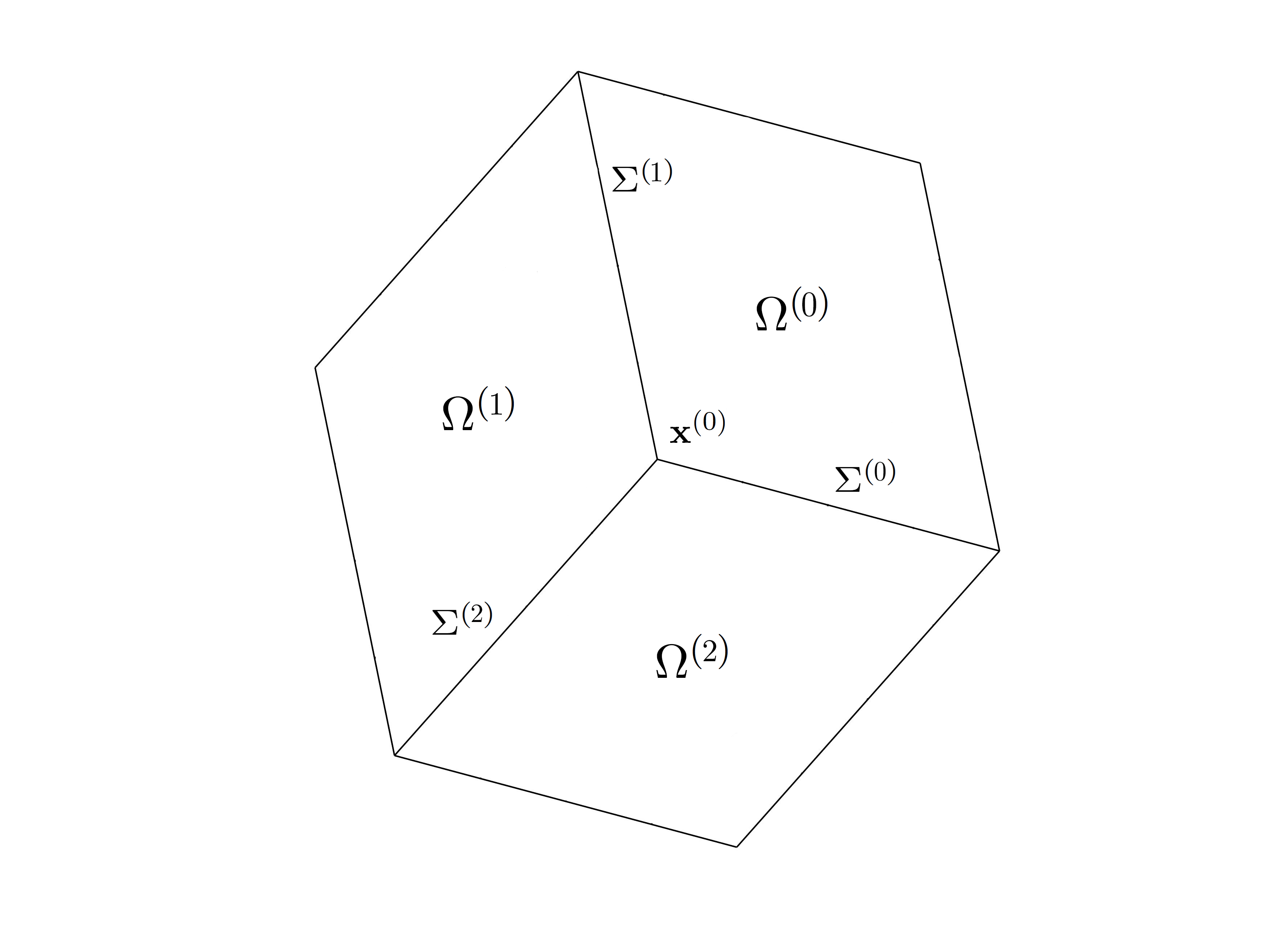} &
  \includegraphics[trim=280mm 100mm 250mm 60mm,width=0.37\textwidth,clip]{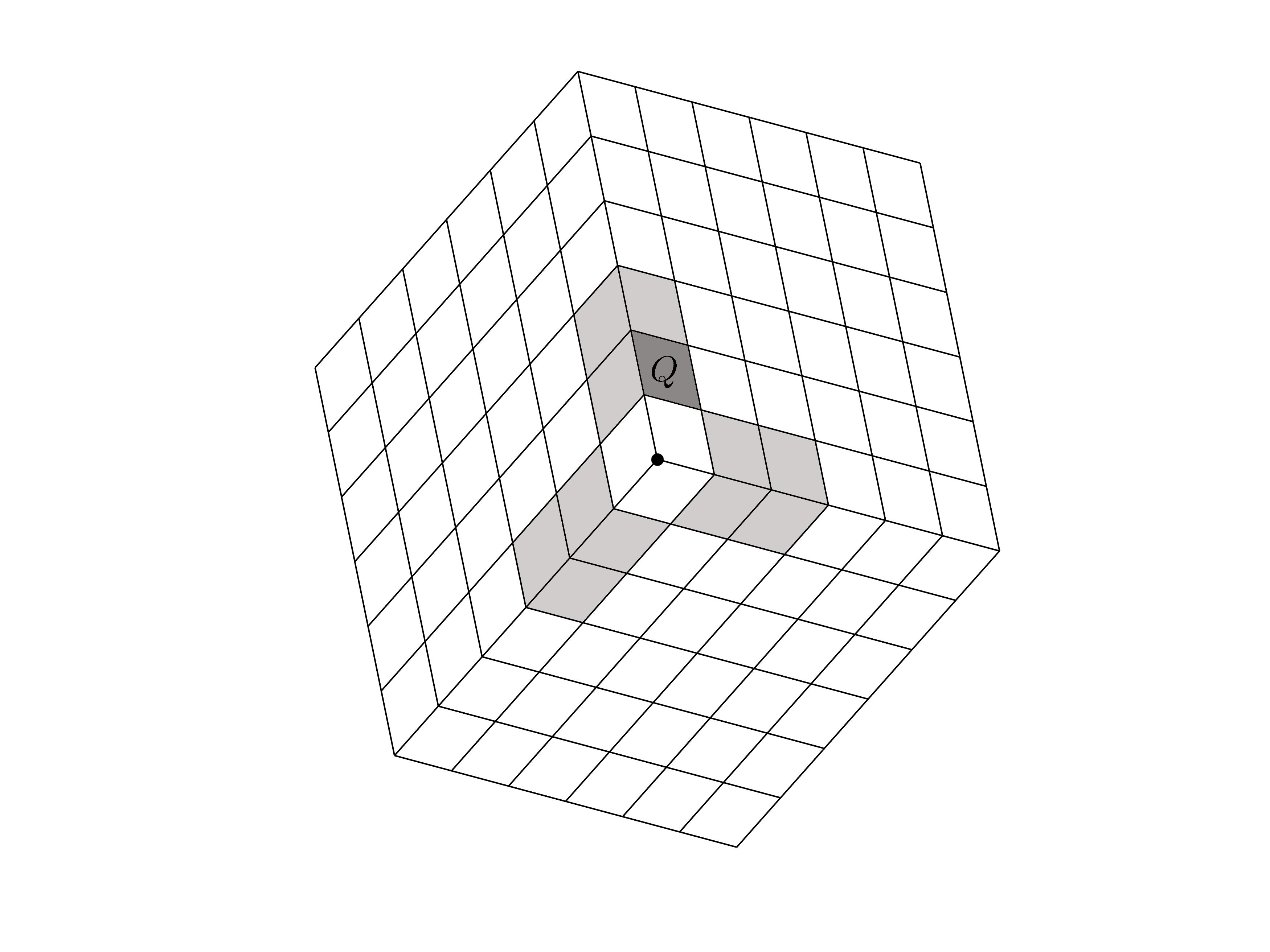} 
 \end{tabular}
 \caption{Left: A bilinearly parameterized three-patch domain~$\Omega$ whose geometry mappings~$\f{F}^{(i)}$, $i \in \{0,1,2 \}$, are given in standard configuration with respect to the inner vertex~$\f{x}^{(0)}$. Right: The associated set~$\GG$ of elements for $k=5$ with a particular element $Q \in \GG$ given by $Q=\f{F}^{(0)}((0,\frac{1}{6})\times (\frac{1}{6},\frac{1}{3}))$ and highlighted in {\RV dark} gray. {\RV The light gray elements are the other ones where the vertex basis functions for the inner vertex are locally linearly dependent.}}
 \label{fig:counterexample}
 \end{center}
\end{figure}

{\RV We compute the basis functions of the $C^1$ space for bicubic splines with $k=5$,} and consider the element~$Q \in \GG$ given by $Q=\f{F}^{(0)}((0,\frac{1}{6})\times (\frac{1}{6},\frac{1}{3}))$ and highlighted in Fig.~\ref{fig:counterexample}~(right). Restricting the basis functions to~$Q$, we can verify that $18$ basis functions are non-vanishing on this element. This directly implies that the basis functions of the space $\UW$ are locally linearly dependent, because the maximum number of linearly independent spline basis functions on one element is $(p+1)^2$, which means at most 16 functions {\RV for the bicubic case}.


In fact, the local linear dependence on this element is caused by the vertex basis functions. To study this in more detail, let us consider the set $\PhiB_{\f{x}^{(0)}} = \left\{ \phiV{0}{\f{j}} \ : \ \f{j} \in \JChi  \right\}$, and restrict the six vertex functions of this set to {\RV the element $Q$.} 
It is easy to verify that each of the six functions is non-vanishing on $Q$, and simplifies there to $\phiV{0}{\f{j}} = \sigma_0^{j_1+j_2} \left(g_{\mathbf{j}}^{(0,0,\rm{next})}  \circ \left(\f{F}^{(0)}\right)^{-1} \right)$. Since all the six functions $g_{\mathbf{j}}^{(0,0,\rm{next})}$, $\f{j} \in \JChi$, are just linear combinations of the five same functions, see \eqref{eq:func_g1}, the corresponding six functions~$\phiV{0}{\f{j}}$ are linearly dependent on $Q$. Thereby, it is interesting to note that the six functions $g_{\mathbf{j}}^{(0,0,\rm{next})}$ can be even represented on $Q$ just as linear combinations of three common functions, namely of the edge functions 
\[
 \fSigma_{(1,0)}^{(i_{1},0)} \circ \left(\f{F}^{(0)}\right)^{-1}, \mbox{ }\fSigma_{(2,0)}^{(i_{1},0)} \circ \left(\f{F}^{(0)}\right)^{-1} \mbox{ and } 
 \fSigma_{(1,1)}^{(i_{1},0)} 
 \circ \left(\f{F}^{(0)} \right)^{-1} .
\]
\begin{remark}
It is possible to ensure local linear independence of the $C^1$ spline {\RV basis} by assuming that the internal degree and regularity satisfy $r < p-3$, {\RV see Proposition~\ref{prop:low_regularity} below}. For instance, quintic functions with $C^1$ regularity ($p=5$, $r=1$) are locally linearly independent. However, the interesting case of the highest allowed regularity, $r = p-2$, is in general locally linearly dependent.
\end{remark}

\subsection{\RV Local and quasi-local linear independence results} \label{sec:lli_one_level}
{\RVV We will now study the (local) linear independence of particular subsets of the basis \eqref{eq:basis_A} of the $C^1$ isogeometric spline space~$\UW$ {\RV restricted to certain regions}. Since the basis functions are (mapped) piecewise polynomials on the mesh $\GG$ {\RV defined in \eqref{eq:mesh}}, {\MK it is sufficient} to {\MK prove}
the local linear independence relations below 
just for any element $Q \in \GG$, instead of for any open domain $\widetilde{\Omega} \subseteq \Omega$.
}

{\RV We start with an auxiliary lemma based on the definition of the set in \eqref{eq:coeff_set}. The proof is not shown, as it is an immediate consequence of the local linear independence of B-splines.}

\begin{lemma} \label{lem:twobasis}
 Let $\overline{\PhiB}, \widehat{\PhiB} \subseteq \Phi$. If both $\overline{\PhiB}$ and $ \widehat{\PhiB}$  are locally linearly independent, and if we further have
 $\coeff(\overline{\PhiB}) \cap \coeff(\widehat{\PhiB}) = \emptyset$,
 then the union of functions $\overline{\PhiB} \cup \widehat{\PhiB}$
 is locally linearly independent.
\end{lemma}

{\RV The following lemma generalizes a result for the two patch case from \cite{BrGiKaVa20}.}
\begin{lemma} \label{lem:patch_edge_functions}
The set of patch interior and edge basis functions $\PhiB_{\Omega} \cup \PhiB_{\Sigma}$
is locally linearly independent.
\end{lemma}
\begin{proof}
The local linear independence of the set of patch interior basis functions $\PhiB_\Omega$ {\RV is trivial, as they correspond to standard B-splines}.
For the set of edge basis functions $\PhiB_\Sigma$,
for each edge the set $\PhiB_{\Sigma^{(i)}}$, $i \in \indSigma$,  is locally linearly independent due to \cite[Proposition~3]{BrGiKaVa20}, where the result was proved for the extended set $\widetilde{\PhiB}_{\Sigma^{(i)}}$.
By construction of the edge basis functions, 
and in particular due to \eqref{repmultiedge}, we have that $ \coeff( \PhiB_{\Sigma^{(i_1)}}) \cap \coeff( \PhiB_{\Sigma^{(i_2)}}) = \emptyset $
for any $i_1,i_2 \in \indOmega$, with $i_1 \neq i_2$, which further implies together with Lemma~\ref{lem:twobasis} that the set of edge basis functions $\PhiB_\Sigma = \bigcup_{i \in \indSigma} \PhiB_{\Sigma^{(i)}}$ is locally linearly independent.

It remains to show that the union of patch interior and edge basis functions $\PhiB_\Omega \cup \PhiB_\Sigma$ 
is locally linearly independent. This is again a direct consequence of Lemma~\ref{lem:twobasis} using on the one hand the fact that each of the two sets 
is locally linearly independent, and on the other hand that, by construction of the individual basis functions, the two sets satisfy the condition $ \coeff\left( \PhiB_{\Omega} \right) \cap \coeff \left( \PhiB_{\Sigma} \right) = \emptyset$,
as can be clearly seen again from \eqref{repmultiedge}.
\end{proof}

For vertex basis functions local linear independence is not true in general, as we have seen in the counterexample of Section~\ref{subsec:locallineardependence}, but we can prove a partial result. To do so, we define for a vertex~$\f{x}^{(i)}$, $i \in \indChi$, the set of elements $Q \in \GG$ adjacent to the vertex~$\f{x}^{(i)}$, see Fig.~\ref{fig:adjacentels}, and denote it by
\[
\GG_{ \f{x}^{(i)}} := \{Q \in \GG \, : \, \f{x}^{(i)} \in \partial Q \}.
\]
{\RVV The next two lemmas show the relations of the different basis functions in this set of elements.}
\begin{figure}
\centerline{\includegraphics[trim=0cm 10cm 0cm 7cm, clip, width=0.55\textwidth]{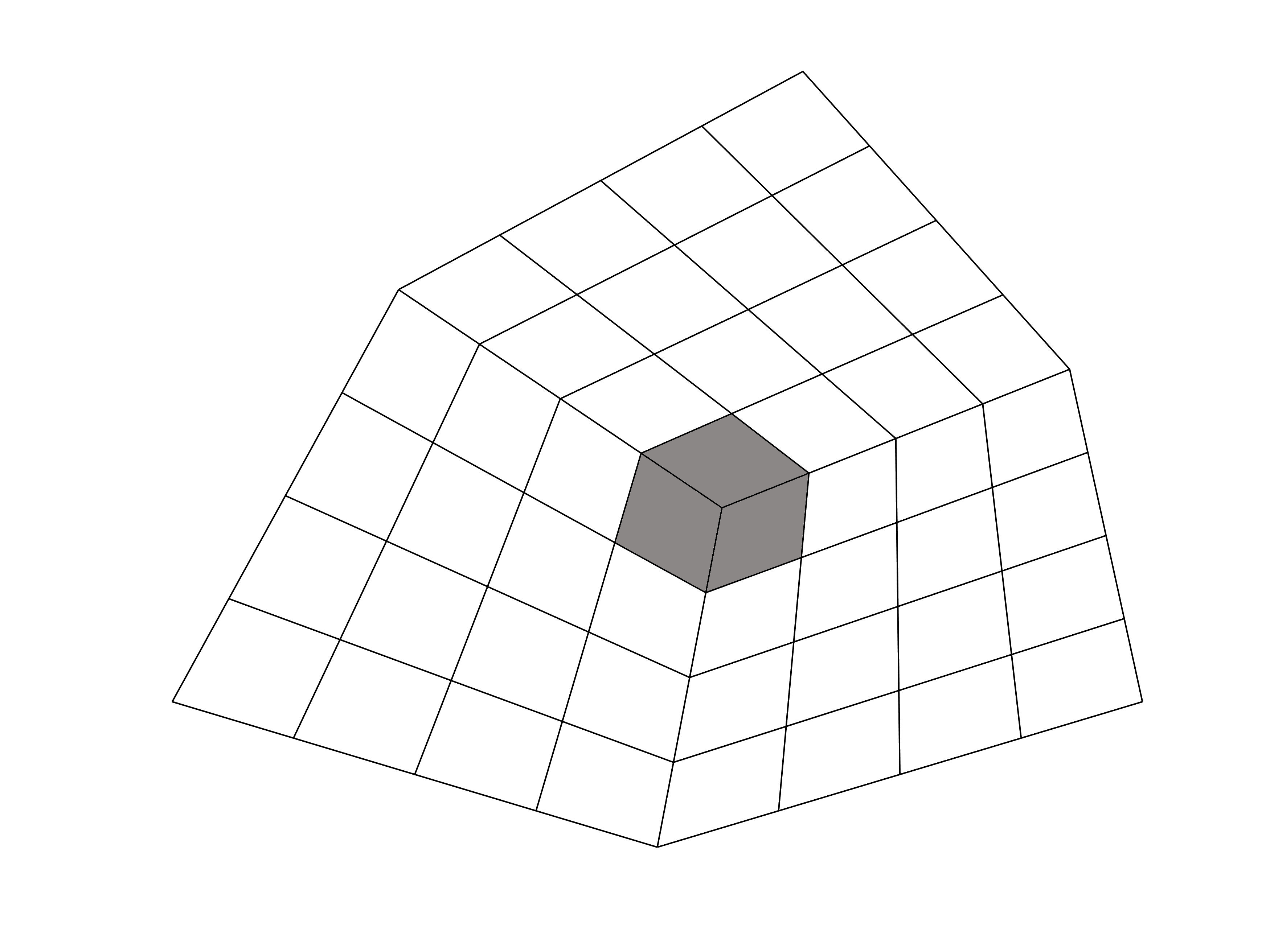}}
\caption{Example of the set of elements $\GG_{ \f{x}^{(i)}}$, highlighted in gray, adjacent to a vertex.}
\label{fig:adjacentels}
\end{figure}
\begin{lemma} \label{lem:vertex}
For any $i \in \indChi$ and for every element $Q \in \GG_{\f{x}^{(i)}}$, the vertex functions $\PhiB_{\f{x}^{(i)}}$ are linearly independent in $Q$.
\end{lemma}
\begin{proof}
Let $\f{x}^{(i)}$, $i \in \indChi$, be an arbitrary vertex, and let $Q \in \GG_{\f{x}^{(i)}}$ be an arbitrary element adjacent to the vertex~$\f{x}^{(i)}$. By definition, the six vertex basis functions~$\phiV{i}{\f{j}}$, $\f{j} \in \JChi$, do not vanish on $Q$, and {\RV they satisfy} the $C^2$ interpolation condition~\eqref{eq:C2interpolation} at the vertex~$\f{x}^{(i)}$, which yields their linear independence in $Q$.
\end{proof}

\begin{lemma} \label{lem:vertex2}
For any $m \in \indChi$, let us define
\[
\Psi_{\f{x}^{(m)}} = \PhiB \setminus \PhiB_{\f{x}^{(m)}} = \PhiB_{\Omega} \cup 
\PhiB_{\Sigma} \cup 
\left( \PhiB_{\chi} \setminus \PhiB_{\f{x}^{(m)}} \right).
\]
Then, for the set of vertex functions $\PhiB_{\f{x}^{(m)}}$ and for every element $Q \in \GG_{\f{x}^{(m)}}$, it holds that $\myspan(\PhiB_{\f{x}^{(m)}}|_Q) \cap \myspan(\Psi_{\f{x}^{(m)}}|_Q) = 0$.
\end{lemma}
\begin{proof}
Let the vertex $\f{x}^{(m)}$, $m \in \indChi$, and the element $Q \in \GG_{\f{x}^{(m)}}$. Due to Lemma~\ref{lem:vertex}, the six vertex basis functions in $\PhiB_{\f{x}^{(m)}}$ 
are linearly independent in $Q$, which was a direct result of their $C^2$~interpolation condition~\eqref{eq:C2interpolation} at the vertex~$\f{x}^{(m)}$.
Instead all other functions $\phi \in \Psi_{\f{x}^{(m)}}$ satisfy by construction
\begin{equation*}
\Du^{z_1} \Dv^{z_2} \phi  (\f{x}^{(m)}) = 
 0, \mbox{ } 0 \leq z_1,z_2 \leq 2, \mbox{ }z_1+z_2 \leq 2,
\end{equation*} 
which directly implies that the intersection of the spaces spanned by the two sets, restricted to $Q$, only contains the zero function.
\end{proof}

{\RV Finally, we present a result of local linear independence in case of low regularity.
\begin{proposition} \label{prop:low_regularity}
If the degree and regularity satisfy $r < p-3$, the basis $\PhiB$ is locally linearly independent.
\end{proposition}
\begin{proof}
It is sufficient to note that, for regularity $r < p-3$, the vertex basis functions are supported on the elements adjacent to the vertex. The proof then immediately follows from the previous three lemmas.
\end{proof}
}

%% file: s6_hierarchical_C1.tex
We now introduce the construction of the hierarchical $C^1$ spline space. We start by introducing the $C^1$ multi-patch spaces for each level, and then we analyze their properties to apply the construction of Section~\ref{sec:hierarchical}. In particular, we have to check the nestedness of the spaces, and that the assumptions of Theorem~\ref{thm:lin-indep-H} are satisfied. As we will see, this will impose some constraints to modify the refinement algorithm.

\subsection{\mbox{$C^1$ multi-patch spaces on each level and hierarchical construction}}
Let us assume that we have a multi-patch domain $\Omega$ with an analysis-suitable $G^1$ parameterization $\mathbf{F}$ as in Section~\ref{sec:multipatch-geometry}, and the $C^1$ space $\UV^0$, with the corresponding subspace $\UW^0$ as in Section~\ref{sec:C1space}. By successively applying dyadic refinement, we construct a sequence of spaces $\UV^\ell$ and their corresponding subspaces $\UW^\ell$, for $\ell = 0, \ldots, N-1$. The associated meshes are denoted by $\GG^\ell$. To apply the construction of hierarchical splines from Section~\ref{sec:hierarchical}, the subspaces $\UW^\ell$ and their bases $\PhiB^\ell$ respectively play the role of $\mathbb{U}^\ell$ and $\Psi^\ell$ in the construction of that section. We further assume that each subdomain $\Omega^\ell$ is the union of elements of the mesh $\GG^{\ell-1}$. We will denote by $\QQ$ the hierarchical mesh, and by $\mathcal{H}_\UW$ the set of hierarchical $C^1$ splines, that we will prove to be a basis.

Note that the parameterizations, the geometric entities and the gluing data in Section~\ref{sec:multipatch-geometry} are independent of the level $\ell$. Instead, the discrete spaces and their corresponding bases and basis functions clearly depend on $\ell$, and we will use the $\ell$ superindex to refer to them. For instance, the basis will be denoted by $\PhiB^\ell$, and the univariate spline spaces will be denoted by $\US{p}{r,\ell}$. Moreover, the set of elements from each level adjacent to a vertex will be denoted by $\GG^\ell_{\mathbf{x}^{(i)}}$.


In the following, we analyze the properties of the subspaces $\UW^\ell$ to apply the construction of the hierarchical space.

\subsection{Nestedness and refinement mask} \label{sec:nestedness}
The first property we need to prove is the nestedness of the subspaces, that is, that $\UW^\ell \subset \UW^{\ell+1}$ for $\ell = 0, \ldots, N-2$. Nestedness is clear for the spaces $\UV^\ell$, while for the subspaces $\UW^\ell$ it relies on the characterization from Proposition~\ref{prop:characterization}.
\begin{proposition}
Let $N \in \mathbb{N}$. The sequence of spaces $\UW^\ell, \ell = 0,1, \ldots, N-1$ is nested, i.e., $\UW^0 \subset \UW^1 \subset \ldots \subset \UW^{N-1}$.
\end{proposition}
\begin{proof}
The result is an immediate consequence of Proposition~\ref{prop:characterization} and the nestedness of the univariate spline spaces, $\US{p-1}{r,\ell} \subset \US{p-1}{r,\ell+1}$ and $\US{p}{r+1,\ell} \subset \US{p}{r+1,\ell+1}$, for $\ell = 0, \ldots, N-2$.
\end{proof}

Thanks to the nestedness of the subspaces $\UW^\ell$, we can define the {\RV set of truncated hierarchical splines} as described in Section~\ref{sec:hierarchical}, and we will denote it by ${\cal T}_\UW$. The explicit definition of the functions {\RV in ${\cal T}_\UW$} requires to use the coefficients of the two level relations between the $C^1$ basis functions of consecutive levels, also called the refinement mask, that we describe in the following.

Let us denote with an upper index $\ell$ the vectors of standard isogeometric functions \eqref{stdbspl-new} of level $\ell$, {\MK recall also Fig.~\ref{fig:subvectors}}. Then, we have the relation between functions of two consecutive levels
\begin{equation} \label{eq:refmask-bsp}
\begin{bmatrix}
{\bf N}_{0}^{(k),\ell} \\
{\bf N}_{1}^{(k),\ell} \\
{\bf N}_{2}^{(k),\ell} \\
{\bf N}_{3}^{(k),\ell}
\end{bmatrix}
=
\begin{bmatrix}
\Theta^{\ell+1}_{00} & \Theta^{\ell+1}_{01} & \Theta^{\ell+1}_{02} & \Theta^{\ell+1}_{03} \\
0 & \Theta^{\ell+1}_{11} & 0 & \Theta^{\ell+1}_{13} \\
0 & 0 & \Theta^{\ell+1}_{22} & \Theta^{\ell+1}_{23} \\
0 & 0 & 0 & \Theta^{\ell+1}_{33} \\
\end{bmatrix}
\begin{bmatrix}
{\bf N}_{0}^{(k),\ell+1} \\
{\bf N}_{1}^{(k),\ell+1} \\
{\bf N}_{2}^{(k),\ell+1} \\
{\bf N}_{3}^{(k),\ell+1}
\end{bmatrix}
, \text{ for } k \in \indOmega.
\end{equation}
The refinement mask for the patch interior functions, which coincide with $\mathbf{N}^{(k),\ell}_3$, is simply a restriction of the relation \eqref{eq:refmask-bsp} to their corresponding indices.

We recall that the edge functions, and ``extended'' edge functions, of level $\ell$ associated with an edge $\Sigma^{(i)}$, for $i \in \indSigma$, can be expressed in terms of standard isogeometric functions through the matrix $\widetilde{E}^\ell_{i,k}$, as given by \eqref{reptwoedge}. Let us introduce the block diagonal matrix
\[
\widetilde \Lambda^{\ell+1} = 
\begin{bmatrix}
\widetilde\Lambda_p^{r+1,\ell+1} & 0 \\
0 & \frac{1}{2}\widetilde\Lambda_{p-1}^{r,\ell+1}
\end{bmatrix},
\]
where $\widetilde\Lambda^{s,\ell+1}_q$ stands for the refinement matrix for univariate B-splines of level $\ell$ of degree $q$ and regularity $s$. By generalizing the results in \cite{BrGiKaVa20} for the two-patch case, and noting that functions in the subvectors $\mathbf{N}_3$ correspond to patch interior functions, we obtain the following refinement relation for the edge functions: 
\begin{align*} 
\phivec_{\Sigma^{(i)}}^\ell & = \Lambda^{\ell+1} \phivec_{\Sigma^{(i)}}^{\ell+1} + 
E_{i,0}^\ell \Theta_{23}^{\ell+1} \mathbf{N}_3^{(i_0),\ell+1} + 
E_{i,1}^\ell \Theta_{13}^{\ell+1} \mathbf{N}_3^{(i_1),\ell+1} \\
& = \Lambda^{\ell+1} \phivec_{\Sigma^{(i)}}^{\ell+1} + 
E_{i,0}^\ell \Theta_{23}^{\ell+1} \phivec_{\Omega^{(i_0)}}^{\ell+1} + 
E_{i,1}^\ell \Theta_{13}^{\ell+1} \phivec_{\Omega^{(i_1)}}^{\ell+1},
\end{align*}
where $\Lambda^{\ell+1}$ is the restriction of $\widetilde \Lambda^{\ell+1}$ to rows and columns corresponding to functions away from the vertices, and the matrices $E_{i,0}^\ell$ and $E_{i,1}^\ell$ are the matrices in \eqref{repmultiedge} for basis functions of level $\ell$.

For {\MK the} vertex basis functions, {\RV and recalling that the indices ${\bf j}=(j_1,j_2) \in \JChi$ are sorted moving first on the first index, as already explained in Section~\ref{subsec:standard_representation},} let us first introduce the diagonal matrix
\[
D_{{\bf x}^{(i)}}^{\ell+1} = \mathrm{diag}\left(\left[1, \frac{{\sigma_i^\ell}}{{\sigma_i^{\ell+1}}}, \left(\frac{{\sigma_i^\ell}}{{\sigma_i^{\ell+1}}}\right)^{2}, \frac{{\sigma_i^\ell}}{{\sigma_i^{\ell+1}}}, \left(\frac{{\sigma_i^\ell}}{{\sigma_i^{\ell+1}}}\right)^{2}, \left(\frac{{\sigma_i^\ell}}{{\sigma_i^{\ell+1}}}\right)^{2}\right]\right).
\]
Exploiting the expression for vertex functions in terms of standard isogeometric functions~\eqref{repvert} and the refinement mask for extended edge functions from \cite{BrGiKaVa20} and for standard mapped B-splines in \eqref{eq:refmask-bsp}, we obtain that for each vertex function of index ${\bf j} \in \JChi$ associated with the vertex ${\bf x}^{(i)}$, $i \in \indChi$, we have
\begin{align*}
&\phivec_{{\bf x}^{(i)}}^{\ell}= D_{{\bf x}^{(i)}}^{\ell+1} \phivec_{{\bf x}^{(i)}}^{\ell+1} 
+ \sum_{m=0}^{\nu_i-1} K_{i,{m}}^\ell \widehat{\Lambda}^{\ell+1} \phivec_{\Sigma^{(i_m)}}^{\ell+1} 
+ \delta_{b} K_{i,{\nu_i+1}}^\ell \widehat{\Lambda}^{\ell+1} \phivec_{\Sigma^{(i_{\nu_i+1})}}^{\ell+1} \\
&+\sum_{m=0}^{\nu_i-1}
\left(K_{i,{m}}^\ell \widehat{E}_{i_{m},1}^\ell 
\begin{bmatrix}
\Theta_{03}^{\ell+1} \\
\Theta_{13}^{\ell+1}
\end{bmatrix}
+K_{i,{m+1}}^\ell \widehat{E}_{i_{m+1},0}^\ell 
\begin{bmatrix}
\Theta_{03}^{\ell+1} \\
\Theta_{23}^{\ell+1}
\end{bmatrix}
- V_{i,m}^\ell \Theta_{03}^{\ell+1}
\right) \phivec_{\Omega^{(i_m)}}^{\ell+1},
\end{align*}
where $\delta_b$ indicates whether $\mathbf{x}^{(i)}$ is an interior or a boundary vertex, i.e., 
\begin{equation*}
\delta_{b}=\begin{cases}{} 0 & \hbox{if}\, i \in \indChiI, \; \\
1 & \hbox{if}\, i \in \indChiB, \; 
\end{cases}
\end{equation*}
$\widehat{\Lambda}^{\ell+1}$ is the restriction of $\widetilde{\Lambda}^{\ell+1}$ to rows corresponding to the five ``extended'' edge functions close to the vertex and to columns of active edge functions, while all the other matrices have been introduced above. Note that the $\Theta$ matrices appearing in the previous equation are very sparse, and in practice one can restrict the computations to the few coefficients that are nonzero.

\subsection{Condition for linear independence of hierarchical $C^1$ {\RV splines}}
The second property we need to prove is (P1), that guarantees linear independence of the hierarchical {\RV splines, and therefore that they form a basis}. The proof is based on the linear independence results from Section~\ref{sec:lli_one_level}.

\begin{theorem}\label{thm:linear_independence}
Let the spaces $\{\UW^\ell\}_{\ell=0}^{N-1}$, with bases $\PhiB^\ell$, obtained by dyadic refinement, the hierarchical mesh $\mathcal{Q}$, and let the hierarchical $C^1$ spline set $\mathcal{H}_\UW$ be defined as in \eqref{eq:hbasis}, and $\mathcal{T}_\UW$ be defined as in \eqref{eq:thbasis}. If for every active vertex function $\phi \in \PhiB^\ell_{\f{x}^{(i)}} \cap \mathcal{H}_\UW$ there exists an active element {\RVV in $\GG^\ell_{\f{x}^{(i)}} \cap \mathcal{Q}$}, then both the functions in $\mathcal{H}_\UW$ and the truncated functions in $\mathcal{T}_\UW$ are linearly independent.
\end{theorem}
\begin{proof}
{\RVV
We have to prove property (P1), i.e., that for every level $\ell$ the functions in $\Phi^\ell|_{\Omega^\ell \setminus \Omega^{\ell+1}}$ are linearly independent. {\RVV Let us denote $D^\ell = \Omega^\ell \setminus \Omega^{\ell+1}$, to alleviate notation.} Using the decomposition \eqref{eq:basis_A}, this is equivalent to prove that 
\[
\sum_{i \in \indOmega} \sum_{\phi \in \PhiB^\ell_{\Omega^{(i)}}|_{D^\ell}} c_\phi \phi + 
\sum_{i \in \indSigma} \sum_{\phi \in \PhiB^\ell_{\Sigma^{(i)}}|_{D^\ell}} c_\phi \phi + 
\sum_{i \in \indChi} \sum_{\phi \in \PhiB^\ell_{\mathbf{x}^{(i)}}|_{D^\ell}} c_\phi \phi = 0 
\]
implies that all the coefficients $c_\phi$ are equal to zero.

By hypothesis, for any active vertex function 
{\RVV $\phi \in \PhiB^\ell_{\mathbf{x}^{(i)}}|_{D^\ell}$}, there exists an active element $Q \in \GG^\ell_{\f{x}^{(i)}} \cap \mathcal{Q}$ and which is obviously contained in 
{\RVV $D^\ell$.} By taking the restriction to $Q$, Lemmas~\ref{lem:vertex} and~\ref{lem:vertex2} imply that the corresponding coefficient $c_\phi$ is equal to zero. Since the argument is valid for any vertex, all the coefficients in the third term of the sum must be zero. Finally, the coefficients for the first and second term are zero by Lemma~\ref{lem:patch_edge_functions}, and the result follows.
}
\end{proof}


\subsection{Refinement algorithm}
Theorem~\ref{thm:linear_independence} gives us the only requirement for the linear independence of the hierarchical $C^1$ spline functions: {\RV the support of an active vertex function of level $\ell$ must contain an active element of the same level and adjacent to the vertex.}
We now present a refinement algorithm that guarantees that this property is always satisfied.

Let us first {\RV denote the set of elements in the hierarchical mesh adjacent to any vertex as $\GG_{\chi}$.}
Then, for each element adjacent to a vertex, $Q \in \GG_\chi \cap \GG^\ell_{\f{x}^{(i)}}$ for some $i \in \indChi$ and $\ell \in \{0, \ldots, N-1\}$, and such that $Q \subset \Omega^{(i_m)}$ for some $m = 0, \ldots \nu_i-1$, we define the vertex-patch neighborhood
\[
\mathcal{N}_{\chi}(Q) = \{ Q' \in \GG^\ell \cap \QQ : Q' \subset \Omega^{(i_m)} \cap \supp \phi, \text{ for } \phi \in \PhiB_{\f{x}^{(i)}}^\ell \} \setminus Q,
\]
formed by the elements of level $\ell$ contained in {\RV the patch} $\Omega^{(i_m)}$ and in the support of vertex functions of level $\ell$ associated to ${\f{x}^{(i)}}$.
The marking algorithm proposed in Algorithm~\ref{alg:hrefine-nonadmissible} enforces that whenever an element adjacent to a vertex is marked for refinement, the elements in its vertex-patch neighborhood are also marked, see also Fig.~\ref{fig:algo1}.

\begin{algorithm}[!ht]
\caption{\texttt{MARK\_VERTEX-PATCH} ($\QQ, \mathcal{M}$)}
\label{alg:hrefine-nonadmissible}
\begin{algorithmic}
\State \textbf{Input: }hierarchical mesh $\QQ$, marked elements $\mathcal{M} \subseteq \QQ$
\State set $\displaystyle \mathcal{V} = \bigcup_{Q \in \mathcal{M}\cap \GG_\chi} \mathcal{N}_\chi(Q) \setminus{\mathcal{M}}$
\State set $\mathcal{M} = \mathcal{M}\cup \mathcal{V}$
\State \textbf{Output: } updated set of marked elements $\mathcal{M}$
\end{algorithmic}
\end{algorithm}

\begin{figure}[!t]
\begin{subfigure}[Mark element $Q$ adjacent to a vertex.]{
\includegraphics[trim=0cm 5cm 0cm 3cm, clip, width=0.46\textwidth]{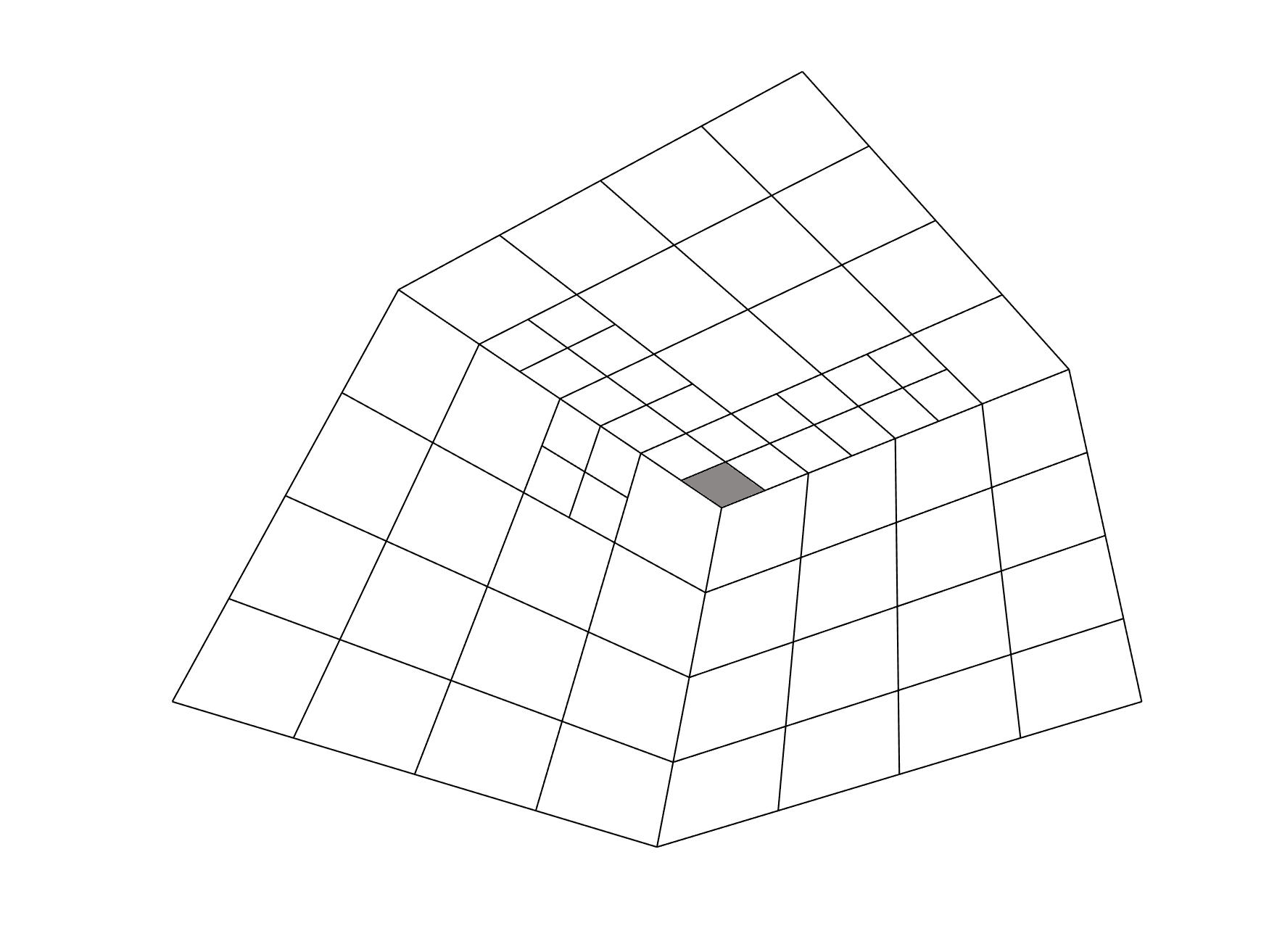}
}
\end{subfigure}
\begin{subfigure}[Refine elements in $\mathcal{N}_{\chi}(Q)$.]{
\includegraphics[trim=0cm 5cm 0cm 3cm, clip, width=0.46\textwidth]{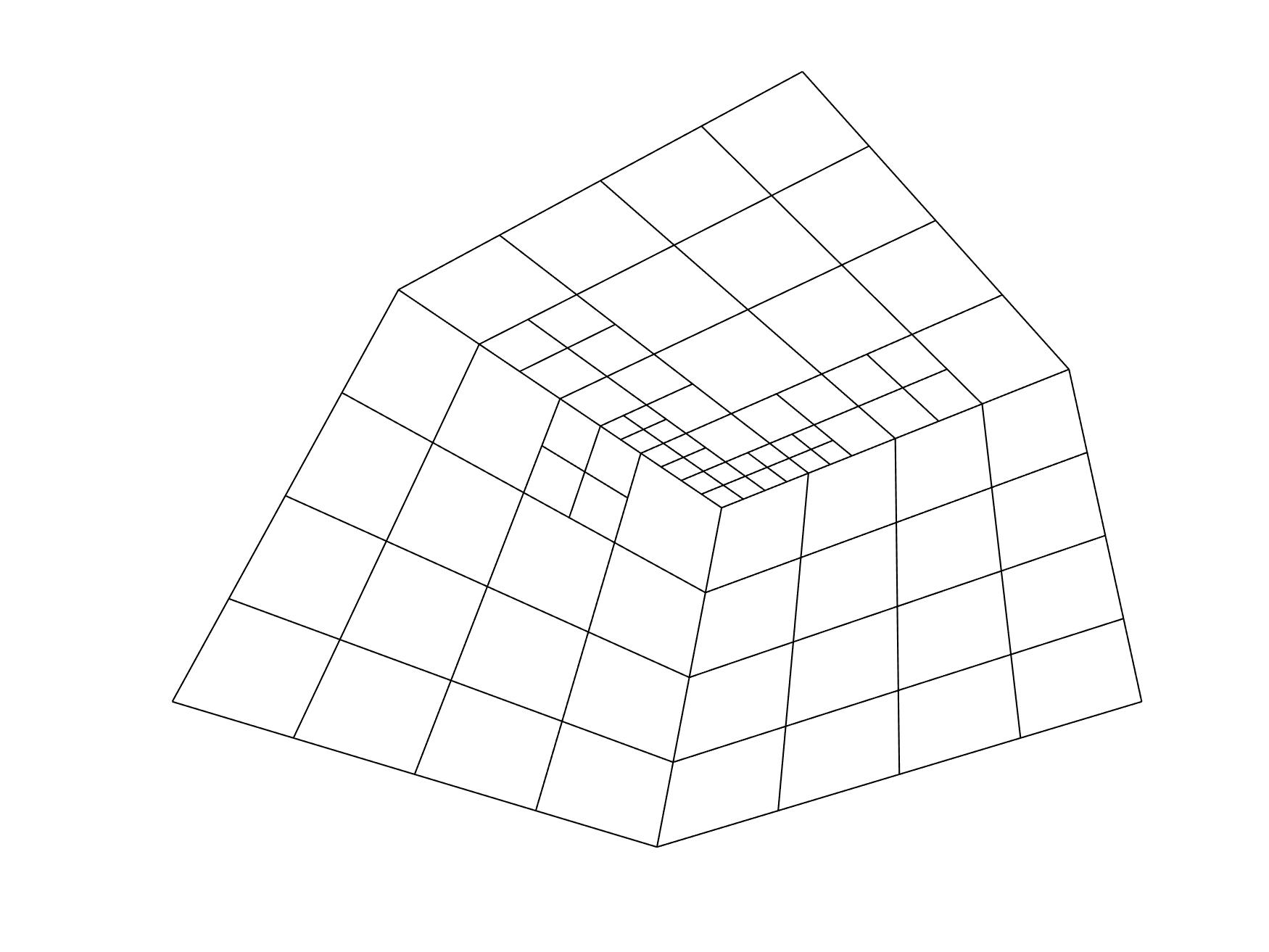}
}
\end{subfigure}
\begin{subfigure}[Mark element $Q$ not adjacent to any vertex.]{
\includegraphics[trim=0cm 5cm 0cm 3cm, clip, width=0.47\textwidth]{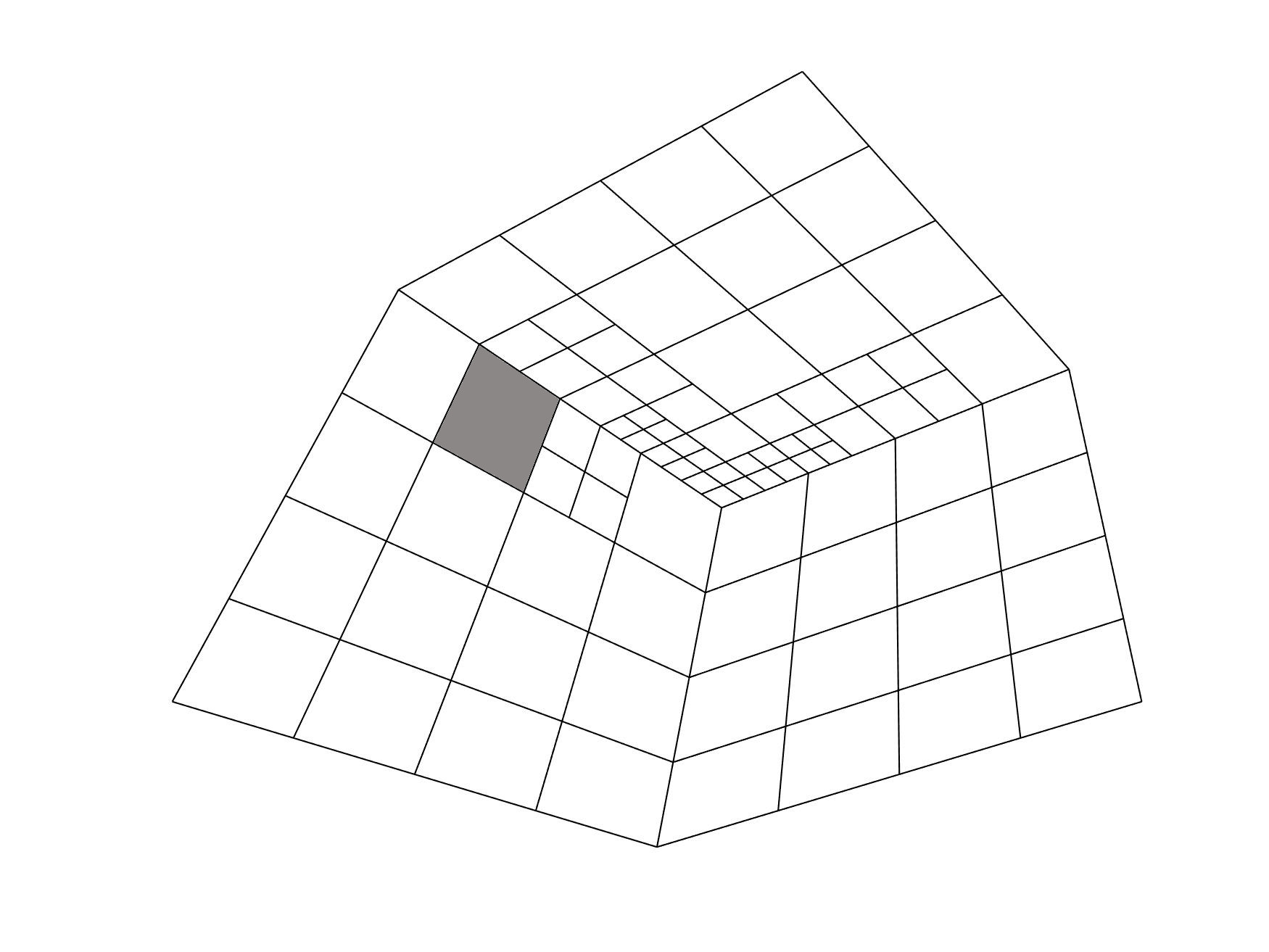}
}
\end{subfigure}
\begin{subfigure}[Refine $Q$ without additional refinement.]{
\includegraphics[trim=0cm 5cm 0cm 3cm, clip, width=0.47\textwidth]{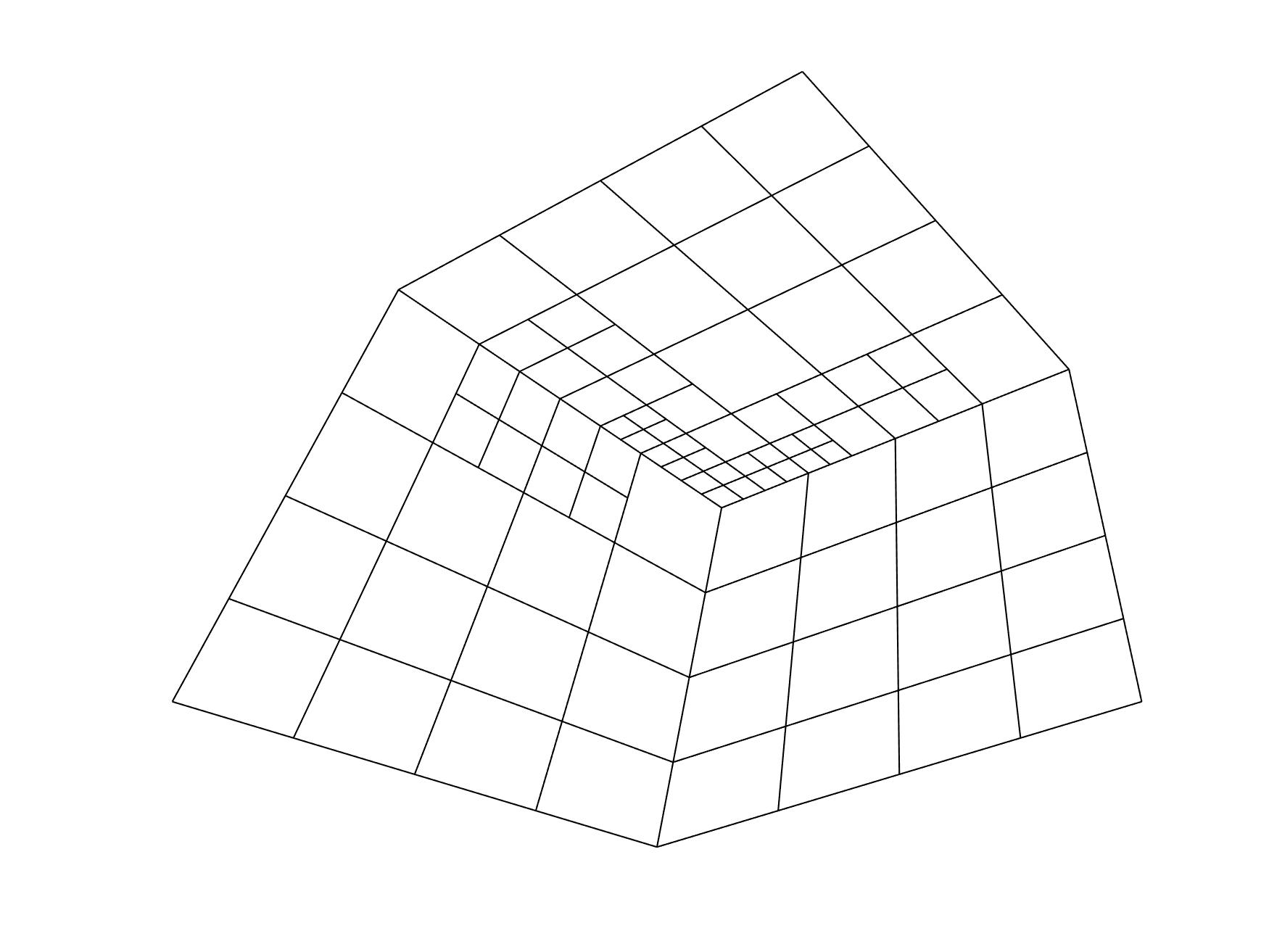}
}
\end{subfigure}
\caption{When one element $Q$ adjacent to a vertex is marked, highlighted in dark gray in (a), all the elements in $\mathcal{N}_{\chi}(Q)$ are also refined (b). If the marked element $Q$ is not adjacent to any vertex (c), no other elements need to be refined (d), even if $Q$ belongs to the support of vertex basis functions.}
\label{fig:algo1}
\end{figure}

\begin{remark}
If the initial mesh is very coarse one should also check whether marking the vertex-patch $\mathcal{N}_\chi(Q)$ marks any element adjacent to another vertex. This would force to mark also the vertex-patch neighborhood of that element, finally propagating the marking to all the boundary elements of the patch. To avoid this check, and to simplify the algorithm, we assume that the coarsest mesh is not coarser than $4 \times 4$ elements per patch, that is $k \ge 3$ inner knots, which prevents this situation to happen.
\end{remark}

The modification in the algorithm to mark the vertex-patch neighborhood can be easily combined with admissible refinement as introduced in \cite{buffa2016c}, allowing to construct admissible hierarchical meshes such that the constraint on Theorem~\ref{thm:linear_independence} is also satisfied. We recall that a mesh is admissible of class $\mu > 1$, if for any element of the hierarchical mesh the non-vanishing functions belong to at most $\mu$ different levels \cite{buffa2016c}. {\RV We will use the terms ${\cal H}$-admissibility and ${\cal T}$-admissibility depending on whether we look for admissibility with hierarchical splines or with truncated hierarchical splines, and} we respectively define, for an element $Q \in \QQ \cap G^\ell$ of level $\ell \ge 0$, its $\mathcal{H}$-neighborhood and $\mathcal{T}$-neighborhood as
\begin{align*}
\mathcal{N}_{\mathcal{H}}(Q, \mu) = \{ Q' \in \QQ \cap G^{\ell-\mu+1} : Q' \cap S_{\mathrm{ext}}(Q,\ell-\mu+1) \not = \emptyset \}, \\
\mathcal{N}_{\mathcal{T}}(Q, \mu) = \{ Q' \in \QQ \cap G^{\ell-\mu+1} : Q' \cap S_{\mathrm{ext}}(Q,\ell-\mu+2) \not = \emptyset \},
\end{align*}
for $\ell - \mu + 1 >0$, and $\mathcal{N}_{\mathcal{H}}(Q, \mu) = \mathcal{N}_{\mathcal{T}}(Q, \mu) = \emptyset$ if $\ell - \mu + 1 < 0$, see \cite[Section~4.1]{reviewadaptiveiga}. In the definition, the multi-level support extension $S_{\mathrm{ext}}(Q, k)$ is defined as the union of the supports of all basis functions of level $k$ that do not vanish on $Q$, namely
\[
S_{\mathrm{ext}}(Q, k) = \bigcup \{ \supp \phi : \phi \in \Phi^k \wedge Q \cap \supp \phi \not = \emptyset\}.
\]
Using a generic notation $\mathcal{N}(Q, \mu)$ for both neighborhoods, and with the convention that $\mathcal{N}(Q, 0) = \emptyset$ for non-admissible meshes, the refinement algorithm in Algorithm~\ref{alg:hrefine-admissible} guarantees both the linear independence of the {\RV $C^1$ hierarchical splines} and the admissibility of the adaptive mesh.

\begin{algorithm}[!ht]
\caption{\texttt{REFINE}$(\QQ,\mathcal{M},\mu)$, refine guaranteeing linear independence and admissibility}
\label{alg:hrefine-admissible}
\begin{algorithmic}
\State \textbf{Input: } admissible hierarchical mesh $\QQ$, marked elements $\mathcal{M} \subseteq \QQ$, admissible class $\mu$
\Repeat
\State set $\mathcal{M} = $ \texttt{MARK\_VERTEX-PATCH} $(\QQ, \mathcal{M})$
\State set $\displaystyle \mathcal{U} = \bigcup_{Q \in \mathcal{M}} \mathcal{N}(Q,\mu) \setminus \mathcal{M}$
\State set $\mathcal{M} = \mathcal{M}\cup \mathcal{U}$
\Until $\mathcal{U} = \emptyset$
\State update $\QQ$ by replacing the elements in $\mathcal{M}$ by their children
\State \textbf{Output: } refined admissible mesh $\QQ$
\end{algorithmic}
\end{algorithm}

For any marked element the recursive refinement of its neighborhood guarantees the admissibility of the refined mesh as in \cite{buffa2016c}. If the marked element is adjacent to a vertex the marking of any other element in its vertex-patch neighborhood guarantees to satisfy the hypotheses of Theorem~\ref{thm:linear_independence} and, consequently, the linear independence of the (truncated) hierarchical functions.

\begin{theorem} \label{thm:algorithm}
Assuming that we start from a coarse mesh $\GG^0$ with at least $4 \times 4$ elements per patch, and applying successive refinement with Algorithm~\ref{alg:hrefine-admissible}, the construction of hierarchical $C^1$ splines gives a set $\mathcal{H}_\UW$ or $\mathcal{T}_\UW$ of linearly independent functions. Moreover, if $\mu > 1$ the mesh constructed in Algorithm~\ref{alg:hrefine-admissible} is admissible.
\end{theorem}

\subsection{Linear complexity of the refinement algorithm}
We now prove a complexity estimate for the refinement algorithm, in the spirit of \cite{bdd04,stevenson07}, {\RV extending the analysis presented in \cite{bgmp16} for the single-patch case to the multi-patch case}. {\RVV The resulting complexity estimate depends on the valence of the vertices, but not on the particular parameterization of the geometry.}



{\RV We start defining a distance between elements of the mesh, and analyzing its properties. To do so, for any element $Q \in G^\ell$ we set $\Pi^0(Q) = \overline{Q}$, from which we recursively define the regions of elements around $Q$ in the mesh of level $\ell$ as
\begin{align*}
& \Pi^{s+1}(Q) = \bigcup \{ \overline{Q'} : Q' \in G^\ell \, , \, \overline{Q'} \cap \Pi^s(Q) \not = \emptyset \}, \text{ for } s = 0, 1, \ldots
\end{align*}
We note that the number of elements of level $\ell$ contained in $\Pi^{s}(Q)$ is bounded, with a bound that depends on $s$ and on $\nu = \max_{i \in \indChi} \nu_i$, the maximum valence of the vertices of the geometry. 
The distance is first defined for two elements of the same level, $Q, Q' \in G^\ell$, as
\[
\dist(Q, Q') = 2^{-\ell} s, \text{ with } s = \min\{r : Q' \subset \Pi^r(Q) \},
\]
and it is worth to note that this characterizes $\Pi^s(Q)$ as the region occupied by elements $Q' \in G^\ell$ such that $\dist(Q,Q') \le 2^{-\ell} s$. Then, for elements of different levels $Q \in G^\ell, Q' \in G^{\ell'}$, assuming without loss of generality that $\ell > \ell'$, we define their distance as the maximum distance between $Q$ and the descendants of $Q'$ of level $\ell$, namely
\begin{equation} \label{eq:distance_multilevel}
\dist(Q, Q') = \max_{Q'_d \in G^\ell, Q'_d \subset Q'} \dist (Q, Q'_d).
\end{equation}
It is then a simple exercise to prove that the distance satisfies a triangular inequality, in the sense that for any levels $\ell, \ell'$ and $\ell''$, it holds that
\begin{equation*}
\dist(Q,Q') \le \dist(Q,Q'') + \dist(Q'', Q'), \text{ for any } Q \in G^\ell ,Q' \in G^{\ell'}, Q'' \in G^{\ell''},
\end{equation*}
{\RV see Appendix~\ref{app:distance} for the proof.} Moreover, from these definitions and the fact that we refine dyadically, given two elements $Q, Q' \in G^\ell$ at distance $\dist(Q,Q') = s 2^{-\ell}$, for any two descendants $Q_d, Q_d' \in G^{\ell+k}$, with $Q_d \subset Q$, $Q_d' \subset Q'$ and $k>0$, their distance is bounded by
\begin{equation} \label{eq:dist_descendants}
\dist(Q_d, Q_d') \le (2^k(s+1)-1) 2^{-(\ell+k)} \le 2^{-\ell}(s+1).
\end{equation}

Finally, we also note that the distance of two elements $Q, Q' \in G^\ell$ contained in the support of a basis function,  $Q, Q' \subset \supp \phi$ for $\phi \in \Phi^\ell$, satisfies
\begin{equation} \label{eq:Csupp}
\dist(Q, Q') \le 2^{-\ell} C_{\supp},
\end{equation}
where the constant $C_{\supp}$ depends on the degree $p$ and the regularity $r$. It reaches its maximum value in the case of maximum regularity $r=p-2$, from which we obtain $C_{\supp} = \max\{p,5\}$, coming from the support of edge and vertex basis functions as depicted in Fig.~\ref{fig:basis_supp}.

Next, we bound the distance of an element to other elements on its neighborhood. 
Given an element adjacent to a vertex, $Q \in \GG^\ell$ and $Q' \in \mathcal{N}_\chi(\QQ,Q)$, using the support of vertex functions it holds that
\[
\dist(Q, Q') \le 2^{-\ell} C_{\mathcal{N}_\chi},
\]
where $C_{\mathcal{N}_\chi}$ depends on the regularity $r$, the worst case being $C_{\mathcal{N}_\chi} = 2$ for $r=p-2$. 
{\RVV Moreover, given an admissible hierarchical mesh $\QQ$, an element $Q \in \QQ \cap \GG^\ell$ and $Q' \in \mathcal{N}_{\mathcal{H}}(Q, \mu)$ if $\QQ$ is ${\cal H}$-admissible, or $Q' \in \mathcal{N}_{\mathcal{T}}(Q, \mu)$ if $\QQ$ is ${\cal T}$-admissible, it respectively holds that 
\begin{align*}
\dist(Q, Q') \le 2^{-\ell} C_{\mathcal{N_H}}, \text{ with } C_{\mathcal{N_H}} = 2^{\mu - 1} (C_{\supp} + 1) - 1, \\
\dist(Q, Q') \le 2^{-\ell} C_{\mathcal{N_T}}, \text{ with } C_{\mathcal{N_T}} = 2^{\mu - 2} (C_{\supp} + 2) - 1,
\end{align*}
where we have used the definitions of $\mathcal{N}_{\mathcal{H}}(Q,\mu)$ and $S_{\mathrm{ext}}(Q,\ell-\mu+1)$, respectively $\mathcal{N}_{\mathcal{T}}(Q,\mu)$ and $S_{\mathrm{ext}}(Q,\ell-\mu+2)$, the bound \eqref{eq:Csupp} and the first bound in \eqref{eq:dist_descendants}. Note that it always holds that $ C_{\mathcal{N_T}} < C_{\mathcal{N_H}}$.}}



The following lemma adapts the result of \cite[Lemma~12]{bgmp16} to the  hierarchical multi-patch configuration here considered.
\begin{lemma} \label{lem:complexity_auxiliary}
Let $\QQ$ be an admissible mesh of class $\mu \ge 2$ satisfying the assumptions on Theorem~\ref{thm:linear_independence}, $Q' \in \QQ$, and $\QQ^* = \mathtt{REFINE}(\QQ, \{Q'\}, \mu)$ the mesh given by Algorithm~\ref{alg:hrefine-admissible} when marking only $Q'$. Then, for any element $Q \in \QQ^* \setminus \QQ$  it holds that
\[
\dist(Q,Q') \le {\RV 2^{-\ell(Q)+1}} C_{\dist}, \, \text{ with } {\RV C_{\dist} = \frac{C_{\mathcal{N}_\chi} + C_{\mathcal{N}}}{1-2^{1-\mu}}},
\]
{\RV where $C_{\mathcal{N}} = C_{\mathcal{N_T}}$ and $C_{\mathcal{N}} = C_{\mathcal{N_H}}$ for $\mathcal{T}$- and $\mathcal{H}$-admissible meshes, respectively.}
\end{lemma}
\begin{proof}
Let us assume that $\QQ$ and $\QQ^*$ are obtained by $\mathcal{T}$-admissible refinement, the case of $\mathcal{H}$-admissible refinement is proved analogously. Since $Q$ is activated from applying Algorithm~\ref{alg:hrefine-admissible}, there exists a sequence of elements $Q' = Q_J, Q_{J-1}, \ldots, Q_0$, such that $Q$ is a child of $Q_0$, and for each $j=1,\ldots, J$ either 
\begin{align*}
  Q_{j-1} \in {\RV \mathcal{N}_{\mathcal{T}}(Q_j,\mu)} \quad \text{ or } \quad
  Q_{j-1} \in {\RV \mathcal{N}_{\chi}(Q_j)},
\end{align*}
and moreover two markings of the second type, i.e., due to the vertex-patch neighborhood, do not appear consecutively. In the first case we have
\begin{equation} \label{eq:aux_neighborhood}
\dist(Q_j,Q_{j-1}) \le {\RV 2^{-\ell(Q_j)} C_{\mathcal{N_T}}} \; \text{ and } \; \ell(Q_{j-1}) = \ell(Q_j) - \mu + 1,
\end{equation}
while the second case gives
\[
\dist(Q_j,Q_{j-1}) \le {\RV 2^{-\ell(Q_j)} C_{\mathcal{N}_\chi}} \; \text{ and } \; \ell(Q_{j-1}) = \ell(Q_j).
\]
{\RV From the triangular inequality for the distance, it holds that}
\[
\dist(Q,Q') \le \dist(Q,Q_0) + \dist(Q_0,Q') \le \dist(Q,Q_0) + \sum_{j=1}^J \dist(Q_j,Q_{j-1}),
\]
and since $Q$ is a child of $Q_0$ {\RV obtained by} dyadic refinement, {\RV from the definition of the distance the first term satisfies $\dist(Q,Q_0) = 2^{-\ell(Q)}$.}

For the sum, we know that two markings from the vertex-patch neighborhood do not appear consecutively, thus we can put ourselves in the worst case scenario, where the two types of marking alternate at every step. For simplicity, we can assume that $J$ is even and, without loss of generality, that we mark the admissibility neighborhood at odd steps, and the vertex-patch neighborhood at even steps. We then have
\[
\ell(Q_j) = \left \{
\begin{array}{ll}
\ell(Q_0) + (\mu-1) \, (j+1)/2  &\text{ if $j$ is odd}, \\
\ell(Q_0) + (\mu-1) \, j/2 &\text{ if $j$ is even},
\end{array}
\right.
\]
from what we obtain
\begin{align*}
  & \sum_{j=1}^J \dist(Q_j,Q_{j-1}) = \sum_{k=1}^{J/2} \left(\dist(Q_{2k},Q_{2k-1}) + \dist(Q_{2k-1},Q_{2k-2}) \right)\\
  & \le \sum_{k=1}^{J/2} {\RV \left( 2^{-\ell(Q_{2k})} C_{\mathcal{N}_\chi} + 2^{-\ell(Q_{2k-1})} C_{\mathcal{N_T}} \right) = \sum_{k=1}^{J/2} 2^{-\ell(Q_0) - k(\mu-1)} (C_{\mathcal{N}_\chi} + C_{\mathcal{N_T}}) } \\
  & < {\RV 2^{-\ell(Q_0)} (C_{\mathcal{N}_\chi} + C_{\mathcal{N_T}})} \sum_{k=0}^{\infty} 2^{- k(\mu-1)} = \frac{{\RV 2^{-\ell(Q)+1}}}{1- 2^{1-\mu}} (C_{\mathcal{N}_\chi} + {\RV C_{\mathcal{N_T}}}),
\end{align*}
where in the last step we have used that $Q$ is a child of $Q_0$, and the same arguments as in \cite{bgmp16}. The proof for the $\mathcal{H}$-admissible case is analogous, replacing the $\mathcal{N_T}$ neighborhood by the $\mathcal{N_H}$ neighborhood, and the constant {\RV $C_{\mathcal{N_T}}$ by $C_{\mathcal{N_H}}$} in \eqref{eq:aux_neighborhood}.
\end{proof}

The following theorem states the linear complexity of the refinement algorithm, and generalizes the single-patch complexity estimates \cite[Theorem~13]{bgmp16} to the multi-patch case, see also \cite[Theorem~2.4]{bdd04} and \cite[Theorem~3.2]{stevenson07}.
\begin{theorem}
Let $\QQ_0 = \GG^0$ and $\mu \ge 2$, and let $\QQ_0, \QQ_1, \ldots, \QQ_J$ the sequence of admissible meshes generated from the call to Algorithm~\ref{alg:hrefine-admissible}, namely
\[
\QQ_j = \mathtt{REFINE}(\QQ_{j-1}, {\cal M}_{j-1}, \mu), \quad {\cal M}_{j-1} \subseteq \QQ_{j-1} \text{ for } j \in \{1, \ldots, J\}.
\]
Then, there exists a positive constant $\Lambda$ such that 
$\#\QQ_J - \# \QQ_0 \le \Lambda \sum_{j=0}^{J-1} \# {\cal M}_j$,
which {\RV depends on $p,r,\mu$, 
the maximum valence $\nu$ and the admissibility type.}
\end{theorem}
\begin{proof}
The proof is completely analogous to the one in \cite[Theorem~13]{bgmp16}, using the new estimate introduced in Lemma~\ref{lem:complexity_auxiliary} in place of \cite[Lemma~12]{bgmp16}. The only difference is in the bound of the number of elements in the set
\[
B(Q',j) = \{ Q \in \GG^j : \dist(Q, Q') < 2^{1-j} C_{\dist} \},
\]
the set of elements of level {\RV $0 \le j \le \ell(Q')+1$} with distance to $Q' \in \QQ$ smaller than $2^{1-j} C_{\dist}$. A bound independent on $Q'$ and $j$ is easily proved {\RV by relating $B(Q',j)$ to a region $\Pi^s(Q')$, with $s$ depending on $C_{\dist}$, and from the boundedness of the number of elements contained in $\Pi^s(Q')$.}
\end{proof}

%% file: s7_numerical_tests.tex
This section contains some numerical tests showing the application of the hierarchical $C^1$ spaces to adaptive isogeometric methods. In the examples, we consider both the Poisson problem and the biharmonic problem, where the adaptive refinement is {\RV driven by an a posteriori error estimator}. The implementation is done based on the algorithms for hierarchical splines from \cite{garau2018}, using the representation in terms of B-splines in Section~\ref{subsec:standard_representation} and the refinement mask in Section~\ref{sec:nestedness} for the evaluation and truncation of basis functions. {\RV For building the geometry parameterizations {\MK in Examples~\ref{ex:Example1} and \ref{ex:Example2}} we follow the same approach as in \cite{FJKT23}, creating first an analysis-suitable $G^1$ geometry as a template, and then refitting it into a pullback {\MK of $\UW^2$}.
}

Our code is written in Matlab and is an extension of the one for the two-patch case \cite{BrGiKaVa20}, to which we refer for the details. The most important differences come from the need for local re-parameterizations, as already mentioned in Remark~\ref{rem:orientation}. These can be easily computed by arranging the control points in a two-dimensional array and applying simple changes in the directions of the array\footnote{In Matlab, one can use the commands \texttt{fliplr}, \texttt{flipud} and \texttt{transpose}, or just consecutive uses of \texttt{rot90} if the Jacobian is assumed to be positive.}, and the same kind of arrangement should be applied to the indices of B-spline functions when computing the coefficients in Sections~\ref{subsec:standard_representation} and~\ref{sec:nestedness}. Moreover, since every edge is attached to two vertices, the orientation of the edge in the vertex configuration may differ with respect to the one used for the definition of edge basis functions, in the sense that the relative position of two adjacent patches will change. In this case it is necessary, first, to correctly identify the five ``extended'' edge functions which are close to the vertex to compute the restricted matrices $\widehat E_{i_m,1}$ and $\widehat E_{i_{m+1},0}$ in Section~\ref{subsec:standard_representation}, and second to take into account the sign change in the gluing data and the vectors ${\bf d}^{(i)}(\xi)$ and ${\bf t}^{(i)}(\xi)$ defined in Section~\ref{subsubsec:ASG1}, which also influences the sign of the last two columns of matrices $K_{i,m}$ and $K_{i,m+1}$.


\subsection{Poisson problem} \label{sec:poisson}
In the first two examples we consider the Poisson problem
\[
\left \{
\begin{array}{rl}
- \Delta u = f & \text{ in } \Omega, \\
u = g & \text{ on } \partial \Omega, 
\end{array}
\right.
\]
which we solve by an adaptive isogeometric method, see, e.g., \cite{BrBuGiVa19}.
More precisely, we {\sl solve} the problem in its variational formulation imposing the Dirichlet boundary condition by Nitsche's method. {\RV Imposing the boundary condition strongly is cumbersome, but still possible, because the restriction of the $C^1$ basis functions to the boundary is linearly dependent}. Let us denote $\mathbb{W}_h = \mathrm{span}\{\mathcal{H}_\UW\}$, we determine $u_h \in \mathbb{W}_h$ such that for all $v_h \in \mathbb{W}_h$
\begin{align*}
\int_\Omega \nabla u_h \cdot \nabla v_h - \int_{\Gamma_D} \frac{\partial u_h}{\partial {\bf n}} v_h - \int_{\Gamma_D} u_h \frac{\partial v_h}{\partial \bf n} + 
\int_{\Gamma_D} \frac{\gamma}h_Q u_h v_h \\
= 
\int_\Omega f v_h - \int_{\Gamma_D} g \frac{\partial v_h}{\partial \bf n} + \int_{\Gamma_D} \frac{\gamma}h_Q g v_h,
\end{align*}
where $h_Q$ is the local element size, and $\gamma = 10(p+1)$, with $p$ being the degree, is the penalization {\RV parameter}. The {\sl error estimate} is computed with the residual-based estimator
\[
\varepsilon^2(u_h) = \sum_{Q \in {\cal Q}} \varepsilon^2_{Q}(u_h), \; \text{ with } \; \varepsilon^2_{Q}(u_h)=h_Q^2 \int_{Q} \vert f+\Delta u_h\vert^2.
\]
The {\sl marking} of the elements at each iteration is done using D\"orfler's strategy. 
In the {\sl refinement} step we apply Algorithm~\ref{alg:hrefine-admissible}, and therefore we refine dyadically the marked elements, plus the ones necessary to guarantee linear independence and {\RV admissibility}.

For both examples we report the results for degrees $p=3, 4$ and $5$, with regularity~$r=p-2$ and $C^1$ smoothness across the interfaces. We test the methods obtained by employing both the non-truncated and the truncated basis, and the refinement providing admissibility of class $\mu=2,3$. The goal is to show that using the $C^1$ space basis does not spoil the properties of the local refinement, and in particular the advantage over uniform refinement.
\begin{figure}[t!]
\begin{center}
\begin{subfigure}[Domain of Example~\ref{ex:Example1}.]{
  \includegraphics[trim=10cm 0cm 8cm 1cm,width=0.4\textwidth,clip]{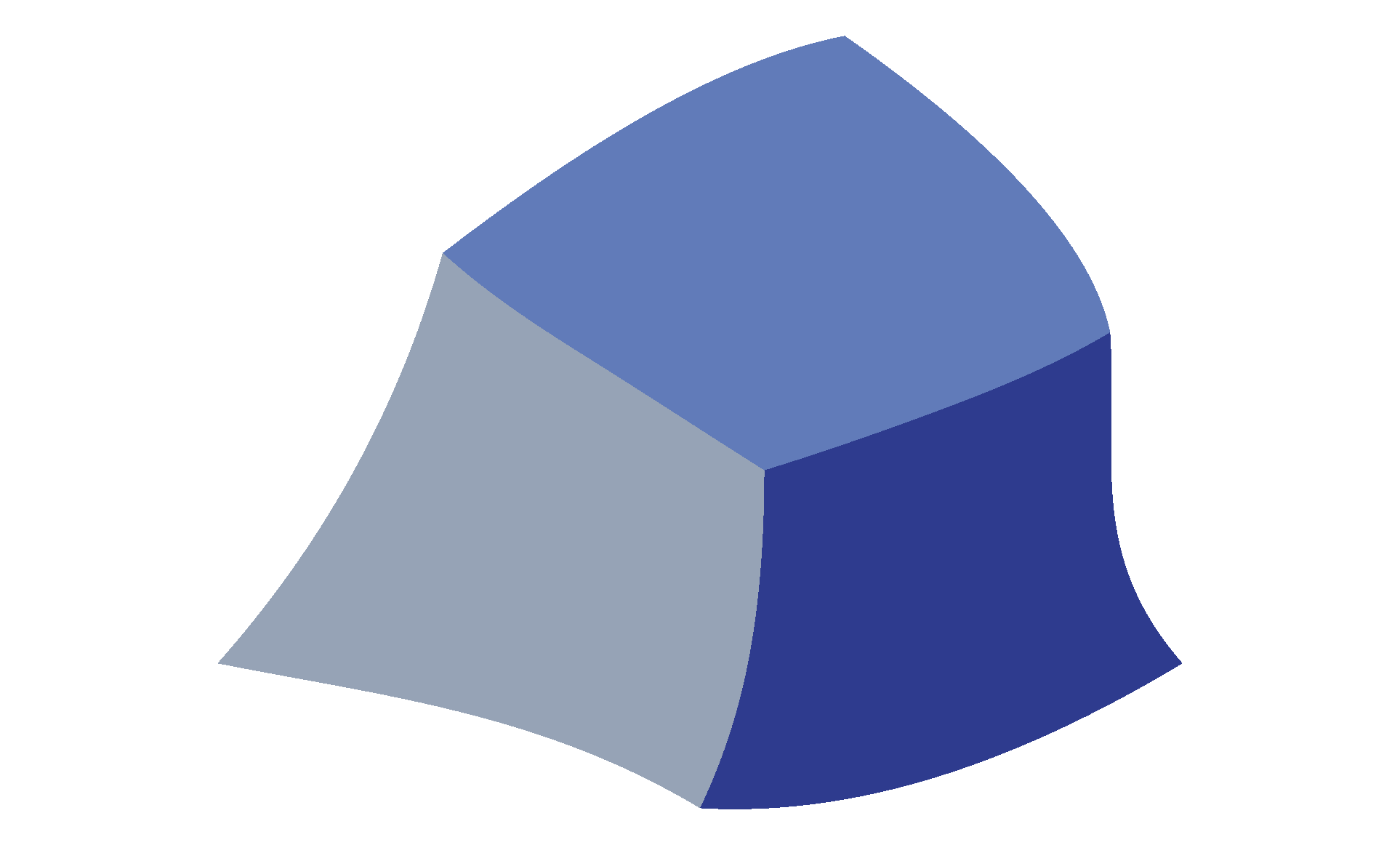} 
\label{fig:ex1-domain}
}
\end{subfigure}
\begin{subfigure}[Exact solution of Example~\ref{ex:Example1}.]{
	\includegraphics[trim=1cm 0mm 0cm 0mm,width=0.47\textwidth,clip]{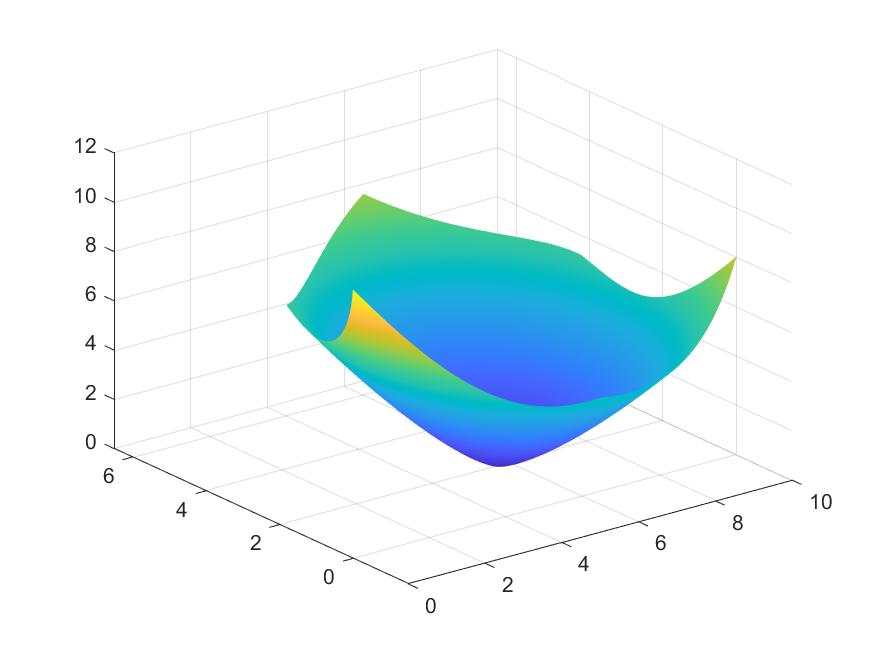}
\label{fig:ex1-solution}
}
\end{subfigure}
\begin{subfigure}[Domain of Example~\ref{ex:Example2}.]{
  \includegraphics[trim=10cm 5cm 8cm 2cm,width=0.4\textwidth,clip]{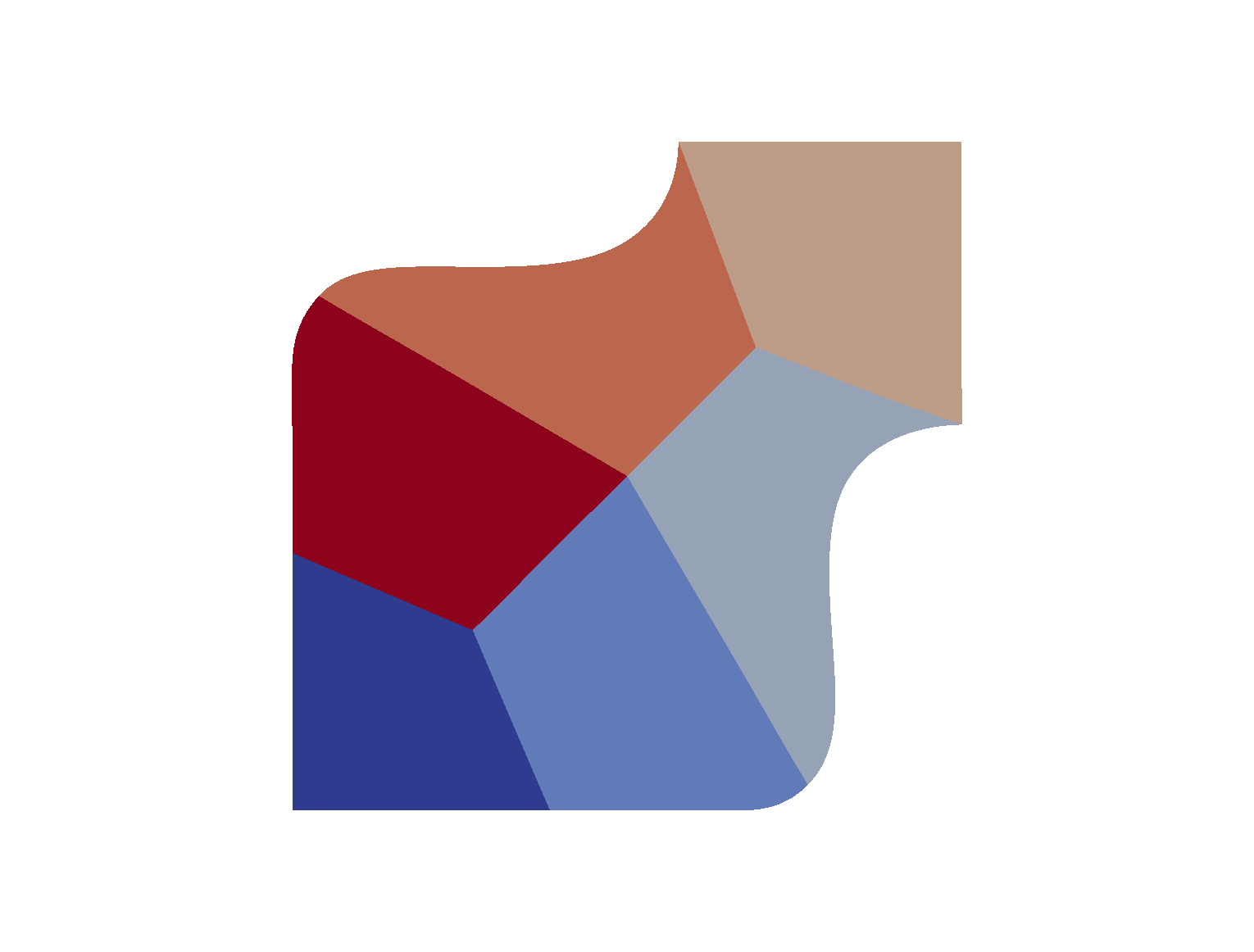} 
\label{fig:ex2-domain}
}
\end{subfigure}
\begin{subfigure}[Exact solution of Example~\ref{ex:Example2}.]{
	\includegraphics[trim=1cm 0mm 0cm 0mm,width=0.47\textwidth,clip]{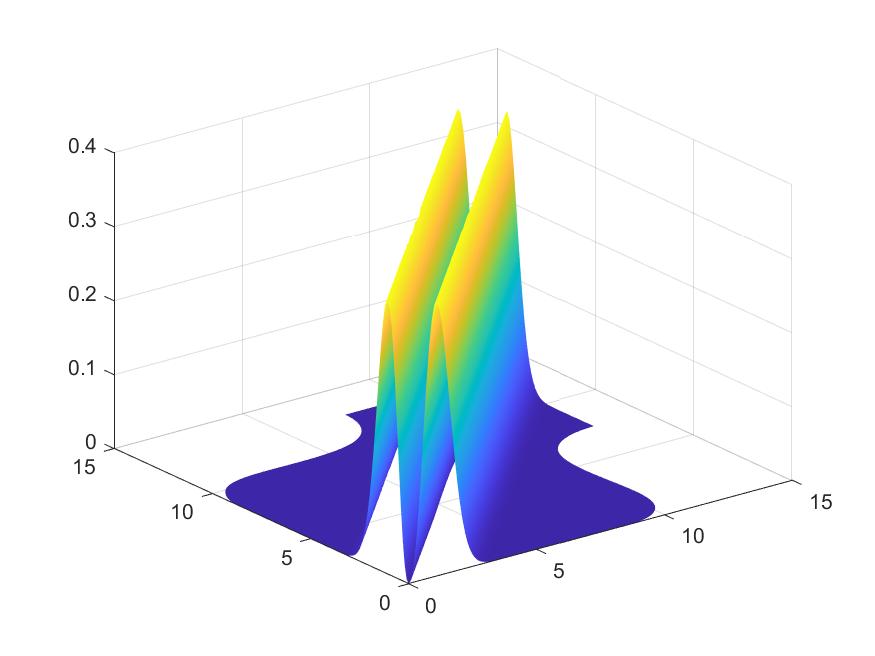}
\label{fig:ex2-solution}
}
\end{subfigure}
 \caption{{\RV Domains and exact solutions of Example~\ref{ex:Example1} (top) and Example~\ref{ex:Example2} (bottom).}}
 \label{fig:domains1}
 \end{center}
\end{figure}

\begin{example} \label{ex:Example1}
For the first numerical example we consider the three-patch domain shown in Fig.~\ref{fig:ex1-domain} which has been constructed in \cite{KaSaTa19a} and possesses an analysis-suitable $G^1$ multi-patch parameterization. We study the Poisson problem with exact solution
\begin{equation*}
{\RV u(\mathbf{x})= \left \Vert \mathbf{x}-\mathbf{P} \right \Vert^{\frac{4}{3}}},
\end{equation*}
which is characterized by a {\RV singularity} at the point $\mathbf{P} = \left(17/3, 2\right)$, coinciding with the interior vertex of the geometry, see Fig.~\ref{fig:ex1-solution}. 

The starting coarse mesh has $4 \times 4$ elements on each patch, and we use D\"orfler's parameter equal to $0.80$ for marking the elements. We run the adaptive method until the hierarchical space reaches {\RV twelve} levels. The behavior of the error in $H^1$ semi-norm with respect to the number of degrees of freedom (NDOF) is presented in Fig.~\ref{fig:laplacianex1}, where it is evident the advantage of using local refinement over the uniform one, regardless of the chosen basis and of the admissibility class. For higher degrees, the refinement for the truncated basis tends to give smaller errors, but in all cases the optimal convergence rate is achieved. 

In Fig.~\ref{fig:hmsh_ex1} we show the meshes obtained for degree $p=5$ {\RV and admissibility class $\mu=2$} after reaching six levels. As already observed in \cite{BrBuGiVa19} for THB-splines, the refinement based on the truncated basis is more local than the one for the hierarchical basis. The figure also shows how marking elements adjacent to the vertex extends the refinement to the elements in the vertex-patch neighborhood, but this does not affect the convergence rates.

\begin{figure}
\begin{subfigure}[Degree 3.]
{  \includegraphics[trim=5mm 0mm 5mm 0mm,width=0.48\textwidth,clip]{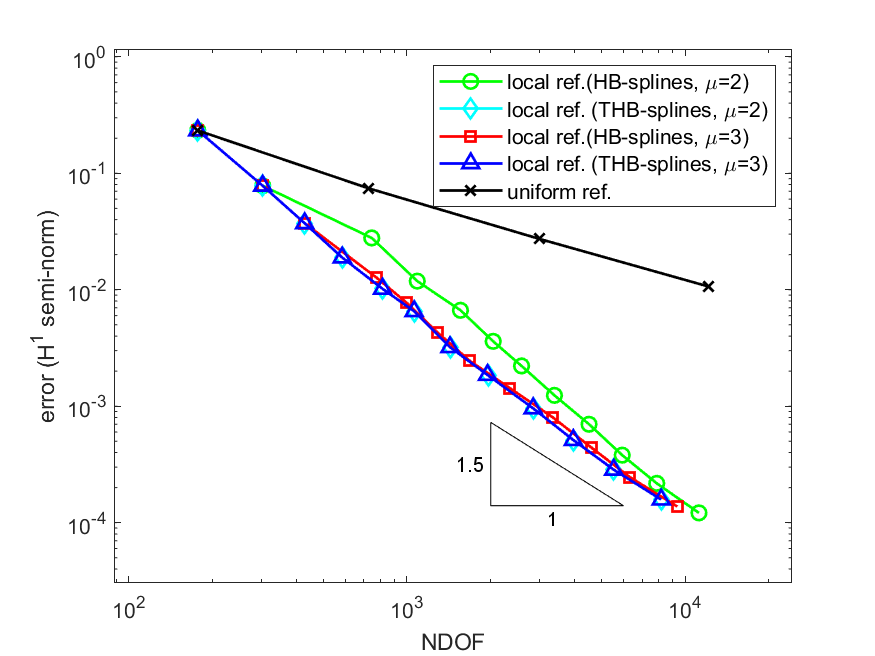}}
\end{subfigure}
\begin{subfigure}[Degree 4.]
{  \includegraphics[trim=5mm 0mm 5mm 0mm,width=0.48\textwidth,clip]{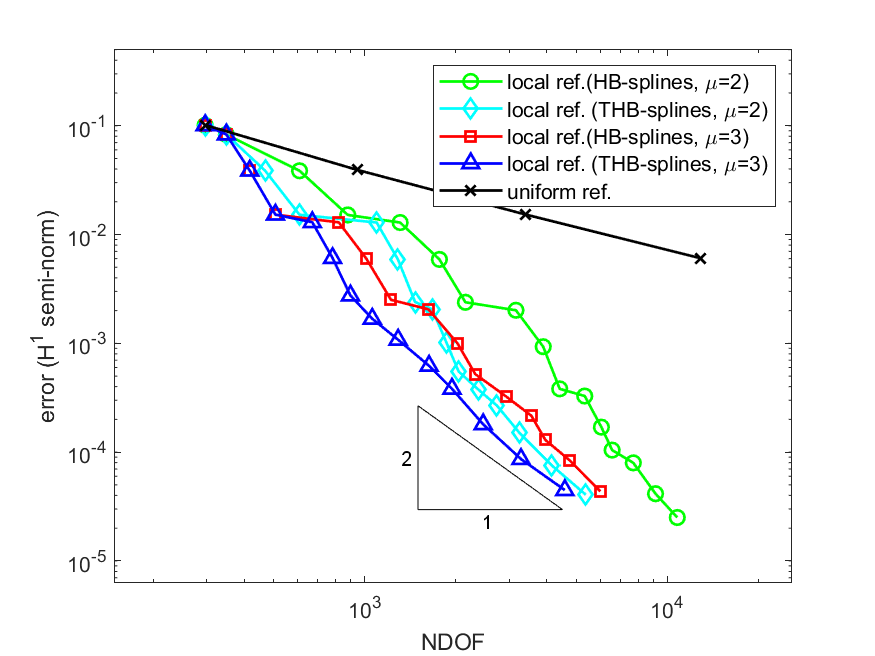}}
\end{subfigure}
\centering
\begin{subfigure}[Degree 5.]
{  \includegraphics[trim=5mm 0mm 5mm 0mm,width=0.50\textwidth,clip]{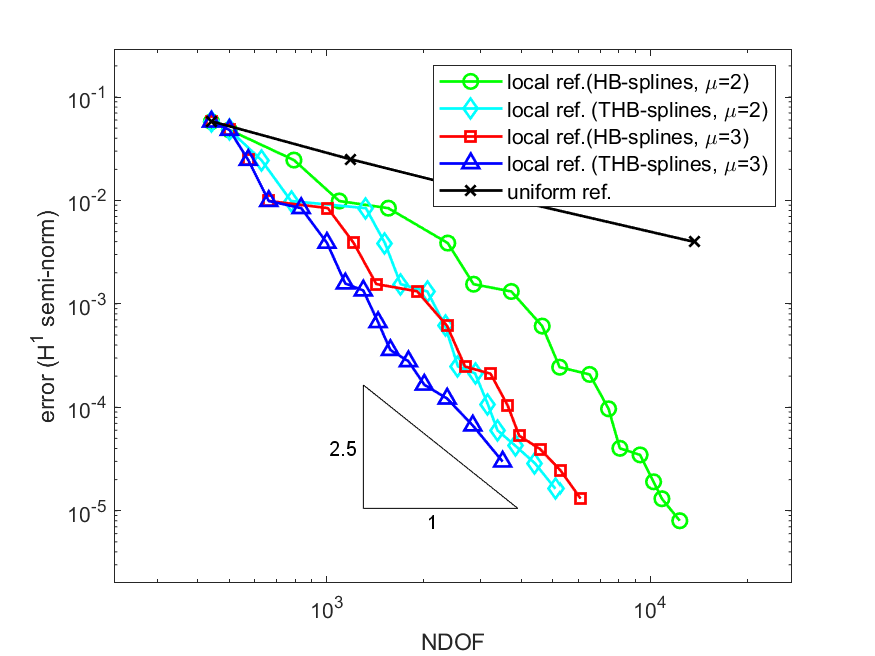}}
\end{subfigure}
 \caption{Example~\ref{ex:Example1} (Poisson problem on a three-patch domain): convergence plots for degrees 3, 4 and 5 with admissibility $\mu=2,3$ in the refinement.}
 \label{fig:laplacianex1}
\end{figure}
\begin{figure}[th]
\begin{center}
\begin{subfigure}[Non-truncated, $\mu=2$, {\RV 3723} NDOF.]{
  \includegraphics[trim=60mm 10mm 60mm 10mm,width=0.47\textwidth,clip]{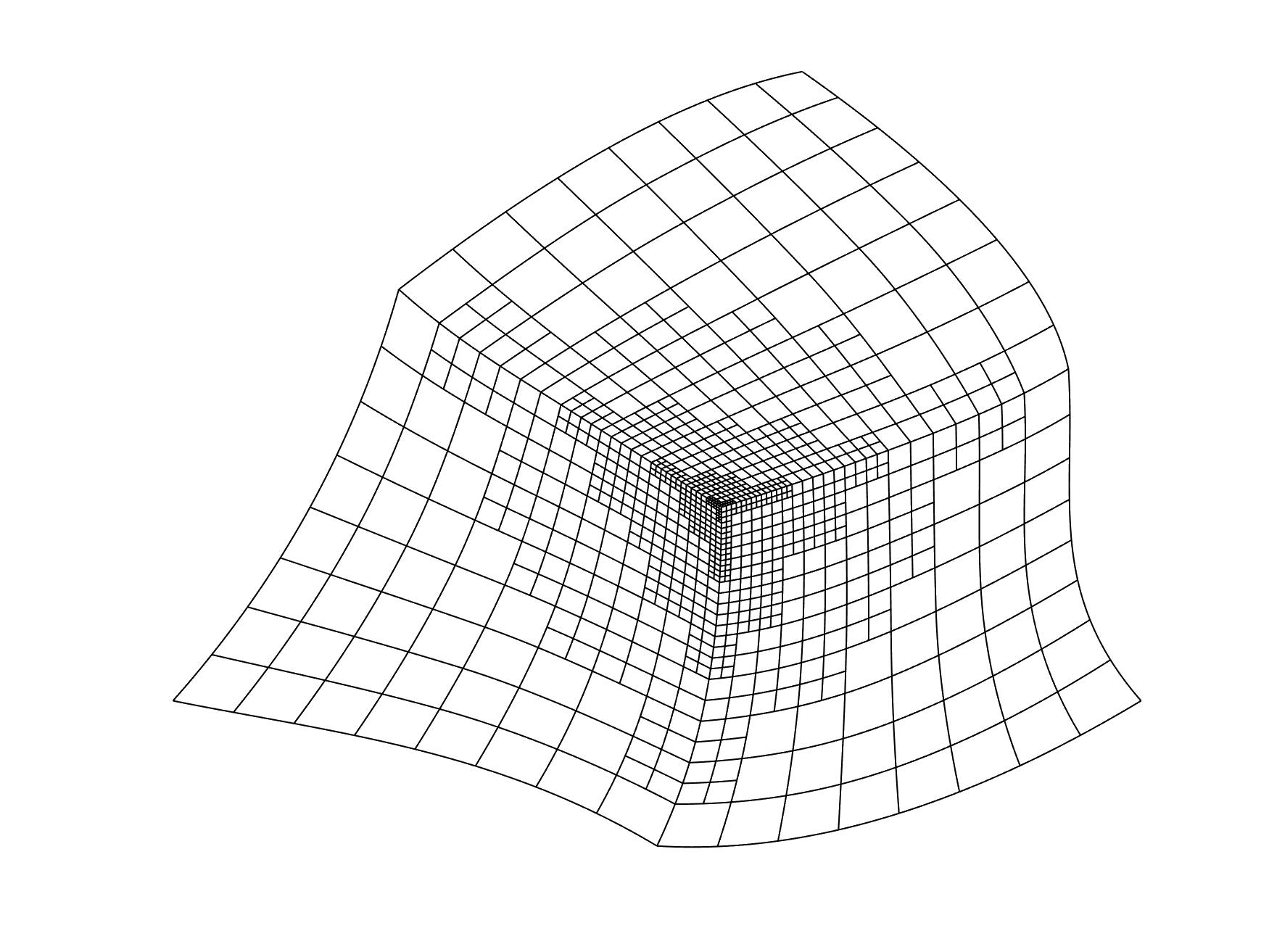}
}
\end{subfigure}
\begin{subfigure}[Truncated, $\mu=2$, {\RV 2049} NDOF.]{
  \includegraphics[trim=60mm 10mm 60mm 10mm,width=0.47\textwidth,clip]{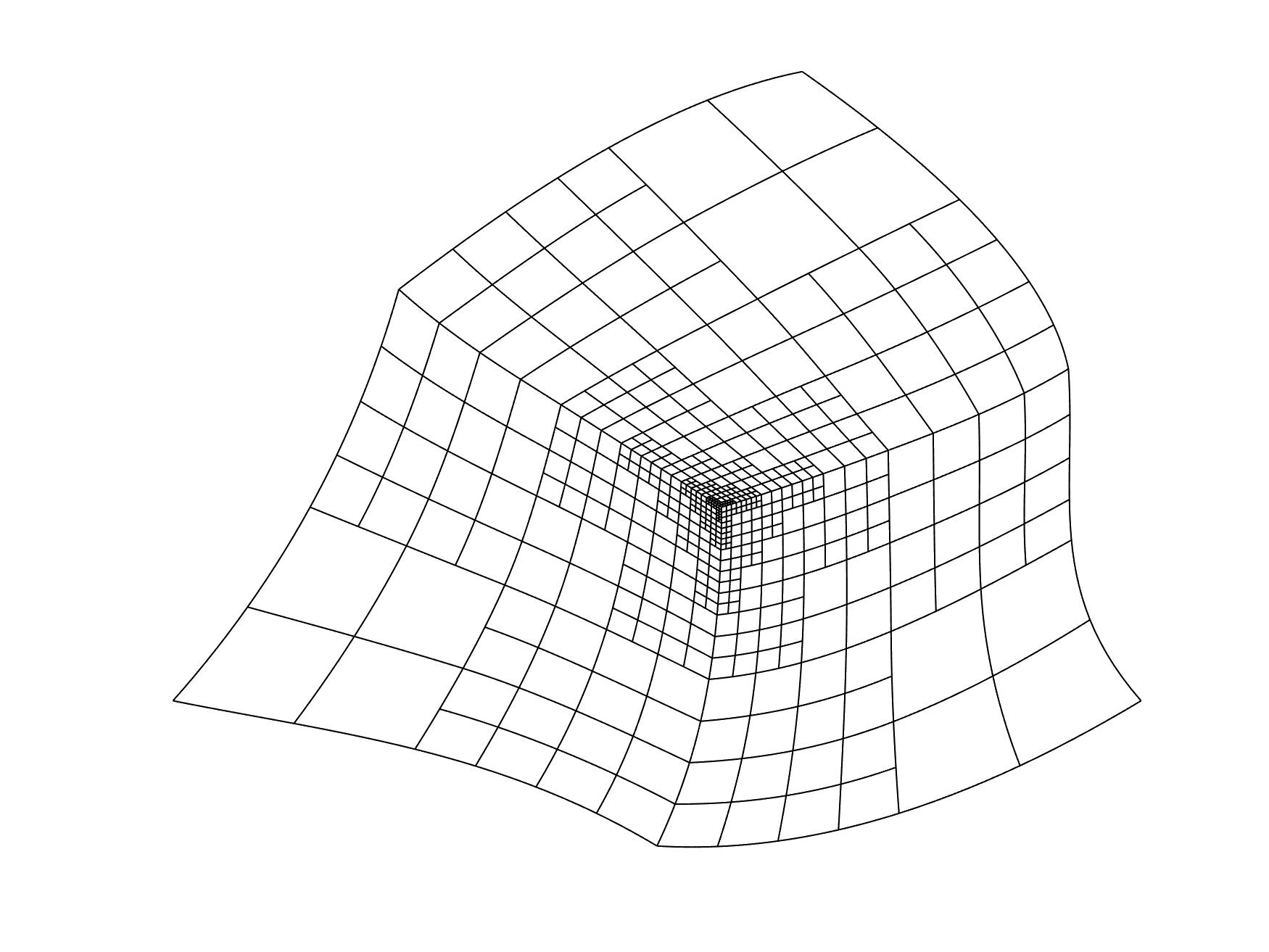}
}
\end{subfigure}
 \caption{Example~\ref{ex:Example1}: six-level meshes obtained for degree~$5$.}
 \label{fig:hmsh_ex1}
 \end{center}
\end{figure}

\end{example}

\begin{example} \label{ex:Example2}
In the second example, we consider the six-patch domain shown in Fig.~\ref{fig:ex2-domain} with an analysis-suitable $G^1$ multi-patch parameterization, {\RV which is a refitting {\MK of the geometry from Example~2 in \cite{KaSaTa17b} into a pullback of $\UW^2$}},
and the data of the problem are chosen such that the exact solution is
\begin{equation*}
{\RV u(x,y) = (y-x)^{\frac{7}{3}}e^{-(y-x)^2}}.
\end{equation*}
{\RV This solution has a singularity along the straight line $y=x$} crossing the whole domain: in the middle part it coincides with two interfaces, while at the two endpoints it crosses two patches, see Fig.~\ref{fig:ex2-solution}. {\RV Moreover, next to the singularity line there are two other smooth but quite sharp ridges, which also require local refinement}. In this example the initial mesh has $6 \times 6$ elements on each patch, and we run the adaptive method with D\"orfler's parameter equal to {\RV $0.80$ until the dimension of the hierarchical space exceeds $8\cdot 10^4$. The example of the hierarchical mesh in Fig.~\ref{fig:laplacianex2b} indicates that the refinement is indeed localized around the singularity. In Fig.~\ref{fig:laplacianex2a} we compare the convergence results obtained for uniform refinement and for the adaptive method with ${\cal T}$-admissible meshes and $\mu=3$, which show that adaptive refinement gives a clear advantage over uniform meshes. However, the optimal convergence rate is not achieved, because the edge singularity would require to refine in an anisotropic fashion. Since the solution belongs to $H^{17/6-\epsilon}(\Omega)$ for any $\epsilon >0$, using the same heuristic arguments as in Section~6.1.6 of \cite{reviewadaptiveiga} (see also references therein), the expected convergence rate for isotropic meshes is $\min\{p/2, 11/6\}$, where $11/6$ doubles the one obtained for uniform refinement.}


\begin{figure}
\begin{center}
  \begin{subfigure}[{\RV Mesh for $p=4$, 80961 NDOF.}]
{\includegraphics[trim=100mm 10mm 140mm 0mm,width=0.43\textwidth,clip]{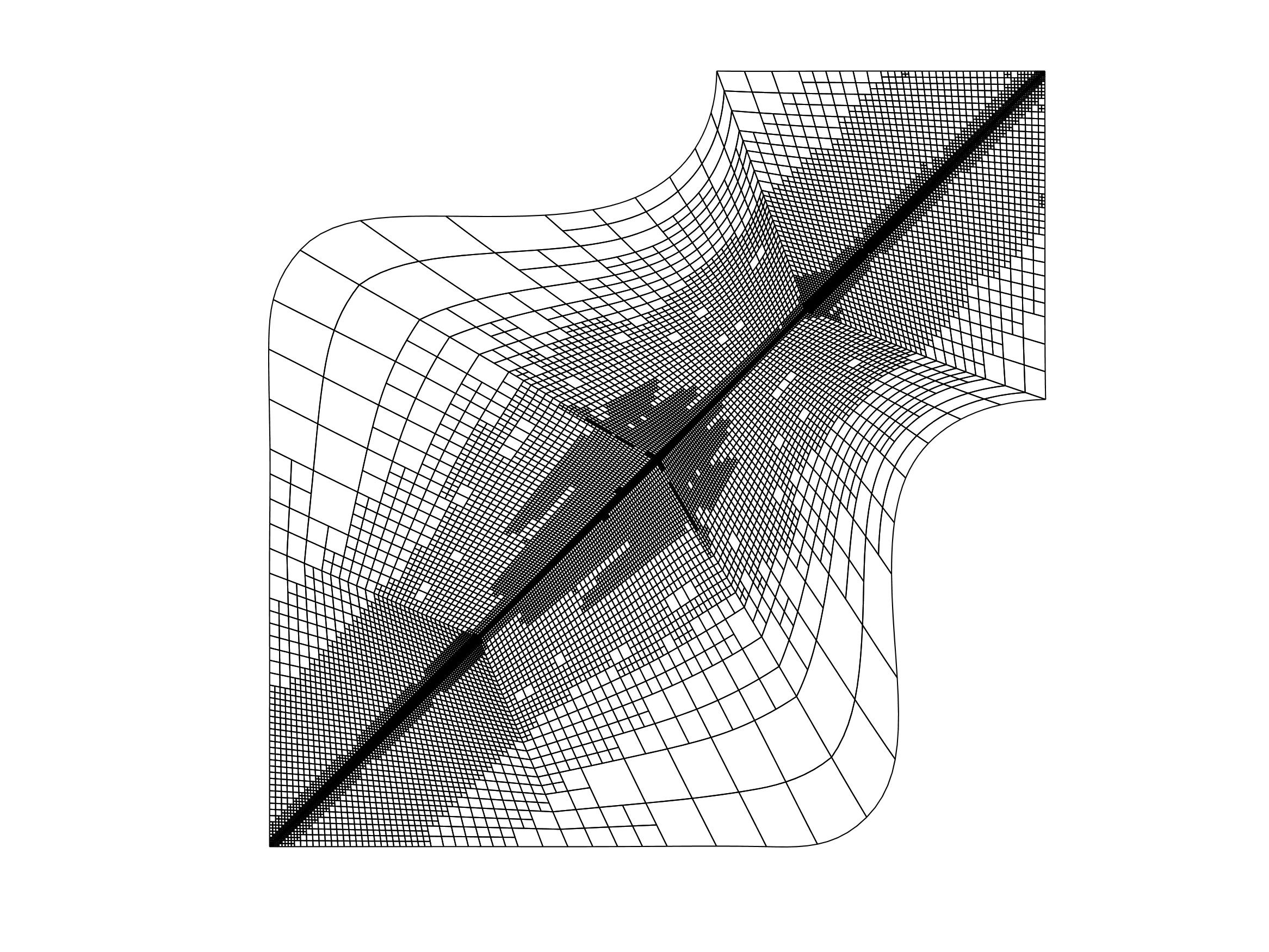} \label{fig:laplacianex2b}}
\end{subfigure}
  \begin{subfigure}[{\RV Convergence plots with degrees 3, 4 and 5.}]
{\includegraphics[trim=5mm 0mm 10mm 0mm,width=0.53\textwidth,clip]{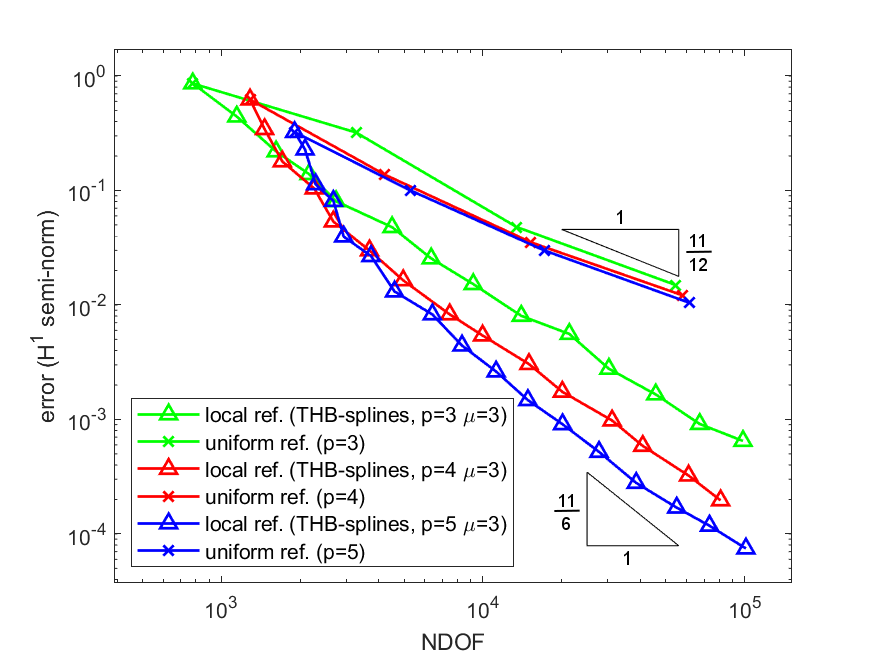} \label{fig:laplacianex2a}} 
\end{subfigure}
 \caption{Example~\ref{ex:Example2}: {\RV a hierarchical mesh, and convergence for ${\cal T}$-admissible meshes with $\mu=3$.}}
 \label{fig:laplacianex2}
 \end{center}
\end{figure}


\end{example}

\subsection{Biharmonic problem}
For the final numerical test, we consider the biharmonic problem
\[
\left \{
\begin{array}{rl}
\Delta^2 u = f & \text{ in } \Omega, \\
u = g_1 & \text{ on } \partial \Omega, \\
\displaystyle \frac{\partial u}{\partial n} = g_2 & \text{ on } \partial \Omega.
\end{array}
\right.
\]
In order to solve the direct formulation of this problem with a Galerkin method, we need to use a discretization space of $C^1$ functions, and therefore in this case the $C^1$ hierarchical basis is a natural choice to define an adaptive isogeometric method. Let us denote $\mathbb{W}_h = \mathrm{span}\{\mathcal{H}_\UW\}$ and $\mathbb{W}_{0,h} = \mathbb{W}_h \cap H^2_0(\Omega)$. The problem is to find $u_{0,h} \in \mathbb{W}_{0,h}$ such that for all $v_h \in \mathbb{W}_{0,h}$ it holds that 
\[
\int_\Omega \Delta u_{0,h} \Delta v_h = \int_\Omega f v_h - \int_\Omega \Delta u_{b,h} \Delta v_h,
\]
where $u_{b,h} \in \mathbb{W}_h$ is a discrete function that satisfies the boundary conditions, {\RV and we solve it with an adaptive isogeometric method.}
To avoid the computation of third and fourth order derivatives that would appear on Nitsche's method, the boundary conditions are imposed strongly through a projection into the space generated by boundary functions. For the same reason, instead of the residual error estimator we use the estimator presented in \cite{AnBuCo20}, which follows the original idea of \cite{BaSm93}, by enriching the space with $C^1$ bubble functions of degree $p+1$, and support on one single element. In particular, if we define the space of bubble functions $\mathbb{B}_h$, and define our solution as $u_h = u_{0,h} + u_{b,h}$, we compute an estimator of the error as the unique function $e_h \in \mathbb{B}_h$ such that for all $b_h \in \mathbb{B}_h$ it holds
\[
\int_\Omega \Delta e_{h} \Delta b_h = \int_\Omega f b_h - \int_\Omega \Delta u_{h} \Delta b_h,
\]
and an estimate of the error on each element $Q \in \QQ$ is given by computing the energy norm $\| e_h \|_{E(Q)}$.

\begin{example} \label{ex:Example3}
For the last numerical test we solve the biharmonic problem in the L-shaped domain composed of eight bilinearly parameterized patches as depicted in Fig.~\ref{fig:ex3-domain}, with exact solution, in polar coordinates $(\rho, \theta)$, given by 
\[
u(\rho,\theta)=\rho^{z+1}(C_1\,F_1(\theta) - C_2\,F_2(\theta)),
\]
where
\begin{align*}
&C_1=\frac{1}{z-1} \sin\left(\frac{3(z-1)\pi}{2}\right) - 
\frac{1}{z-1}\sin\left(\frac{3(z+1)\pi}{2}\right),\\
&C_2=\cos\left(\frac{3(z-1)\pi}{2}\right) - \cos\left(\frac{3(z+1)\pi}{2}\right),\\
&F_1(\theta)=\cos((z-1)\theta) - \cos((z+1)\theta),\\
&F_2(\theta)=\frac{1}{z-1}\sin((z-1)\theta) - \frac{1}{z+1}\sin((z+1)\theta).
\end{align*}
and $z= 0.544483736782464$, that is, the smallest positive solution of 
\[
\sin (z \omega) + z \sin (\omega) = 0,
\]
with $\omega = 3\pi / 2$ for the L-shaped domain, see \cite[Section~3.4]{Grisvard}. It is well known that this solution has a singularity at the re-entrant corner. We present the results for degrees $p=3, 4, 5$, with regularity~$r = p-2$, obtained by employing both the non-truncated and the truncated basis, with admissibility of class $\mu=3$, and D\"orfler parameter equal to $0.80$.

\begin{figure}[th]
\begin{center}
\begin{subfigure}[Domain of Example~\ref{ex:Example3}]{
  \includegraphics[trim=10cm 0cm 8cm 1cm,width=0.4\textwidth,clip]{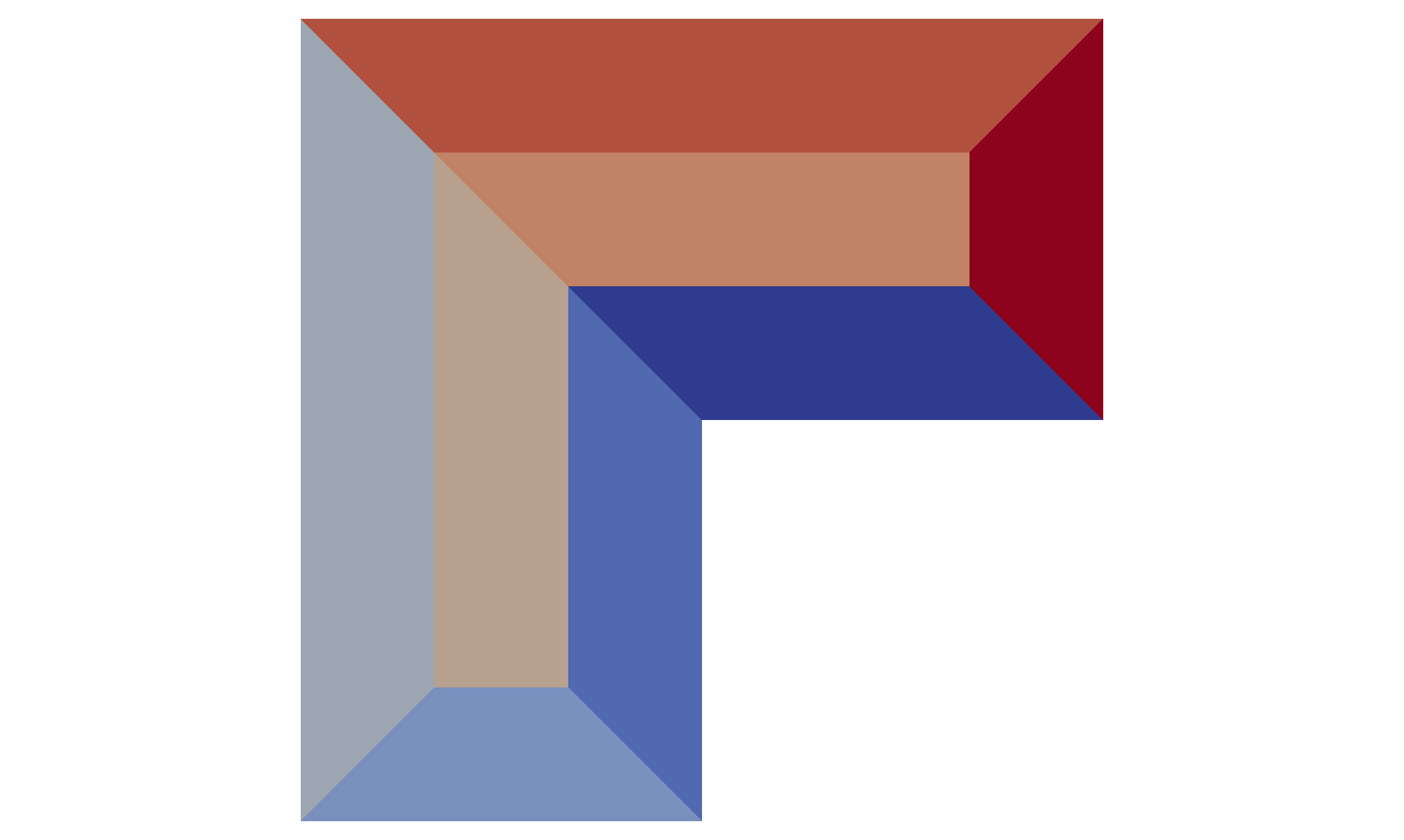} 
\label{fig:ex3-domain}
}
\end{subfigure}
\begin{subfigure}[Exact solution of Example~\ref{ex:Example3}]{
  \includegraphics[trim=10mm 0mm 10mm 0mm,width=0.42\textwidth,clip]{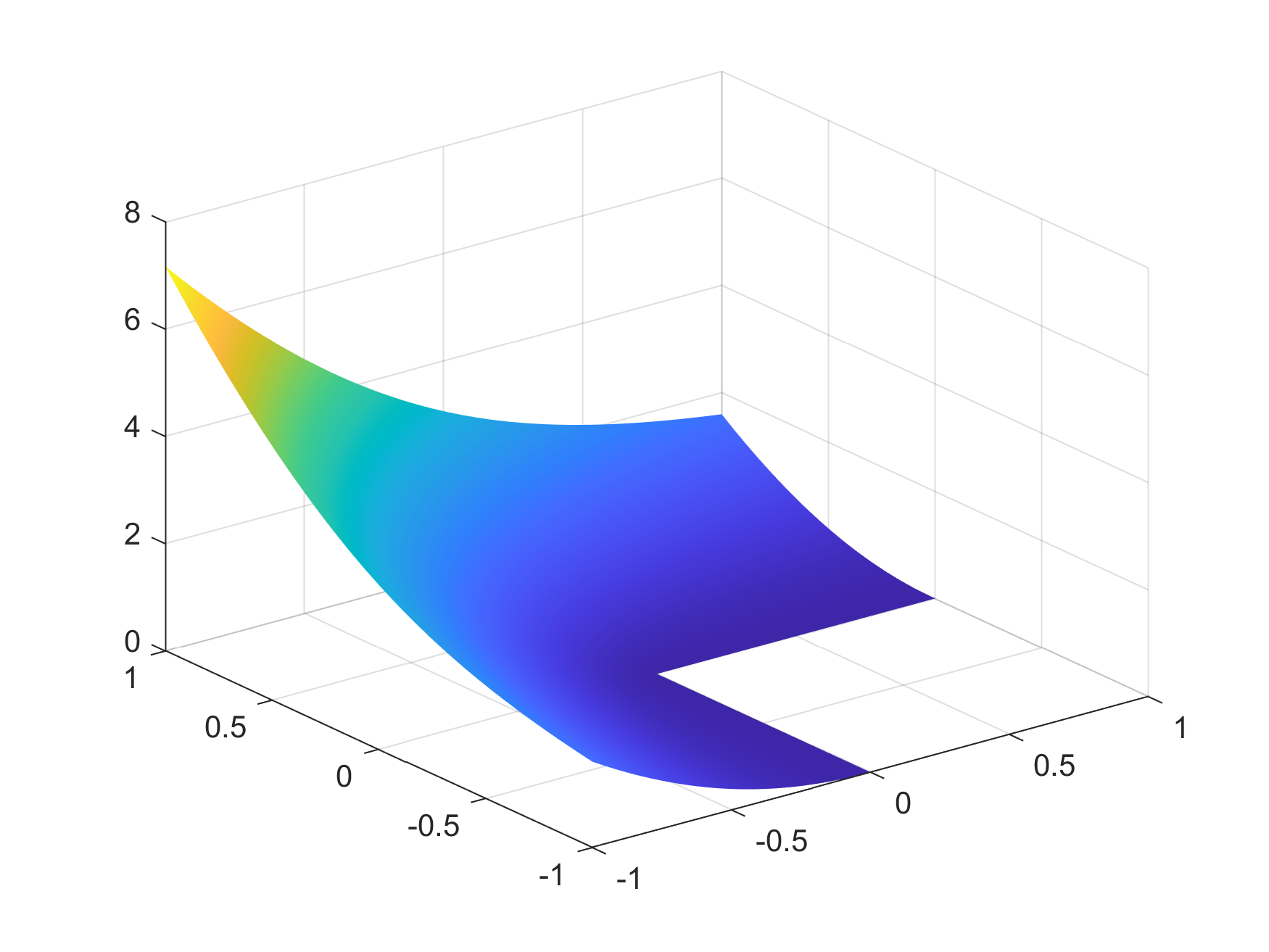} 
\label{fig:ex3-solution}
}
\end{subfigure}
 \caption{Domain and exact solution of Example~\ref{ex:Example3}. }
 \label{fig:domains2}
 \end{center}
\end{figure}

In the test the initial mesh has $4 \times 4$ elements on each patch, and we run the adaptive method until the hierarchical space reaches twelve levels. In Fig.~\ref{fig:bilaplacianex3} (right), where we plot the obtained errors in {\RV the} $H^2$ semi-norm, {\RV the advantage of using local refinement over the uniform one becomes clear}, regardless of the employed basis. In the plot the convergence rate appears to be slightly better than the optimal one, which indicates that the asymptotic regime has not yet been reached, except for degree $p=3$. In Fig.~\ref{fig:bilaplacianex3} (left) we show the hierarchical meshes obtained when solving the problem with the truncated basis and stopping the iterations at six levels, and we see a similar behavior as for the other examples, with some elements refined away from the singularity for degree $p=5$ but without affecting the convergence of the method.

\begin{figure}
\begin{center}
	\begin{subfigure}[Mesh for $p=4$, 1081 NDOF.]{\includegraphics[trim=80mm 10mm 80mm 10mm,width=0.42\textwidth,clip]{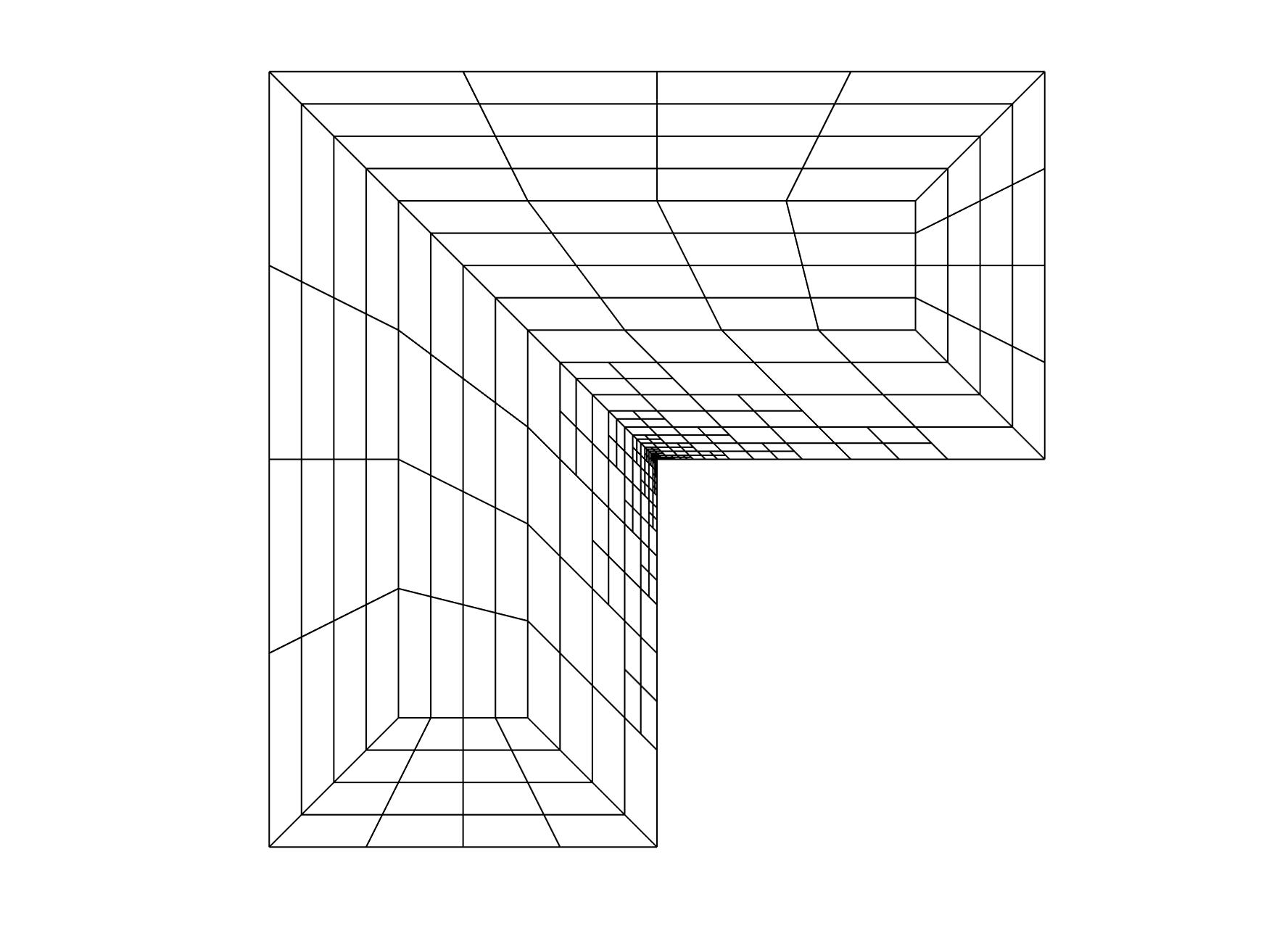}}\end{subfigure}
  \begin{subfigure}[Error for degree 3.]{\includegraphics[trim=5mm 0mm 10mm 5mm,width=0.48\textwidth,clip]{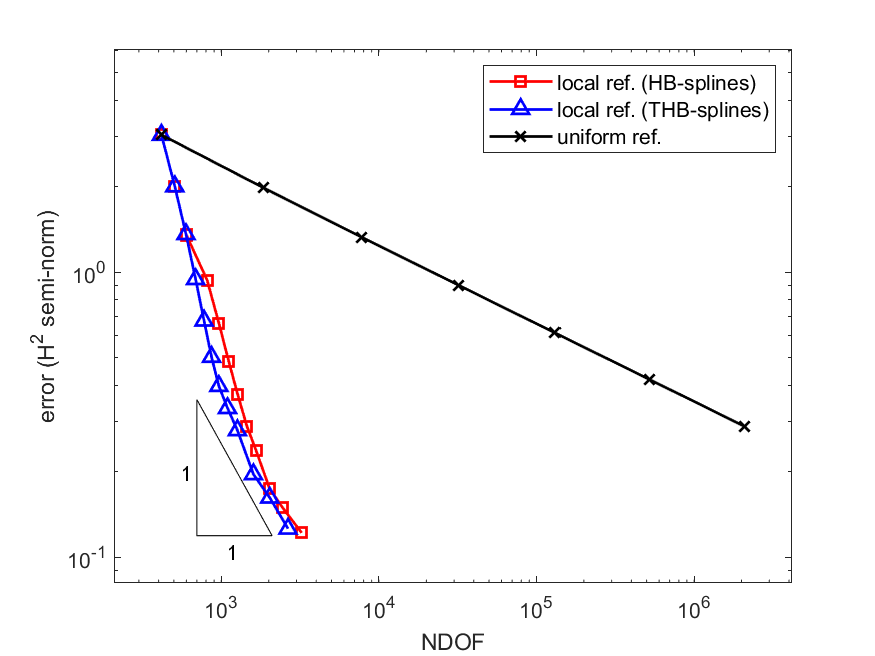}}\end{subfigure}
  \begin{subfigure}[Error for degree 4.]{\includegraphics[trim=5mm 0mm 10mm 5mm,width=0.48\textwidth,clip]{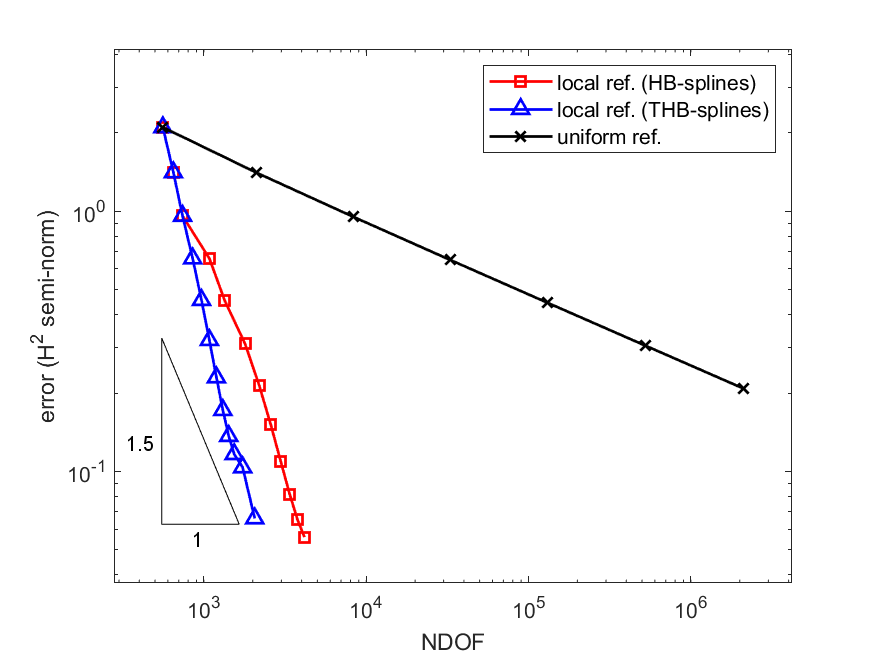}}\end{subfigure}
  \begin{subfigure}[Error for degree 5.]{\includegraphics[trim=5mm 0mm 10mm 5mm,width=0.48\textwidth,clip]{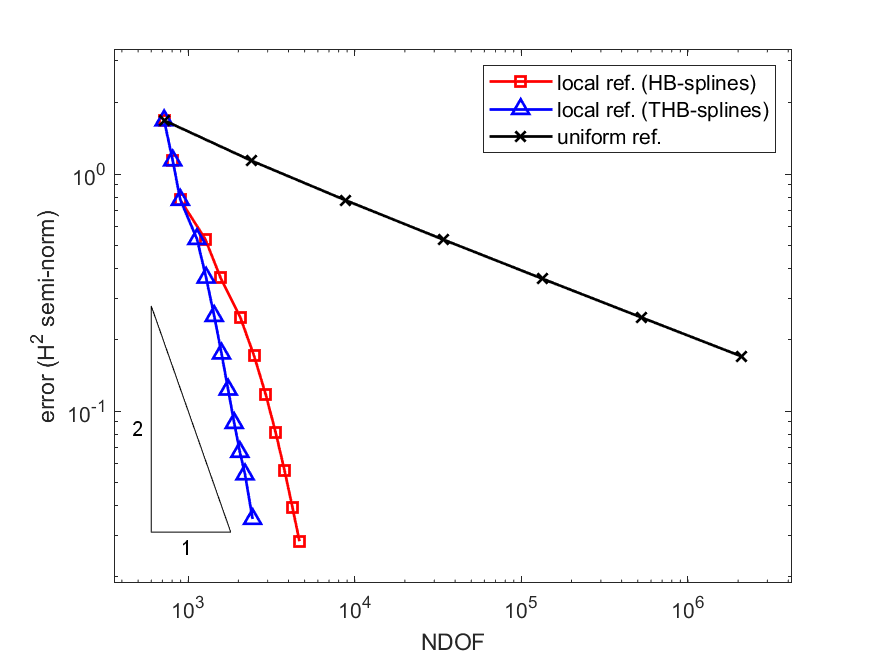}}\end{subfigure}
 \caption{Example~\ref{ex:Example3}: a ${\cal T}$-admissible mesh with six levels, and convergence for degrees 3, 4 and 5.}
 \label{fig:bilaplacianex3}
 \end{center}
\end{figure}


\end{example}

%% file: s8_conclusions.tex
We developed an adaptive isogeometric method for solving PDEs over planar analysis-suitable $G^1$ multi-patch geometries with $C^1$ hierarchical splines. Since the $C^1$ spline spaces on one level lack local linear independence, as we demonstrated on the basis of an example, we analyzed the hierarchical spline construction under relaxed assumptions and proved that linear independence of the {\RV set of hierarchical splines} can still be obtained.

The design of the adaptive method involved the investigation of several properties of the $C^1$ isogeometric spline space of each level, to guarantee that the relaxed assumptions are satisfied. This comprises its detailed characterization, the local linear independence of particular subsets of basis functions, as well as the refinement masks between two consecutive levels of refinement. In addition, we proved key properties of the resulting $C^1$ hierarchical spline space and its associated basis such as nestedness on refined meshes and, under a mild assumption on the mesh near the vertices, linear independence of the {\RV set of hierarchical splines that form the basis}. We presented a refinement algorithm with linear complexity, which guarantees the construction of graded hierarchical meshes that fulfill the condition for linear independence. Finally, the potential of the adaptive scheme was demonstrated by solving the Poisson problem as well as the biharmonic problem over different planar analysis-suitable $G^1$ multi-patch parameterizations, where the numerical results indicated {\RV that the $C^1$ basis with the presented refinement algorithm improve the convergence with respect to uniform refinement.}

In future work, we plan to extend our adaptive isogeometric spline method to the case of analysis-suitable $G^1$ multi-patch surfaces as well as to the application of further fourth order PDEs such as the Kirchhoff-Love shell problem \cite{kiendl-bletzinger-linhard-09}. From a more theoretical perspective, we plan to analyze the convergence properties of the adaptive method, for which it is first necessary to study the approximation properties of the (hierarchical) $C^1$ spline spaces.

%% file: appendices.tex
\section{Definitions for the computation of $C^1$ basis functions} \label{sec:appendix}
For the sake of completeness, we present in this appendix further definitions that are necessary to define and compute the basis functions of Section~\ref{sec:multipatch}, and therefore also for the hierarchical basis.

\subsection{Modified univariate basis functions} \label{sec:modified_basis}
The modified basis functions~$\UM{p}{r}{j}$, for $j=0,1$, $\UM{p}{r+1}{j}$, for $j=0,1,2$, and 
$\UM{p-1}{r}{j}$, for $j=0,1$, are given by
\begin{equation*} 
\begin{array}{ll}
\displaystyle \UM{p}{r}{0}(\xi) = \sum_{j=0}^{1}\UN{p}{r}{j}(\xi), & \displaystyle \UM{p}{r}{1}(\xi) = \frac{1}{p(k+1)}\UN{p}{r}{1}(\xi), \\
\displaystyle \UM{p}{r+1}{0}(\xi) = \sum_{j=0}^{2}\UN{p}{r+1}{j}(\xi), & \displaystyle  \UM{p}{r+1}{1}(\xi) = \frac{1}{p(k+1)} \sum_{j=1}^{2}A(j)\UN{p}{r+1}{j}(\xi),  \\
\displaystyle \UM{p}{r+1}{2}(\xi)= \frac{B}{p(p-1)(k+1)^2} \UN{p}{r+1}{2}(\xi),
\end{array}
\end{equation*}
with $A(j)=j$ and $B=1$ for $r <p-2$, and $A(j)=2j-1$ and $B=2$ for $r=p-2$, and
\begin{equation*} 
 \UM{p-1}{r}{0}(\xi) = \sum_{j=0}^{1}\UN{p-1}{r}{j}(\xi), \quad \UM{p-1}{r}{1}(\xi) = \frac{1}{(p-1)(k+1)}\UN{p-1}{r}{1}(\xi).
\end{equation*}

\subsection{Computation of gluing data} \label{sec:gluing_data}
For analysis-suitable $G^1$ multi-patch parameterizations, the linear functions~$\alpha^{(i,0)}$ and $\alpha^{(i,1)}$ and the quadratic function $\beta^{(i)}$ are uniquely 
determined up to a common function~$\gamma^{(i)}$ (with $\gamma^{(i)}(\xi) \neq 0$) via 
\begin{equation*}
   \begin{aligned} \label{eq:gluing_functions}
  	\alpha^{(i,0)} (\xi)  & = \gamma^{(i)} (\xi)  \det  \left [\begin{array}{ll}
          \Du \f{F}^{(i_0)}(0,\xi) & 
        \Dv \f{F}^{(i_0)}(0,\xi) 
         \end{array}\right ],\\
 	\alpha^{(i,1)} (\xi)  & =  \gamma^{(i)} (\xi)  \det  \left [\begin{array}{ll}
         \Du \f{F}^{(i_1)}(\xi,0) & 
        \Dv \f{F}^{(i_1)}(\xi,0) 
        \end{array}\right ],\\
 \beta^{(i)} (\xi) & =  \gamma^{(i)} (\xi) \det  \left [\begin{array}{ll}
         \Dv \f{F}^{(i_1)}(\xi,0)  & 
        \Du \f{F}^{(i_0)}(0,\xi) 
         \end{array}\right ],
   \end{aligned}
  \end{equation*}
and there always exist (non-unique) linear functions $\beta^{(i,0)}$ and $\beta^{(i,1)}$ such that~\eqref{eq:beta} holds, see \cite{CoSaTa16}. To uniquely 
determine the linear functions~$\alpha^{(i,0)}$, $\alpha^{(i,1)}$, $\beta^{(i,0)}$ and $\beta^{(i,1)}$ for each inner edge~$\Sigma^{(i)}$, we assume that 
they are selected by minimizing the terms
\[
 ||\alpha^{(i,0)}-1 ||^{2}_{L_{2}([0,1])} +  ||\alpha^{(i,1)} -1 ||^{2}_{L_{2}([0,1])}
\]
and
\[
  ||\beta^{(i,0)} ||^{2}_{L_{2}([0,1])} +  ||\beta^{(i,1)} ||^{2}_{L_{2}([0,1])},
\]
see \cite{KaSaTa19a}. For each boundary edge~$\Sigma^{(i)}$, $i \in \indSigmaB$, we can simply assign trivial functions~$\alpha^{(i,0)} \equiv 1$ and 
$\beta^{(i,0)} \equiv 0$. 

\subsection{Functions involved in the definition of edge and vertex basis functions} \label{sec:function-details}

The functions appearing in the definition of edge basis functions of Section~\ref{sec:edge_functions} are given by
\begin{equation} \label{eq:func_trace}
  \begin{aligned}
 \fSigma_{(j_1,0)}^{(i,0)}(\xi_1,\xi_2)  & = \UN{p}{r+1}{j_1}(\xi_2)\UM{p}{r}{0}(\xi_1) - \beta^{(i,0)}(\xi_2)(\UN{p}{r+1}{j_1})'(\xi_2)\UM{p}{r}{1}(\xi_1), \\
 \fSigma_{(j_1,0)}^{(i,1)}(\xi_1,\xi_2)  & = \UN{p}{r+1}{j_1}(\xi_1)\UM{p}{r}{0}(\xi_2) - \beta^{(i,1)}(\xi_1)(\UN{p}{r+1}{j_1})'(\xi_1)\UM{p}{r}{1}(\xi_2),
  \end{aligned} 
\end{equation}
and
\begin{equation} \label{eq:func_der}
 \begin{aligned}
  \fSigma_{(j_1,1)}^{(i,0)}(\xi_1,\xi_2) & =  \alpha^{(i,0)}(\xi_2)\UN{p-1}{r}{j_1}(\xi_2)\UN{p}{r}{1}(\xi_1), \\
  \fSigma_{(j_1,1)}^{(i,1)}(\xi_1,\xi_2) & = -\alpha^{(i,1)}(\xi_1)\UN{p-1}{r}{j_1}(\xi_1)\UN{p}{r}{1}(\xi_2).
 \end{aligned}
\end{equation}
Note that the expression is greatly simplified for boundary edges, first because only the patch $i_0$ must be considered, and second because one can use the values $\alpha^{(i,0)} \equiv 1$ and $\beta^{(i,0)} \equiv 0$.

The functions $g_{\mathbf{j}}^{(i,m,\rm{prec})}$ and $g_{\mathbf{j}}^{(i,m,\rm{next})}$, appearing in the definition of vertex basis functions of Section~\ref{sec:vertex_functions}, are respectively given by
\begin{equation} \label{eq:func_g1}
 \begin{aligned}
 g_{\mathbf{j}}^{(i,m, \rm{next})}(\xi_1,\xi_2) & = \sum_{w=0}^{2} \coefc_{\mathbf{j},w}^{(i_{m+1})} \left(\UM{p}{r+1}{w}(\xi_2)\UM{p}{r}{0}(\xi_1) - 
 \beta^{(i_{m+1},0)}(\xi_2)(\UM{p}{r+1}{w})'(\xi_2)\UM{p}{r}{1}(\xi_1) \right) \\
 & + \sum_{w=0}^{1} \coefd_{\mathbf{j},w}^{(i_{m+1})} \alpha^{(i_{m+1},0)}(\xi_2) \UM{p-1}{r}{w}(\xi_2)\UM{p}{r}{1}(\xi_1),
 \end{aligned}
 \end{equation}
 \begin{equation} \label{eq:func_g2}
 \begin{aligned}
 g_{\mathbf{j}}^{(i,m,\rm{prec})}(\xi_1,\xi_2)& = \sum_{w=0}^{2} \coefc_{\mathbf{j},w}^{(i_{m})} \left(\UM{p}{r+1}{w}(\xi_1)\UM{p}{r}{0}(\xi_2) - 
 \beta^{(i_{m},1)}(\xi_1)(\UM{p}{r+1}{w})'(\xi_1)\UM{p}{r}{1}(\xi_2) \right) \\
 & - \sum_{w=0}^{1} \coefd_{\mathbf{j},w}^{(i_{m})} \alpha^{(i_{m},1)}(\xi_1) \UM{p-1}{r}{w}(\xi_1)\UM{p}{r}{1}(\xi_2),
 \end{aligned}
 \end{equation}
with the coefficients $\coefc^{(k)}_{\mathbf{j},w}$ and $\coefd^{(k)}_{\mathbf{j},w}$, for $k = i_m,i_{m+1}$, given by
\begin{align*}
&  \coefc_{\mathbf{j},0}^{(k)} = {\RV \delta_{0 j_1}}{\RV \delta_{0 j_2}}, \quad
 \coefc_{\mathbf{j},1}^{(k)} = \mathbf{b}^\delta_{\mathbf{j}} \cdot \f t^{(k)}(0),  \quad  
 \coefc_{\mathbf{j},2}^{(k)} = (\f{t}^{(k)}(0))^T  \; A^\delta_{\mathbf{j}} \; \f{t}^{(k)}(0) + \mathbf{b}^\delta_{\mathbf{j}} \cdot (\f t^{(k)})'(0), \\
& \coefd_{\mathbf{j},0}^{(k)} = \mathbf{b}^\delta_{\mathbf{j}} \cdot \f{d}^{(k)}(0), \quad
 \coefd_{\mathbf{j},1}^{(k)} = (\f{t}^{(k)}(0))^T \; A^\delta_{\mathbf{j}} \; \f{d}^{(k)}(0) + \mathbf{b}^\delta_{\mathbf{j}} \cdot(\f{d}^{(k)})'(0), 
\end{align*}
and for each $\mathbf{j} \in \JChi$ we use the auxiliary matrix and vector
\begin{equation*}
  A^\delta_{\mathbf{j}} = 
    \left(\begin{array}{cc}
            {\RV \delta_{2 j_1} \delta_{0 j_2}} & {\RV \delta_{1 j_1} \delta_{1 j_2}} \\
            {\RV \delta_{1 j_1}\delta_{1j_2}} & {\RV \delta_{0 j_1}\delta_{2 j_2}}
           \end{array} 
    \right) \; \text{ and } \; 
\mathbf{b}^\delta_{\mathbf{j}} = ({\RV \delta_{1 j_1} \delta_{0 j_2},\delta_{0 j_1}\delta_{1 j_2}}).
\end{equation*}
Denoting ${\bf w} = (w_1, w_2)$, the remaining function $h_{\mathbf{j}}^{(i,m)}$ is given by 
\begin{equation} \label{eq:func_g3}
 h_{\mathbf{j}}^{(i,m)}(\xi_1,\xi_2) = \sum_{w_1=0}^{1} \sum_{w_2=0}^{1} {\coefe}^{(i_m)}_{\mathbf{j},\mathbf{w}} \UM{p}{r}{w_1}(\xi_1)\UM{p}{r}{w_2}(\xi_2),
\end{equation}
with the coefficients $\coefe_{\mathbf{j},\mathbf{w}}^{(i_m)}$ defined as
\begin{align*}
&  \coefe_{\mathbf{j},(0,0)}^{(i_m)} = \delta_0^{j_1}\delta_0^{j_2},\mbox{ } 
 \coefe_{\mathbf{j},(1,0)}^{(i_m)}= \mathbf{b}^\delta_{\mathbf{j}} \cdot \f t^{(i_{m})}(0),\mbox{ } 
 \coefe_{\mathbf{j},(0,1)}^{(i_m)} = \mathbf{b}^\delta_{\mathbf{j}} \cdot \f t^{(i_{m+1})}(0), \\
& \coefe_{\mathbf{j},(1,1)}^{(i_m)} = (\f{t}^{(i_{m})}(0))^T  \; A^\delta_{\mathbf{j}} \; \f{t}^{(i_{m+1})}(0) + \mathbf{b}^\delta_{\mathbf{j}} \cdot \Du \Dv \f{F}^{(i_m)}(0,0).
\end{align*}
Note that $\coefe_{\mathbf{j},(0,0)}^{(i_m)} = \coefc_{\mathbf{j},0}^{(i_m)} = \coefc_{\mathbf{j},0}^{(i_{m+1})}$, and that $\coefe_{\mathbf{j},(1,0)}^{(i_m)} = \coefc_{\mathbf{j},1}^{(i_{m})}$ and $\coefe_{\mathbf{j},(0,1)}^{(i_m)} = \coefc_{\mathbf{j},1}^{(i_{m+1})}$.

\subsection{Computation of matrices for representation in terms of B-splines} \label{sec:computations}

%




By replacing the modified basis functions with their definitions from Appendix~\ref{sec:modified_basis}
in the expressions \eqref{eq:func_g1}, and putting in evidence the expression of the edge functions in \eqref{eq:func_trace} and \eqref{eq:func_der}, we obtain the explicit representation of $g_{\mathbf{j}}^{(i,m,\rm{prec})}$ in terms of B-splines:
\begin{align*}
 &g_{\mathbf{j}}^{(i,m,\rm{prec})}
 = 
c_{\mathbf{j},0}^{(i_{m})} \sum_{j=0}^2 \fSigma_{(j,0)}^{(i_m,1)}
+ \frac{c_{\mathbf{j},1}^{(i_{m})}}{p(k+1)} \sum_{j=1}^2 A(j) \fSigma_{(j,0)}^{(i_m,1)}
\\
&+ \frac{B c_{\mathbf{j},2}^{(i_{m})}}{p(p-1)(k+1)^2} \fSigma_{(2,0)}^{(i_m,1)}
+ \frac{d_{\mathbf{j},0}^{(i_{m})}}{p(k+1)} \sum_{j=0}^{1} \fSigma_{(j,1)}^{(i_m,1)}
+ \frac{d_{\mathbf{j},1}^{(i_{m})}}{p(p-1)(k+1)^2} \fSigma_{(1,1)}^{(i_m,1)}
\\
&= c_{\mathbf{j},0}^{(i_{m})}\fSigma_{(0,0)}^{(i_m,1)}
+\left(c_{\mathbf{j},0}^{(i_{m})}+\frac{A(1)c_{\mathbf{j},1}^{(i_{m})}}{p(k+1)}\right)\fSigma_{(1,0)}^{(i_m,1)}
\\
&+\left(c_{\mathbf{j},0}^{(i_{m})}+\frac{A(2)c_{\mathbf{j},1}^{(i_{m})}}{p(k+1)}+\frac{Bc_{\mathbf{j},2}^{(i_{m})}}{p(p-1)(k+1)^2}\right)\fSigma_{(2,0)}^{(i_m,1)}
+ \frac{d_{\mathbf{j},0}^{(i_{m})}}{p(k+1)}\fSigma_{(0,1)}^{(i_m,1)}
\\
&+\left(\frac{d_{\mathbf{j},0}^{(i_{m})}}{p(k+1)}+\frac{d_{\mathbf{j},1}^{(i_{m})}}{p(p-1)(k+1)^2}\right)\fSigma_{(1,1)}^{(i_m,1)}
.
\end{align*}
In a completely analogous fashion, starting from \eqref{eq:func_g2} we obtain that
\begin{align*}
 &g_{\mathbf{j}}^{(i,m,\rm{next})}
= c_{\mathbf{j},0}^{(i_{m+1})}\fSigma_{(0,0)}^{(i_{m+1},0)}
+\left(c_{\mathbf{j},0}^{(i_{m+1})}+\frac{A(1)c_{\mathbf{j},1}^{(i_{m+1})}}{p(k+1)}\right)\fSigma_{(1,0)}^{(i_{m+1},0)}
\\
&+\left(c_{\mathbf{j},0}^{(i_{m+1})}+\frac{A(2)c_{\mathbf{j},1}^{(i_{m+1})}}{p(k+1)}+\frac{Bc_{\mathbf{j},2}^{(i_{m+1})}}{p(p-1)(k+1)^2}\right)\fSigma_{(2,0)}^{(i_{m+1},0)}
+ \frac{d_{\mathbf{j},0}^{(i_{m+1})}}{p(k+1)}\fSigma_{(0,1)}^{(i_{m+1},0)}
\\
&+\left(\frac{d_{\mathbf{j},0}^{(i_{m+1})}}{p(k+1)}+\frac{d_{\mathbf{j},1}^{(i_{m+1})}}{p(p-1)(k+1)^2}\right)\fSigma_{(1,1)}^{(i_{m+1},0)}
.
\end{align*}
From these expressions we get the matrices $K_{i,{m}}$ and $K_{i,{m+1}}$ of Section \ref{subsec:standard_representation}, that we wrote there replacing $A(j)$ and $B$ with their particular values for $r=p-2$.

Similarly, replacing the expression of the modified univariate basis functions in \eqref{eq:func_g3}, we have
\begin{align*}
&h_{\mathbf{j}}^{(i,m)}
= 
%
%
{\coefe}^{(i_m)}_{\mathbf{j},(0,0)} 
{\MK \UN{p}{r}{(0,0)} }
+ \left({\coefe}^{(i_m)}_{\mathbf{j},(0,0)} + \frac{{\coefe}^{(i_m)}_{\mathbf{j},(1,0)}}{p(k+1)}\right) 
{\MK \UN{p}{r}{(1,0)}}
+ \left({\coefe}^{(i_m)}_{\mathbf{j},(0,0)}+\frac{{\coefe}^{(i_m)}_{\mathbf{j},(0,1)}}{p(k+1)}\right) 
{\MK \UN{p}{r}{(0,1)}}\\
&+\left({\coefe}^{(i_m)}_{\mathbf{j},(0,0)}+\frac{{\coefe}^{(i_m)}_{\mathbf{j},(0,1)}}{p(k+1)}+\frac{{\coefe}^{(i_m)}_{\mathbf{j},(1,0)}}{p(k+1)}+\frac{{\coefe}^{(i_m)}_{\mathbf{j},(1,1)}}{p^2(k+1)^2}\right) 
{\MK \UN{p}{r}{(1,1)}},
\end{align*}
from which we get, using the relations between $c_{\mathbf{j}}$ and $e_{\mathbf{j}}$ in Appendix~\ref{sec:function-details}, the matrix $V_{i,{m}}$ of Section \ref{subsec:standard_representation}.

{\RV
\section{Proof of the triangular inequality for the distance} \label{app:distance}

\begin{lemma} \label{lemma:descendants}
Let $Q, Q' \in G^\ell$, with $\dist(Q, Q') = s 2^{-\ell}$. For any descendants $Q_d, Q_d' \in G^{\ell+k}$, with $Q_d \subset Q, Q_d' \subset Q'$ and $k>0$, it holds that
\[
2^{-(\ell+k)} (2^k(s-1)+1) \le \dist(Q_d, Q_d') \le 2^{-(\ell+k)} (2^k(s+1) - 1).
\]
Moreover, for any such $Q_d$ there exists $Q_d' \subset Q'$ of the same level such that $\dist(Q, Q') = \dist(Q_d, Q_d')$.
\end{lemma}
\begin{proof}
We first note that, from the definition of the distance and the regions $\Pi^s(Q)$, there exists a sequence of elements $\{Q_j\}_{j=0}^s \subset G^\ell$ such that $Q_0 = Q$, $Q_s = Q'$ and $Q_{j+1} \subset \Pi^j(Q_0)$. The minimum distance between two descendants is obtained when $Q_d$ is adjacent to $Q_1$ and $Q_d'$ is adjacent to $Q_{s-1}$. Since every element in $G^\ell$ is refined into $2^k\times 2^k$ elements of level $\ell+k$, the number of elements of level $\ell+k$ between the descendants will be $2^k(s-1)$, and therefore the minimum distance is
\[
2^{-(\ell+k)} (2^k(s-1)+1) \le \dist(Q_d, Q_d').
\]
Similarly, the maximum distance will be obtained when $Q_d$ and $Q_d'$ are in corners respectively opposite to elements of $Q_1$ and $Q_{s-1}$. In this case, we have $2 (2^k-1)$ additional elements between them (half contained in $Q$ and half in $Q'$), and the distance is bounded by
\[
\dist(Q_d, Q_d') \le 2^{-(\ell+k)} (2^k(s+1) - 1).
\]
To prove the second statement, it is sufficient to choose $Q_d'$ with respect to $Q'$ in the same relative position of $Q_d$ with respect to $Q$ (up to possible rotations).
\end{proof}

\begin{proposition}
Let $Q \in G^\ell, Q' \in G^{\ell'}$ and $Q'' \in G^{\ell''}$, with arbitrary levels $\ell, \ell', \ell''$. Then, it holds that
\[
\dist(Q, Q') \le \dist(Q, Q'') + \dist(Q'', Q').
\]
\end{proposition}
\begin{proof}
We assume, without loss of generality, that $\ell \ge \ell'$, and prove the result case by case. The idea is to always use descendants of the finest level.

1) If $\ell = \ell' = \ell''$, the result is trivial, by the definition of the distance and the regions $\Pi^s$.

2) If $\ell = \ell' > \ell''$, from point 1) the result is true for any descendant $Q_d'' \in G^\ell$, $Q_d'' \subset{Q''}$, and by definition of the distance \eqref{eq:distance_multilevel} it is true for $Q''$.

3) If $\ell = \ell' < \ell''$, from Lemma~\ref{lemma:descendants} there exist $Q_d \subset Q, Q_d' \subset Q', Q_d, Q_d' \in G^{\ell''}$, with $\dist(Q,Q') = \dist(Q_d, Q_d')$, and the result follows again from point 1) and the definition of the distance \eqref{eq:distance_multilevel}.

In the next cases, $\ell > \ell'$ and by definition there exists $Q_d' \subset Q'$, $Q_d' \in G^\ell$ such that $\dist(Q, Q_d') = \dist(Q, Q')$.

4) If $\ell  > \ell' \ge \ell''$, by definition there exists $Q_{d_{\ell'}}'' \in G^{\ell'}$ such that $\dist(Q', Q_{d_{\ell'}}'') = \dist(Q', Q'')$. Moreover, from Lemma~\ref{lemma:descendants} there exists $Q_{d_{\ell}}'' \in G^\ell$ such that $\dist(Q_d', Q_{d_{\ell}}'') = \dist(Q', Q_{d_{\ell'}}'') = \dist(Q', Q'')$. Thus, using first point 1), and then these equalities and the definition of the distance \eqref{eq:distance_multilevel}, we have
\[
\dist(Q, Q') = \dist(Q, Q_d') \le \dist(Q, Q_{d_{\ell}}'') + \dist(Q_{d_{\ell}}'', Q_d') \le \dist(Q, Q'') + \dist(Q'', Q').
\]

5) If $\ell'' \ge \ell > \ell'$, from Lemma~\ref{lemma:descendants} there exist $Q_d \in G^{\ell''}, Q'_{d_{\ell''}} \in G^{\ell''}$ such that $\dist(Q_d, Q'_{d_{\ell''}}) = \dist (Q, Q_d') = \dist(Q, Q')$. With the same arguments as for the previous point, we have
\[
\dist(Q, Q') 
\le \dist (Q_d, Q'') + \dist (Q'', Q'_{d_{\ell''}}) \le \dist(Q, Q'') + \dist (Q'', Q').
\]

6) If $\ell > \ell'' > \ell'$, let $Q'_{d_{\ell''}} \in G^{\ell''}$ be the ancestor of the element $Q_d'$ defined before point 4). From Lemma~\ref{lemma:descendants} there exists $Q_d'' \in G^\ell$ with $\dist(Q_d'', Q_d') = \dist(Q'', Q'_{d_{\ell''}})$. With the same arguments as above, we have
\begin{align*}
& \dist(Q, Q') = \dist (Q, Q_d') \le \dist(Q, Q_d'') + \dist(Q_d'', Q_d') \\
& \le \dist(Q, Q'') + \dist(Q'', Q'_{d_{\ell''}}) \le \dist(Q, Q'') + \dist(Q'', Q'),
\end{align*}
and the proof is finished.

\end{proof}
}